\documentclass{amsart}
\usepackage{amsmath,amssymb,amsthm,amsfonts,amscd,mathrsfs,mathtools}
\usepackage[utf8]{inputenc}
\usepackage[hyphens]{url}
\usepackage[pagebackref=true]{hyperref}
\usepackage{tikz-cd}
\usepackage{tikz}
\usepgflibrary{arrows}
\usetikzlibrary{shapes.arrows,trees,backgrounds,decorations.markings,shadows,calc, patterns, positioning}
\usepackage{enumitem}
\usepackage{multicol}

\setlist[itemize,description]{leftmargin=*}

\theoremstyle{plain}
    \newtheorem{thm}{Theorem}[section]
    \newtheorem{prop}[thm]  {Proposition}
    \newtheorem{lem}[thm]   {Lemma}
    \newtheorem{cor}[thm]   {Corollary}
    
    \newtheorem{athm}{Theorem}

\theoremstyle{definition}
    \newtheorem{defn}[thm]  {Definition}
    \newtheorem{nota}[thm]  {Notation}
    \newtheorem{ex}[thm]    {Example}

\newcommand{\bA}{\mathbb{A}} %minimal Sullivan model of cA(Q)
\newcommand{\cA}{\mathcal{A}} %A(Q) are the invariants of the cG(Q) action on R[Q]
\newcommand{\bB}{\mathbb{B}} %Hurwitz-like space which is the double bar construction of Hur
\newcommand{\cB}{\mathcal{B}} %cBG is Hur(delcR;G)_{0,one}.
\newcommand{\C}{\mathbb{C}} %complex plane
\newcommand{\fd}{\mathfrak{d}} %auxiliary function with values >=0
\newcommand{\fc}{\mathfrak{c}} %configuration in Hurwitz space
\newcommand{\tfc}{\tilde{\fc}} %another configuration
\newcommand{\fe}{\mathfrak{e}} %map from a PMQ to a group in a PMQ-group pair
\newcommand{\bfe}{\mathbf{e}} %function from a finite set x [0,1] to [0,1]
\newcommand{\fF}{\mathfrak{F}} %pieces of a filtration
\newcommand{\cH}{\mathcal{H}} %a homotopy
\newcommand{\fri}{\mathfrak{i}} %forgetful map fG(P') to fG(P) for P in P'
\newcommand{\fj}{\mathfrak{j}} %a map (0,infty) to (0,infty)
\newcommand{\N}{\mathbb{N}} %natural numbers
\newcommand{\fp}{\mathfrak{p}} %bundle map from fat strata of bB
\newcommand{\bQ}{\mathbb{Q}} %rational numbers
\newcommand{\R}{\mathbb{R}} %real numbers
\newcommand{\fr}{\mathfrak{r}} %action of a group on a PMQ in a PMQ-group pair
\newcommand{\bS}{\mathbb{S}} %bS_{t,t'} is tau_t(Re>=0) cap tau_t' (Re<=0)
\newcommand{\fS}{\mathfrak{S}} %symmetric group
\newcommand{\fs}{\mathfrak{s}} %a local section of a covering / an element of a set S
\newcommand{\uU}{\underline{U}} %sequence of round discs U_i
\newcommand{\fU}{\mathfrak{U}} %a normal neighborhood in tHur and Hur
\newcommand{\cX}{\mathcal{X}} %generic topological space
\newcommand{\cY}{\mathcal{Y}} %another topological space
\newcommand{\bY}{\mathbb{Y}} %a big, closed subspace of bH
\newcommand{\Z}{\mathbb{Z}} %integers

\usepackage{mathbbol}
\DeclareSymbolFontAlphabet{\mathbb}{AMSb}
\DeclareSymbolFontAlphabet{\mathbbl}{bbold}
\newcommand{\one}{\mathbbl{1}} %identity of a group, beautiful version
\newcommand{\geo}{^{\mathrm{geo}}} %geodesic PMQ from a normed group
\newcommand{\conj}{\mathrm{conj}} %conjugacy class in a quandle
\newcommand{\cG}{\mathcal{G}} %groupification of a PMQ
\newcommand{\cK}{\mathcal{K}} %kernel of adjoint action of cG(\Q) on \Q
\newcommand{\Q}{\mathcal{Q}} %generic name for a quandle or PMQ
\newcommand{\hQ}{\hat{\Q}} %the completion of a PMQ
\newcommand{\PMQ}{\mathbf{PMQ}} %category of PMQs
\newcommand{\gen}{f} %admissible generator fGP fQP fGuU fQuU or fQPuU
\newcommand{\fg}{\mathfrak{g}} %generic element of fGP
\newcommand{\op}{\mathrm{op}} %opposite group

\newcommand{\Hur}{\mathrm{Hur}} %Hurwitz space
\newcommand{\mHurm}{\mathring{\mathrm{HM}}} %topological monoid from Hurwitz spaces of open rectangles
\newcommand{\bHurm}{\breve{\mathrm{HM}}} %topological monoid from Hurwitz spaces of horizontally closed rectangles
\newcommand{\bHurmfl}{\bHurm^{\flat}_+} %subspace of previous
\newcommand{\bHurmsh}{\bHurm^{\sharp}_+} %subspace of previous
\newcommand{\zfl}{z^{\flat}} %t-1/2
\newcommand{\tP}{\tilde{P}}
\newcommand{\tp}{\tilde{p}}
\newcommand{\tpsi}{\tilde{\psi}}
\newcommand{\tphi}{\tilde{\phi}}
\newcommand{\tz}{\tilde{z}}
\newcommand{\HurTQG}{\Hur(\fT;\Q,G)}%
\newcommand{\pr}{p} %a projection to a quotient
\newcommand{\FN}{F} %norm filtration
\newcommand{\FNfat}{F^{\mathrm{fat}}} %norm filtration, fat version
\newcommand{\FNlfat}{F_{\mathrm{l}}^{\mathrm{fat}}} %norm filtration, left fat version
\newcommand{\FNrfat}{F_{\mathrm{r}}^{\mathrm{fat}}} %norm filtration, right fat version
\newcommand{\FNlr}{F_{\mathrm{lr}}} %norm filtration, left right version
\newcommand{\FNlrfat}{F_{\mathrm{lr}}^{\mathrm{fat}}} %norm filtration, left right fat version
\newcommand{\delfat}{\del^{\mathrm{fat}}} %fat boundary of cR_1 and of norm filtration strata
\newcommand{\dellfat}{\del_{\mathrm{l}}^{\mathrm{fat}}} %left fat boundary of filtration strata
\newcommand{\delrfat}{\del_{\mathrm{r}}^{\mathrm{fat}}} %right fat boundary of filtration strata
\newcommand{\dellrfat}{\del_{\mathrm{lr}}^{\mathrm{fat}}} %right fat boundary of filtration strata
\newcommand{\fFN}{\fF} %norm filtration stratum
\newcommand{\CmP}{\C\setminus P} %complement of a generic finite set P in the plane
\newcommand{\zleft}{z^{\mathrm{l}}} %chosen point of cY on the left end
\newcommand{\zright}{z^{\mathrm{r}}} %chosen point of cY on the right end
\newcommand{\genleft}{\gen^{\mathrm{l}}} %standard generator joining zleft to *
\newcommand{\genright}{\gen^{\mathrm{r}}} %standard generator joining zright to *
\newcommand{\arc}{\zeta} %arc joining a point in P with *
\newcommand{\totmon}{\omega} %total monodromy in a group
\newcommand{\fG}{\mathfrak{G}} %fG(P) is pi_1(CmP,*) for P in the upper halfplane
\newcommand{\fQ}{\mathfrak{Q}} %fQ(P), fQ(uU) and fQ(uU,P) are the fundamental PMQ, for P in the upper halfplane
\newcommand{\fQext}{\fQ^{\mathrm{ext}}} %extended fundamental PMQ
\newcommand{\fT}{\mathfrak{C}} %nice couple of subspaces of the upper halfplane
\newcommand{\expl}{\mathscr{E}} %explosion
\newcommand{\Ran}{\mathrm{Ran}} %Ran space

\newcommand{\cR}{\mathcal{R}} %a closed rectangle of variable width and height 1
\newcommand{\cRl}{\cR^{\mathrm{l}}} %[0,1/2] \times [0,1]
\newcommand{\cRr}{\cR^{\mathrm{r}}} %[1/2,1] \times [0,1]
\newcommand{\mcRl}{\mcR^{\mathrm{l}}} %(0,1/2) \times (0,1)
\newcommand{\mcRr}{\mcR^{\mathrm{r}}} %(1/2,1) \times (0,1)
\newcommand{\mcR}{\mathring{\cR}} %the interior of the previous
\newcommand{\bcR}{\breve{\cR}} %the vertically closed rectangle of height 1
\newcommand{\bcRlr}{\breve{\cR}^{\mathrm{lr}}} %the horizontally closed rectangle of height 1
\newcommand{\bdel}{\breve{\partial}} %\bdel\bcR_t are the horizontal sides of the previous
\newcommand{\bH}{\mathbb{H}} %closed upper halfplane
\newcommand{\bT}{\mathbb{T}} %a generic, contractible subspace of C
\newcommand{\zcentre}{z_{\mathrm{c}}} %centre of mcR_1
\newcommand{\zcentrel}{z_{\mathrm{c,l}}} %a point of mcR_1, left to zcentre
\newcommand{\zcentrer}{z_{\mathrm{c,r}}} %a point of mcR_1, right to zcentre
\newcommand{\zdiamleft}{z_{\diamond}^{\mathrm{l}}} %i/2
\newcommand{\zdiamright}{z_{\diamond}^{\mathrm{r}}} %1 + i/2
\newcommand{\zdiamup}{z_{\diamond}^{\mathrm{u}}} %1/2+i
\newcommand{\zdiamdown}{z_{\diamond}^{\mathrm{d}}} %1/2
\newcommand{\gencentre}{\gen_{\mathrm{c}}} %centre of mcR_1
\newcommand{\gendiamleft}{\gen_{\diamond}^{\mathrm{l}}} %i/2
\newcommand{\gendiamright}{\gen_{\diamond}^{\mathrm{r}}} %1 + i/2
\newcommand{\cHleft}{\cH^{\mathrm{l}}} %homotopy collapsing left part of mcR_1
\newcommand{\cHright}{\cH^{\mathrm{r}}} %homotopy collapsing right part of mcR_1
\newcommand{\diamo}{\diamond} %a closed rhombus
\newcommand{\bdiamolr}{\breve{\diamo}^{\mathrm{lr}}} %an almost open rhombus
\newcommand{\cHdiamo}{\cH^{\diamo}} % homotopy collapsing mcR_1 to diamo
\newcommand{\hmu}{\hat{\mu}} %concatenation map (product on Hurwitz monoid but forgetting t)
\newcommand{\fTbox}{{\fT}^{\Box}} %nice couple mR_R,mR_R-mR_1
\newcommand{\bfTbox}{\breve{\fT}^{\Box}} %nice couple bR_R,bR_R-mR_1
\newcommand{\fTdiamo}{{\fT}^{\diamond}} %nice couple  cup R,R-mdiamo
\newcommand{\bfTdiamo}{\breve{\fT}^{\diamond}} %nice couple bdiamo cup R,R-mdiamo cup deldiamo
\newcommand{\cHdiamoleft}{\cH^{\diamond,\mathrm{l}}} %homotopy collapsing left part of mcR_1
\newcommand{\cHdiamoright}{\cH^{\diamond,\mathrm{r}}} %homotopy collapsing right part of mcR_1
\newcommand{\rot}{\mathrm{rot}} %an index for a rotating map C to C
\newcommand{\fclr}{\fc^{\mathrm{lr}}} %the basepoint of Hur(---)_{bdel bdiamo}

\newcommand{\reB}{\bar{B}} %reduced bar construction
\newcommand{\uw}{\underline{w}} %barycentric coordinates in a simplex
\newcommand{\um}{\underline{m}} %sequence of elements in a monoid M
\newcommand{\ut}{\underline{t}} %sequence of nonnegative real numbers
\newcommand{\Arr}{\mathrm{Arr}} %set of p+2 x q+2 arrays
\newcommand{\NAdm}{\mathrm{NAdm}} %set of non-admissible p+2 x q+2 arrays
\newcommand{\ufc}{\underline{\fc}} %sequence of configurations in Hurwitz spaces
\newcommand{\del}{\partial} %boundary of a multisimplex or a polygon
\newcommand{\mDelta}{\mathring{\Delta}} %interior of a simplex
\newcommand{\braiding}{\mathfrak{br}} %braiding in a braided monoidal category

\newcommand{\pa}[1]{\left(#1\right)}
\newcommand{\qua}[1]{\left<#1\right>}
\newcommand{\abs}[1]{\left|#1\right|}
\newcommand{\set}[1]{\left\{#1\right\}}
\newcommand{\sca}[1]{[\! [#1]\! ]}

\renewcommand{\phi}{\varphi}
\renewcommand{\epsilon}{\varepsilon}

\DeclareMathOperator{\Id}{Id}

\DeclareMathOperator{\Aut}{Aut} %automorphisms of a group or a PMQ

\title[Deloopings of Hurwitz spaces]{Deloopings of Hurwitz spaces}
\author{Andrea Bianchi}
\thanks{
This work was partially supported by the \emph{Deutsche
  Forschungsgemeinschaft} (DFG, German Research Foundation) under Germany’s
Excellence Strategy (\texttt{EXC-2047/1}, \texttt{390685813}), by the
\emph{European Research Council} under the European Union’s Seventh Framework
Programme (\texttt{ERC StG 716424 - CASe}, PI Karim Adiprasito), and by the
\emph{Danish National Research Foundation} through the \emph{Copenhagen Centre for
Geometry and Topology} (\texttt{DNRF151}).
}
\email{anbi@math.ku.dk}
\address{Department of Mathematical Sciences, University of Copenhagen \newline
Universitetsparken 5, Copenhagen, 2100, Denmark}  
\date{\today}


\begin{document}
\begin{abstract}
For a partially multiplicative quandle (PMQ) $\mathcal{Q}$ we consider the topological monoid $\mathring{\mathrm{HM}}(\mathcal{Q})$ of Hurwitz spaces of configurations in the plane with local monodromies in $\mathcal{Q}$. We compute the group completion of $\mathring{\mathrm{HM}}(\mathcal{Q})$: it is the product of the (discrete) enveloping group $\mathcal{G}(\mathcal{Q})$ with a component of the double loop space of the relative Hurwitz space $\mathrm{Hur}_+([0,1]^2,\partial[0,1]^2;\mathcal{Q},G)_{\mathbbl{1}}$; here $G$ is any group giving rise, together with $\mathcal{Q}$, to a PMQ-group pair. Assuming further that $\mathcal{Q}$ is finite and rationally Poincar\'e and that $G$ is finite, we compute the rational cohomology ring of $\mathrm{Hur}_+([0,1]^2,\partial[0,1]^2;\mathcal{Q},G)_{\mathbbl{1}}$.
\end{abstract}

\maketitle
\section{Introduction}
In \cite[Section 2]{Bianchi:Hur1} we introduced the algebraic notion of partially multiplicative quandle (PMQ)
and the related notion of PMQ-group pair: roughly speaking, a PMQ is a set endowed with two binary operations, called \emph{conjugation} and \emph{product} (the second being partially defined), subject to axioms capturing the usual interrelations between conjugation and product in a group; and a PMQ-group pair is a pair of a PMQ $\Q$ and a group $G$, together with a map of PMQs $\Q\to G$ and an action of $G$ on $\Q$, satisfying axioms resembling the case in which $\Q$ is a conjugation-invariant subset of $G$.
In \cite[Section 3]{Bianchi:Hur2} we defined a \emph{Hurwitz-Ran space} $\Hur(\cX,\cY;\Q,G)$ associated with a
\emph{nice couple} $(\cX,\cY)$ of subspaces $\cY\subseteq\cX\subseteq\bH$ of the closed upper half-plane in $\C$, and with a PMQ-group pair $(\Q,G)$.
In the case $\cY=\emptyset$ the group $G$ plays no essential role and we can write
$\Hur(\cX;\Q)$ for the Hurwitz space:
the reader may think of this as the \emph{absolute} situation, whereas the general case corresponds to the \emph{relative} situation.

In this article, for a PMQ $\Q$, we introduce a topological monoid $\mHurm(\Q)$ arising from Hurwitz spaces: an element of $\mHurm(\Q)$ is a finite configuration $P$ of points in a rectangle $(0,t)\times(0,1)$ of variable width $t\ge0$, together with a $Q$-valued monodromy, defined on certain loops of $\CmP$.
The monoid product is defined according to a well-established principle, relying on the fact that
a rectangle of width $t+t'$ can be regarded as the union of two rectangles of widths $t$ and $t'$ joined along a vertical side.
If $\Q$ is a PMQ with trivial product,
then $\mHurm(\Q)$ recovers the monoid of Hurwitz spaces appearing in \cite[Subsection 2.6]{EVW:homstabhur} and \cite[Subsection 4.2]{ORW:Hurwitz}.

\subsection{Statement of results}
Throughout the article we fix a PMQ-group pair $(\Q,G)=(\Q,G,\fe,\fr)$ (see \cite[Definition 2.15]{Bianchi:Hur1}), and assume that $G$ is generated by the image of the map of PMQs $\fe\colon\Q\to G$.

In addition to the aforementioned topological monoid $\mHurm(\Q)$, we will introduce in this article an auxiliary topological monoid $\bHurm(\Q,G)$.
The two main theorems of the article, that we briefly describe in this subsection, show together that a component of the group completion of $\mHurm(\Q)$ is equivalent to a component of the \emph{double loop space} of a certain relative Hurwitz space $\Hur_+([0,1]^2,\del[0,1]^2;\Q,G)_{\one}$.

For present and future convenience of the reader, we remind that the index ``$+$'' selects the subspace of configurations with non-empty support in a Hurwitz space; the index ``$\one$'' selects the subspace of configurations with trivial total monodromy $\one\in G$; an index given by a finite subset of $\C$, such as ``$\bdel\bdiamolr$'', selects configurations whose support contains the given finite subset; and the index ``$G,G^{\op}$'' refers to a quotient of another Hurwitz space by a certain free action of the group $G\times G^{\op}$. This notation is introduced in detail in \cite{Bianchi:Hur2};

In favourable cases, the rational cohomology ring of $\Hur_+([0,1]^2,\del[0,1]^2;\Q,G)_{\one}$ can be computed explicitly solely in terms of the PMQ $\Q$, and one can then use standard rational homotopy theory to access $H^*(\Omega^2_0\Hur_+([0,1]^2,\del[0,1]^2;\Q,G)_{\one};\bQ)$, which by the group completion theorem is the ring of stable rational cohomology classes of components of $\mHurm(\Q)$.

The first, main result of the article is the following theorem, describing the weak homotopy type of the bar constructions $B\mHurm(\Q)$ and $B\bHurm(\Q,G)$ in terms of certain relative Hurwitz spaces.
The nice couples $(\bdiamolr,\bdel\bdiamolr)$ and $(\diamo,\del\diamo)$ are explicitly given in Definition \ref{defn:diamo}; for instance $\diamo$ is the closed rhombus with vertices $1/2$, $\sqrt{-1}/2$, $1/2+\sqrt{-1}$ and $1+\sqrt{-1}/2$. 
\begin{athm}[Theorem \ref{thm:delooping}]
 \label{thm:main1}
There are weak homotopy equivalences
 \[
 B\mHurm(\Q)\simeq \Hur(\bdiamolr,\bdel\bdiamolr;\Q,G)_{G,G^{\op}};
\quad\quad
 B\bHurm(\Q,G)\simeq \Hur(\diamo,\del\diamo;\Q,G)_{G,G^{\op}}.
 \] 
\end{athm}
Now the space $\Hur(\bdiamolr,\bdel;\Q,G)_{G,G^{\op}}$ admits the space $\Hur(\bdiamolr,\bdel;\Q,G)_{\bdel\bdiamolr,\one}$ as a finite covering space, and the latter space is weakly equivalent to $\bHurm_+(\Q,G)_{\one}$. Similarly, the space $\Hur_+([0,1]^2,\del[0,1]^2;\Q,G)_{\one}$
is weakly equivalent to a covering space of $\Hur(\diamo,\del;\Q,G)_{G,G^{\op}}$: see Subsection \ref{subsec:homologygroupcompletion} for more details. Passing to loop spaces and double loop spaces, we obtain a weak homotopy equivalence
\[
 \Omega B\mHurm(\Q)\simeq \cG(\Q)\times \Omega^2_0 \Hur_+([0,1]^2,\del[0,1]^2;\Q,G)_{\one},
\]
where $\cG(\Q)$ is the (discrete) enveloping group of $\Q$.

Under the additional assumption that $G$ is a finite group and $\Q$ is finite and \emph{rationally Poincar\'e}, the second, main result of the article computes the rational cohomology ring
of $\Hur_+([0,1]^2,\del[0,1]^2;\Q,G)_{\one}$ in terms of a certain algebra $\cA(\Q)$, that we briefly recall after the statement.
\begin{athm}[Theorem \ref{thm:HbBQG}]
 \label{thm:main2}
 Let $(\Q,G)$ be a PMQ-group pair with $\Q$ finite and rationally Poincar\'e and with $G$ finite.
 Then there is an isomorphism of rings
 \[
  H^*\pa{\Hur_+([0,1]^2,\del[0,1]^2;\Q,G)_{\one})\,;\,\bQ}\cong \cA(\Q).
 \] 
\end{athm}
We recall that a PMQ is \emph{rationally Poincar\'e} (or \emph{$\bQ$-Poincar\'e}) if it is locally finite and each component of $\Hur_+((0,1)^2;\Q)$ is a rational homology manifold \cite[Definition 9.4]{Bianchi:Hur2}.
The graded commutative $\bQ$-algebra $\cA(\Q)$ is defined as the sub-algebra of conjugation-invariant elements of $\bQ[\Q]$, the rational PMQ-algebra associated with the PMQ $\Q$ (see \cite[Definition 4.26]{Bianchi:Hur1}). When $\Q$ is Poincar\'e we can consider $\bQ[\Q]$ as a graded $\bQ$-algebra, by putting the generator $\sca{a}\in\bQ[\Q]$ in degree equal to the dimension of $\Hur_+((0,1)^2;\Q)_a$, for $a\in \Q$. The degree of $\sca{a}$ agrees, in fact, with $2h(a)$, where $h\colon\Q\to\N$ is the intrinsic norm of $\Q$ (see \cite[Proposition 9.7]{Bianchi:Hur2}).
This makes also $\cA(\Q)$ into a graded $\bQ$-algebra, and then Theorem \ref{thm:main2} gives an isomorphism of graded $\bQ$-algebras.

The rational cohomology ring of $\Omega_0 B\mHurm(\Q)$,
i.e. the stable rational cohomology ring of the components of $\mHurm(\Q)$, can then in principle be computed
by ``looping twice'' the rational cohomology of the space $\Hur_+([0,1]^2,\del[0,1]^2;\Q,G)_{\one}$, using that
the latter space is simply connected. More precisely, this requires the computation of a minimal Sullivan model for the space 
$\Hur_+([0,1]^2,\del[0,1]^2;\Q,G)_{\one}$.
We conclude the article with an explicit computation,
dealing with the case in which $\Q$ is a finite PMQ with trivial product; this recovers, in particular, \cite[Corollary 5.4]{ORW:Hurwitz}.

\subsection{Outline of the article}
In Section \ref{sec:moorehur} we introduce the topological monoids $\mHurm(\Q)$ and $\bHurm(\Q,G)$, and compute the associated discrete monoids of path components $\pi_0(\mHurm(\Q))$ and $\pi_0(\bHurm(\Q,G))$, see Theorems \ref{thm:pi0mHurm} and \ref{thm:pi0bHurm}.

In Section \ref{sec:barconstruction} we recall the simplicial space $B_\bullet M$ associated with a unital, topological monoid $M$, and we distinguish between ``bar construction'' $BM$ and ``thin bar construction'' $\reB M$, i.e. the geometric realisations of $B_\bullet M$
as a semisimplicial space and, respectively, a simplicial space. We prove in Theorem \ref{thm:bHurmloop} a homotopy equivalence $\bHurm_+(\Q,G)\simeq \Omega B\bHurm(\Q,G)$,
and check that the group completion theorem \cite{McDuffSegal, FM94} applies to the topological monoid $\mHurm(\Q)$.

The main result of Section \ref{sec:deloopings} is Theorem \ref{thm:delooping},
whose direct consequence is
that the bar constructions $B\mHurm(\Q)$ and $B\bHurm(\Q,G)$ admit covering
spaces that are homotopy equivalent to the Hurwitz spaces $\bHurm_+(\Q,G)_{\one}$
and $\Hur_+([0,1]^2,\del[0,1]^2;\Q,G)_{\one}$ respectively. The main application
is Theorem \ref{thm:mainHiso}, computing the homology of the group completion
of $\mHurm(\Q)$ as the tensor product of the group ring $\Z[\cG(\Q)]$ and the homology
of a component of the \emph{double} loop space $\Omega^2\Hur_+([0,1]^2,\del[0,1]^2;\Q,G)_{\one}$; here $\cG(\Q)$ is the enveloping group of $\Q$.

In Section \ref{sec:BQG} we replace $\Hur_+([0,1]^2,\del[0,1]^2;\Q,G)_{\one}$ by a smaller, homotopy equivalent subspace $\bB(\Q_+,G)$, assuming that $\Q$ is augmented. Assuming further that $\Q$ is a normed PMQ,
we prove that $\bB(\Q_+,G)$ admits a norm filtration, whose strata are fibre bundles over the space
$\cB G:=\Hur(\del[0,1]^2;G)_{0;\one}$; the space $\cB G$ is, in turn, shown to be an Eilenberg-MacLane space of type $K(G,1)$.

In Section \ref{sec:cohomology} we assume that $\Q$ is a finite and $\bQ$-Poincar\'e PMQ and $G$ is a finite group, and compute the rational cohomology ring of $\bB(\Q_+,G)$, using the Leray spectral
sequence associated with the filtration on $\bB(\Q_+,G)$: Theorem \ref{thm:HbBQG} identifies
$H^*(\bB(\Q_+,G);\bQ)$ with the ring $\cA(\Q)\subset\bQ[\Q]$ of conjugation $\cG(\Q)$-invariants
of the PMQ-ring $\bQ[\Q]$. As an application, we compute the stable rational cohomology
of classical Hurwitz spaces, recovering, in particular, \cite[Corollary 5.4]{ORW:Hurwitz}.

Throughout the article we make heavy use of the results of \cite[Sections 2-6]{Bianchi:Hur1}
and \cite[Sections 2-6]{Bianchi:Hur2}: we cite every time which specific fact we are needing, so that the reader does not need to be familiar with all details of \cite{Bianchi:Hur1} and \cite{Bianchi:Hur2}.

\subsection{Motivation}
This is the third article in a series about Hurwitz spaces. Our main motivation to study generalised Hurwitz spaces comes from the relation between Hurwitz spaces and moduli spaces of Riemann surfaces given by considering the family of PMQs $\fS_d\geo$, for $d\ge2$: Theorems \ref{thm:main1} and \ref{thm:main2} will be applied in subsequent work \cite{Bianchi:Hur4} to give an alternative proof of the Mumford conjecture on the stable rational cohomology of moduli spaces of Riemann surfaces, originally proved by Madsen and Weiss \cite{MadsenWeiss}.

Moreover, this article shows how generalised Hurwitz spaces can be useful also in the study of classical Hurwitz spaces as topological monoids:
the (double) delooping of the classical monoid of Hurwitz spaces is described by Theorem \ref{thm:main1} as a relative Hurwitz space.

\subsection{Acknowledgments}
This series of articles is a generalisation
and a further development of my PhD thesis \cite{BianchiPhD}. I am grateful to
my PhD supervisor
Carl-Friedrich B\"odigheimer,
Bastiaan Cnossen,
Florian Kranhold,
Jeremy Miller,
Martin Palmer,
Dan Petersen,
Oscar Randal-Williams,
Ulrike Tillmann and
Nathalie Wahl
for helpful comments and mathematical explanations related to this article.
I am also thankful to the anonymous referee for a deep and meticulous analysis of a previous version of the article, and for several suggestions on how to improve the exposition.

\tableofcontents
\section{Hurwitz spaces as topological monoids}
\label{sec:moorehur}
We start by fixing some conventions to simplify the notation.
We fix a PMQ-group pair $(\Q,G)=(\Q,G,\fe,\fr)$ throughout the article; recall that $\fe\colon\Q\to G$ is a map of PMQs and $\fr\colon G\to \Aut_{\PMQ}(\Q)^{\op}$ is a map of groups, giving a right
 action of $G$ on $\Q$, see \cite[Definition 2.15]{Bianchi:Hur1}.  We assume in the entire article that the image of $\fe$ generates $G$ as a group.
 Two examples that the reader may keep in mind are as follows:
 \begin{itemize}
  \item $G$ is a group, $\Q_+\subset G$ is a conjugation-invariant subset, and $\Q=\Q_+\sqcup\set{\one_{\Q}}$ with the adjoined element $\one_{\Q}$ being the unit of $\Q$; we put the trivial product on $\Q$, we define $\fe$ by $\one_{\Q}\mapsto\one_G$ and $\Q_+\hookrightarrow G$; we let the action of $G$ fix $\one_{\Q}$ and conjugate elements of $\Q_+$;
  \item $G=\fS_d$ is the $d$\textsuperscript{th} symmetric group for some $d\ge2$, $\Q=\fS_d\geo$ is the geodesic PMQ from \cite[Subsection 7.1]{Bianchi:Hur1}, obtained from $\fS_d$ by restricting the product, $\fe$ is the identity of the common underlying set, and $\fS_d$ acts on $\fS_d\geo$ by usual conjugation of permutations.
 \end{itemize}
 
We usually denote by $\fT=(\cX,\cY)$ a nice couple, i.e. a couple of semialgebraic subspaces
of the closed upper half-plane $\bH\subset\C$, with $\cY$ closed in $\cX$: see \cite[Definition 2.3]{Bianchi:Hur2}.
In the entire article, we abbreviate the Hurwitz space $\HurTQG$, defined in \cite[Section 3]{Bianchi:Hur2}, as $\Hur(\fT)$;
in particular, if $\cY=\emptyset$, we abbreviate $\Hur(\cX;\Q)$ as $\Hur(\cX)$.

Recall from \cite[Definition 2.9]{Bianchi:Hur2} that if $\fT=(\cX,\cY)$ is a nice couple and if $P\subset\cX$ is a finite subset, we can define a PMQ $\fQ_{\fT}(P)$ as the subset of $\fG(P):=\pi_1(\CmP,*)$ of conjugacy classes of small simple loops spinning clockwise around exactly one point of $P$ among those lying in $\cX\setminus\cY$ (together with the neutral element $\one_{\fG(P)}$; the inclusion $\fQ_{\fT}(P)\subseteq\fG(P)$ and the conjugation action of $\fG(P)$ on $\fQ_{\fT}(P)$ make $(\fQ_{\fT}(P),\fG(P))$ into a PMQ-group pair.
\begin{nota}
 \label{nota:simplifiednotation}  
Let $\bY\subset\bH$ be closed semialgebraic subspace. Then for every semialgebraic subspace $\cX\subset\bH$ we obtain a nice couple $\fT=(\cX,\cY)$ by setting $\cY=\bY\cap \cX$; for all finite subsets $P\subset\cX$ we then have that $\fQ_{\fT}(P)$ and $\fQ_{(\bH,\bY)}(P)$ are the same subset of $\fG(P)$.

We will abuse notation and abbreviate $\fQ_{\fT}(P)$ as $\fQ(P)$ also in certain situations in which there may be some ambiguity on the nice couple $\fT$ we are considering; we leverage on the fact that all nice couples $\fT$ that might reasonably be involved in the argument are obtained as above, for a fixed and evident subspace $\bY\subset\bH$, so that the fundamental PMQ $\fQ_{\fT}(P)$ is unambiguously identified as a subset of $\fG(P)$.
\end{nota}

\begin{nota}
 \label{nota:fc}
We usually denote by $P=\set{z_1,\dots,z_k}$ a finite collection of distinct points in $\bH$,
for some $k\geq 0$. If a nice couple $\fT=(\cX,\cY)$ is under consideration, we will usually assume $P\subset\cX$ and that there is $0\leq l\leq k$ such that $z_1,\dots,z_l$ are precisely the points of $P$ lying in $\cX\setminus \cY$. We let $*=-\sqrt{-1}\in \C$ be our preferred choice of basepoint.
If $\gen_1,\dots,\gen_k$ is an admissible
generating set of $\fG(P)=\pi_1(\CmP,*)$ (see \cite[Definition 2.8]{Bianchi:Hur2}), then we usually assume that $\gen_i$ is represented
by a small simple loop spinning clockwise around $z_i$.

A configuration $\fc\in\Hur(\fT;\Q,G)$ is usually presented as $(P,\psi,\phi)$, with $P$ as above and
$(\psi,\phi)\colon(\fQ(P),\fG(P))\to(\Q,G)$ a map of PMQ-group pairs. Similarly, a configuration
$\fc\in\Hur(\cX;\Q)$ is usually presented as $(P,\psi)$, with $P$ as above and $\psi\colon\fQ(P)\to\Q$ a map of PMQs.
\end{nota}

\subsection{Definition of the Hurwitz-Moore spaces}
We first introduce notation for rectangles and horizontal strips in the plane.
\begin{nota}
\label{nota:cRt}
 For $t\geq 0$ we denote by $\cR_t\subset\bH$ the standard closed rectangle $[0,t]\times[0,1]$ of width $t$ and height 1;
 we also denote $\cR_{\infty}=[0,\infty)\times[0,1]$ the half-infinite, closed strip,
 and by $\cR_{\R}=(-\infty,+\infty)\times[0,1]$ the infinite, closed strip.
 
 For $0\leq t\leq +\infty$ we denote by $\mcR_t=(0,t)\times(0,1)$ the standard open rectangle (or half-infinite strip)
 of width $t$ and height 1, and by $\mcR_{\R}=(-\infty,+\infty)\times(0,1)$
 the infinite open strip. Let also $\del\cR_t=\cR_t\setminus\mcR_t$ for all $0\leq t\leq \infty$,
 and $\del\cR_{\R}=\cR_{\R}\setminus\mcR_{\R}$, denote the boundary of $\cR_t$ and $\cR_{\R}$
 respectively. We use the abbreviations $(\cR_t,\del)$ and $(\cR_{\R},\del)$ for the nice couples
 $(\cR_t,\del\cR_t)$ and $(\cR_{\R},\del\cR_{\R})$, respectively.
 
 For  $0\leq t\leq +\infty$ we denote by $\bcR_t=(0,t)\times[0,1]$ the standard, vertically closed rectangle (or
 vertically closed half-infinite strip) of width $t$ and height 1, and by $\bdel\bcR_t=(0,t)\times\set{0,1}\subset\bcR_t$
 the horizontal boundary of $\bcR$. 
 Similarly, we denote $\bcR_{\R}=\cR_{\R}$ and $\bdel\bcR_{\R}=\partial\bcR_{\R}=(-\infty,+\infty)\times\set{0,1}$. 
 We use the abbreviations $(\bcR_t,\bdel)$ and $(\bcR_{\R},\bdel)$ for the nice couples
 $(\bcR_t,\bdel\bcR_t)$ and $(\bcR_{\R};\bdel\bcR_{\R})$, respectively.
 
 Whenever $t=1$ we drop it from the notation, so we abbreviate $\cR_1$ as $\cR$, $\mcR_1$
 as $\mcR$ and $\bcR_1$ as $\bcR$. See Figure \ref{fig:cRt}.
 \end{nota}
Note that $\bcR_0=\mcR_0=\emptyset$.
For $t\leq t'$ the identity of $\C$ restricts to an inclusion $\mcR_t\subset\mcR_{t'}$,
and induces an inclusion of Hurwitz spaces $\Hur(\mcR_t)\subseteq\Hur(\mcR_{t'})$.

\begin{figure}[ht]
 \begin{tikzpicture}[scale=4,decoration={markings,mark=at position 0.38 with {\arrow{>}}}]
  \draw[dashed,->] (-.1,0) to (1.1,0);
  \draw[dashed,->] (0,-.2) to (0,1.1);
  \draw[dotted, fill=gray, opacity=.5] (0,0) rectangle (.5,1);
\begin{scope}[shift={(1.5,0)}]
  \draw[dashed,->] (-.1,0) to (1.1,0);
  \draw[dashed,->] (0,-.2) to (0,1.1);
  \draw[dotted, fill=gray, opacity=.5] (0,0) rectangle (.75,1);
  \draw[very thick] (0,0) to (.75,0);
  \draw[very thick] (0,1) to (.75,1);
\end{scope}
 \end{tikzpicture}
 \caption{Left: the rectangle $\mcR_{1/2}$. Right: the nice couple $(\bcR_{3/4},\bdel\bcR_{3/4})$.}
 \label{fig:cRt}
\end{figure}

\begin{defn}
 \label{defn:Hurmoore}
 The \emph{open Hurwitz-Moore} space associated with the PMQ $\Q$, denoted by
 $\mHurm(\Q)$ and abbreviated as $\mHurm$ in the entire article,
 is the subspace of $[0,\infty)\times \Hur(\mcR_{\infty})$
 containing couples $(t,\fc)$ such that $\fc\in\Hur(\mcR_{\infty})$ is a configuration supported on $\mcR_t$,
 i.e. $\fc$ takes the form $(P,\psi)$ with $P\subset\mcR_t$.
 
 Similarly, the \emph{vertically closed Hurwitz-Moore space} associated with $(\Q,G)$,
 denoted $\bHurm(\Q,G)$ and abbreviated $\bHurm$ in the entire article,
 is the subspace of $[0,\infty)\times \Hur(\bcR_{\infty},\bdel)$
 containing couples $(t,\fc)$ such that $\fc$ is supported on $\bcR_t$,
 i.e. $\fc=(P,\psi,\phi)$ with $P\subset\bcR_t$.
 \end{defn}

Note that, for fixed $t\geq 0$, the slice of $\mHurm$ containing couples of the form $(t,\fc)$ is homeomorphic to $\Hur(\mcR_t)$:
thus, $\mHurm$, as a set, is the disjoint union $\bigsqcup_{t\geq 0}\Hur(\mcR_t)$.
Similarly, $\bHurm$ is in natural bijection with the set $\bigsqcup_{t\geq 0}\Hur(\bcR_t,\bdel)$.

\begin{nota}
\label{nota:emptysetoneone}
 For a nice couple $\fT$ we denote by $(\emptyset,\one,\one)\in\Hur(\fT;\Q,G)$ the unique
 configuration supported on the empty set, i.e. of the form $(P,\psi,\phi)$ with $P=\emptyset$.
 The complement of $\set{(\emptyset,\one,\one)}$ is denoted by $\Hur_+(\fT;\Q,G)$.
\end{nota}
Note that $(\emptyset,\one,\one)$ is the only point in the spaces
$\Hur(\mcR_0)$ and  $\Hur(\bcR_0,\bdel)$, since $\mcR_0=\bcR_0=\emptyset$.
In other words, we have $\Hur_+(\mcR_0)=\Hur_+(\bcR_0,\bdel)=\emptyset$.

\begin{nota}
 \label{nota:Hurmplus}
 We write $\bHurm$ as the disjoint union $[0,\infty)\times \set{(\emptyset,\one,\one)}\sqcup \bHurm_+$,
 where we set
 \[
  \bHurm_+\colon=\big([0,\infty)\times\Hur_+(\bcR_{\infty},\bdel)\big)\cap\bHurm\ \subset\ [0,\infty)\times\Hur(\bcR_{\infty},\bdel).
 \]
\end{nota}
 By the previous discussion, every couple $(t,\fc)\in\bHurm_+$ satisfies $t>0$.

\begin{lem}
 \label{lem:HurmvsHur}
 The inclusions $\Hur(\mcR)\subset\mHurm$ and $\Hur(\bcR,\bdel)\subset\bHurm$ are homotopy equivalences.
\end{lem}

\begin{proof}
 The proof is almost identical in the two cases, so we will focus on the second case, which is slightly more difficult.
 
 For $s>0$ the map $\Lambda_s\colon\C\to\C$ given by $\Lambda_s(z)=(s\Re(z),\Im(z))$
 is a morphism of nice couples $\Lambda_s\colon(\bcR_\infty,\bdel)\to(\bcR_{\infty},\bdel)$.
  Putting all values of $s\geq 0$ together we obtain a continuous map $\Lambda\colon\C\times(0,\infty)\to\C$;
 by \cite[Proposition 4.4]{Bianchi:Hur2} we obtain a continuous map
 \[
  \Lambda_*\colon \Hur(\bcR_\infty,\bdel)\times (0,\infty)\to \Hur(\bcR_\infty,\bdel).
 \]
 Define
$
 \tilde\Lambda\colon (0,\infty)\times\Hur(\bcR_{\infty},\bdel)\times (0,\infty)\to
 (0,\infty)\times\Hur(\bcR_{\infty},\bdel)
$
 by the formula $\tilde\Lambda(t,\fc;s)=\big(ts,\Lambda_*(\fc,s)\big)$.
 We are now able to define a homotopy $\cH^{\Lambda}\colon \bHurm\times [0,1]\to \bHurm$
 by setting
\[
\cH^{\Lambda}(t,\fc;s)=\left\{
\begin{array}{ll}
\pa{ts+1-s,\tilde\Lambda(\fc,(ts+1-s)/t)}& \mbox{ for all }t>0\mbox{ and }\fc\in\bHurm_+;\\[0.2cm]
\pa{ts+1-s,(\emptyset,\one,\one)} &\mbox{ for }\fc=(\emptyset,\one,\one) \mbox{ and all }t\geq 0.
\end{array}
 \right.
\]
Note that $\cH^{\Lambda}((1,\fc),s)=(1,\fc)$ for all $\fc\in\Hur(\bcR,\bdel)$, including $(\emptyset,\one,\one)$, and all $0\leq s\leq 1$;
 moreover the map $\cH(-;1)$ is the identity of $\bHurm$, whereas the map $\cH^{\Lambda}(-;0)$ has image inside $\Hur(\bcR,\bdel)$.
\end{proof}
The reason for the name \emph{Moore} in Definition \ref{defn:Hurmoore} is that, as we will see
in Subsection \ref{subsec:Hurmtopmon}, there is a natural structure of
topological monoid on both $\mHurm$ and $\bHurm$.
In contrast, $\Hur(\mcR)$ and $\Hur(\bcR)$ are only endowed with the structure of
$E_1$-algebras in a natural way.
The spaces $\mHurm$ and $\bHurm$ play the role of the strictification of the $E_1$-algebras
$\Hur(\mcR)$ and $\Hur(\bcR)$ to actual topological monoids,
just as the Moore loop space $\Omega^{\mathrm{Moore}}X$
of a pointed topological space $X$ is a strictly associative and strictly unital replacement of the $E_1$-algebra given by the usual loop space $\Omega X$.

\subsection{Topological monoid structure}
\label{subsec:Hurmtopmon}
In this subsection we define a topological monoid structure on $\mHurm$ and $\bHurm$.
For $\mHurm$ the basic idea is to \emph{juxtapose} two configurations $\fc\in\Hur(\mcR_t)$ and $\fc'\in\Hur(\mcR_{t'})$
to obtain a larger configuration supported on the rectangle $\mcR_{t+t'}$;
for $\bHurm$ the idea is similar, but using vertically closed rectangles, and juxtaposing
also their horizontal boundaries.

In the entire subsection we focus on $\mHurm$ and write in parentheses the changes needed
in the analogous discussion about $\bHurm$.
Whenever we write $\fQ(P)$ for a subset $P\subset\bH$, we use Notation \ref{nota:simplifiednotation} with
$\bY=\emptyset$ (respectively, $\bY=\bdel\bcR_{\R}$, see Notation \ref{nota:cRt}).

\begin{nota}
 \label{nota:fcplust}
 For $t,t'\geq 0$ we denote by $\mcR_{t'}+t$ the space $(t,t+t')\times(0,1)$.
 Similarly, we denote by $(\bcR_{t'},\bdel)+t$ the nice couple $((t,t+t')\times[0,1],(t,t+t')\times\set{0,1})$, compare with Notation \ref{nota:cRt}.
 
 For a finite subset $P\subset\bH$ as in Notation \ref{nota:fc} and for $t\geq 0$ we denote by $P+t$ the subset
 $\set{z_1+t,\dots,z_k+t}\subset\bH$.
\end{nota}

Note that for $P,t,t'$ as in Notation \ref{nota:fcplust}, if $P\subset\mcR_{t'}$ (respectively $P\subset \bcR_{t'})$, then $P+t\subset\mcR_{t'}+t\subset\mcR_{t+t'}$
(respectively, $P+t\subset\bcR_{t'}+t\subset\bcR_{t+t'}$).
\begin{defn}[Definition 6.7 in \cite{Bianchi:Hur2}]
 \label{defn:taut}
 For $t\in\R$ we define a homeomorphism $\tau_t\colon(\C,*)\to(\C,*)$ by:
  \[
 \def\arraystretch{1.4}
 \tau_t(z)=\left\{
 \begin{array}{ll}
  z & \mbox{if }\Im(z)\leq -1\\
  z+t& \mbox{if } \Im(z)\geq 0\\
  z+(\Im(z)+1)t & \mbox{if }-1\leq \Im(z)\leq 0.
 \end{array}
 \right.  
 \]
\end{defn}
Note that for $t,t'\geq 0$ we have
$\tau_t(\mcR_{t'})=\mcR_{t'}+t$ (respectively, $\tau_t(\bcR_{t'},\bdel)=(\bcR_{t'},\bdel)+t$).

\begin{nota}[Notation 6.8 in \cite{Bianchi:Hur2}]
 \label{nota:bS}
 For $t\in\R$ we denote by $\C_{\Re\geq t}\subset\C$ the subspace containing all $z\in\C$ with $\Re(z)\geq t$.
 Similarly we define $\C_{\Re>t}$, $\C_{\Re\leq t}$, $\C_{\Re<t}$ and $\C_{\Re=t}$, the latter being a vertical line.
 For all $-\infty\leq t\leq t'\leq +\infty$ we define a subspace $\bS_{t,t'}\subset\C$ by
 \[
  \bS_{t,t'}=\tau_t(\C_{\Re\geq 0})\cap\tau_{t'}(\C_{\Re\leq 0}),
 \]
 where we use the conventions $\tau_{-\infty}(\C_{\Re\geq0})=\tau_{+\infty}(\C_{\Re\leq0})=\C$
 and $\tau_{-\infty}(\C_{\Re\leq0})=\tau_{+\infty}(\C_{\Re\geq0})=\emptyset$.
\end{nota}
For all $t,t'\geq0$, we have that $\bS_{0,t+t'}$ is contractible and can be written
as the union of the contractible spaces $\bS_{0,t}$ and $\bS_{t,t+t'}$ along the contractible space $\bS_{t,t}$.
Moreover $\mcR_t\subset\mathring\bS_{0,t}$ (respectively, $\bcR_t\subset\mathring\bS_{0,t}$),
whereas $\mcR_{t'}+t\subset \mathring\bS_{t,t+t'}$ (respectively, $\bcR_{t'}+t\subset\mathring\bS_{t,t+t'}$).
Note also that $\tau_t$ restricts to a homeomorphism $\bS_{0,t'}\to\bS_{t,t+t'}$.

Recall \cite[Definitions 3.15 and 3.16]{Bianchi:Hur2}: if $\bT\subseteq\C$ is a contractible subspace containing $*=-\sqrt{-1}$, then for any nice couple of subspaces $\cY\subseteq\cX\subseteq\mathring{\bT}$ we can give an alternative definition of $\Hur(\cX,\cY)$, denoted by $\Hur^{\bT}(\cX,\cY)$, using $\bT$ instead of the entire $\C$ as ``ambient space'': indeed, for any finite set $P\subseteq\cX$, the fundamental group $\pi_1(\bT\setminus P,*)$ is canonically identified with $\fG(P)=\pi_1(\CmP,*)$, and similarly for fundamental PMQs. We thus get an identification
$\fri^{\C}_{\bT}\colon\Hur(\cX,\cY)\overset{\cong}{\to}\Hur^{\bT}(\cX,\cY)$.

If, moreover, $\xi\colon(\C,*)\to(\C,*)$ is a semialgebraic and orientation-preserving homeomorphism of the plane preserving the upper half-plane $\bH$, and if $\bT',\cX',\cY'$ are three other subspaces of $\C$ as above such that $\xi$ maps $\bT\to\bT'$, $\cX\to \cX'$ and $\cY\to \cY'$, then $\xi$ induces a map $\xi_*\colon \Hur^{\bT}(\cX,\cY)\to\Hur^{\bT'}(\cX,\cY)$.

There is finally a ``disjoint union'' map $-\sqcup-\colon \Hur^{\bT_1}(\cX_1,\cY_1)\times\Hur^{\bT_2}(\cX_2,\cY_2)\to\Hur^{\bT_1\cup\bT_2}(\cX_1\sqcup\cX_2,\cY_1\sqcup\cY_2)$, defined when $\bT_1\cap\bT_2$ is contractible and disjoint from both $\cX_1,\cX_2$ (in particular, this implies that $\cX_1$ and $\cX_2$ are disjoint).
\begin{defn}
 \label{defn:muttp}
 For $t,t'\geq 0$ we define the maps
 $\mu_{t,t'}\colon \Hur(\mcR_t)\times \Hur(\mcR_{t'})\to \Hur(\mcR_{t+t'})$
 and $\mu_{t,t'}\colon \Hur(\bcR_t,\bdel)\times \Hur(\bcR_{t'},\bdel)\to \Hur(\bcR_{t+t'},\bdel)$
 as the following compositions:
 \[
  \begin{tikzcd}[column sep=40pt]
  \Hur(\mcR_t)\times \Hur(\mcR_{t'})
  \ar[r,"{(\fri^{\C}_{\bS_{0,t}},\fri^{\C}_{\bS_{0,t'}})}"]
  & \Hur^{\bS_{0,t}}(\mcR_t)\times \Hur^{\bS_{0,t'}}(\mcR_{t'}) \ar[dl,"\Id\times(\tau_t)_*"']\\
  \Hur^{\bS_{0,t}}(\mcR_t)\times \Hur^{\bS_{t,t+t'}}(\mcR_{t'}+t) \ar[r,"-\sqcup-"] &
  \Hur^{\bS_{0,t+t'}}(\mcR_t\cup(\mcR_{t'}+t)) \ar[dl,"\subset"'] \\
  \Hur^{\bS_{0,t+t'}}(\mcR_{t+t'}) \ar[r,"{(\fri^{\C}_{\bS_{0,t+t'}})^{-1}}"'] &
  \Hur(\mcR_{t+t'});
  \end{tikzcd}
 \]
 \[
  \begin{tikzcd}[column sep=40pt]
  \Hur(\bcR_t,\bdel)\times \Hur(\bcR_{t'},\bdel)
  \ar[r,"{(\fri^{\C}_{\bS_{0,t}},\fri^{\C}_{\bS_{0,t'}})}"]
  & \Hur^{\bS_{0,t}}(\bcR_t,\bdel)\times \Hur^{\bS_{0,t'}}(\bcR_{t'},\bdel) \ar[dl,"\Id\times(\tau_t)_*"']\\
  \Hur^{\bS_{0,t}}(\bcR_t,\bdel)\times \Hur^{\bS_{t,t+t'}}(\bcR_{t'}+t,\bdel) \ar[r,"-\sqcup-"] &
  \Hur^{\bS_{0,t+t'}}(\bcR_t\cup(\bcR_{t'}+t),\bdel) \ar[dl,"\subset"'] \\
  \Hur^{\bS_{0,t+t'}}(\bcR_{t+t'},\bdel) \ar[r,"{(\fri^{\C}_{\bS_{0,t+t'}})^{-1}}",'] &
  \Hur(\bcR_{t+t'},\bdel).
  \end{tikzcd}
 \]

\end{defn}

\begin{defn}
 \label{defn:Hurmtopmon}
 Recall Definition \ref{defn:muttp}. We define a map of sets
 \[
 \mu\colon \mHurm\times\mHurm\to\mHurm\quad\quad\pa{\mbox{respectively, } \mu\colon \bHurm\times\bHurm\to\bHurm }
 \]
 by the formula
$\mu(\,(t,\fc)\,,\,(t',\fc')\,)=\big(t+t'\,,\,\mu_{t,t'}(\fc,\fc')\big)$.
See Figure \ref{fig:Hurmtopmon}.
\end{defn}

\begin{prop}
 \label{prop:Hurmtopmon}
 The map $\mu\colon \mHurm\times\mHurm\to\mHurm$ (respectively, $\mu\colon \bHurm\times\bHurm\to\bHurm$)
 is continuous and makes $\mHurm$ (respectively, $\bHurm$) into a
 topological monoid, with unit $(0,(\emptyset,\one,\one))$.
\end{prop}
The proof of Proposition \ref{prop:Hurmtopmon} is in Appendix \ref{subsec:prop:Hurmtopmon}.

\begin{figure}[ht]
 \begin{tikzpicture}[scale=4,decoration={markings,mark=at position 0.38 with {\arrow{>}}}]
  \draw[dashed,->] (-.05,0) to (.6,0);
  \draw[dashed,->] (0,-1.1) to (0,1.1);
  \node at (0,-1) {$*$};
  \draw[pattern=horizontal lines] (0,-1) to (0,1.1) to node[above]{\tiny$\bS_{0,1/2}$} (.5,1.1)
  to (.5,0) to (0,-1);
  \draw[dotted, fill=gray, opacity=.5] (0,0) rectangle (.5,1);
  \draw[very thick] (0,0) to (.5,0);
  \draw[very thick] (0,1) to (.5,1);
  \node at (.1,.3){$\bullet$}; 
  \node at (.3,1){$\bullet$}; 
  \node at (.45,.5){$\bullet$};
  \draw[thin, looseness=1.2, postaction={decorate}] (0,-1) to[out=82,in=-90]  (.05,.1) to[out=90,in=-90] (.01,.3)  to[out=90,in=90] node[above]{\tiny$a_1$} (.2,.3)  to[out=-90,in=90] (.1,.1) to[out=-90,in=80] (0,-1);
  \draw[thin, looseness=1.2, postaction={decorate}] (0,-1) to[out=79,in=-90] (.28,.9)
  to[out=90,in=-90] (.26,1) to[out=90,in=90] (.34,1) to[out=-90,in=90] node[left]{\tiny$g_2$}
  (.3,.9) to[out=-90,in=77]  (0,-1);
  \draw[thin, looseness=1.2, postaction={decorate}] (0,-1) to[out=76,in=-90] (.4,.3)
  to[out=90,in=-90] (.35,.5) to[out=90,in=90] node[above]{\tiny$a_3$} (.49,.5) to[out=-90,in=90] (.45,.3) to[out=-90,in=74]  (0,-1);
\begin{scope}[shift={(.8,0)}]
  \draw[dashed,->] (-.05,0) to (.6,0);
  \draw[dashed,->] (0,-1.1) to (0,1.1);
  \node at (0,-1) {$*$};
  \draw[pattern=horizontal lines] (0,-1) to (0,1.1) to node[above]{\tiny$\bS_{0,1/2}$} (.5,1.1)
  to (.5,0) to (0,-1);
  \draw[dotted, fill=gray, opacity=.5] (0,0) rectangle (.5,1);
  \draw[very thick] (0,0) to (.5,0);
  \draw[very thick] (0,1) to (.5,1);
  \node at (.15,.0){$\bullet$};
  \node at (.45,.8){$\bullet$};
  \draw[thin, looseness=1.2, postaction={decorate}] (0,-1) to[out=82,in=-90]  (.1,-.1) to[out=90,in=-90] (.05,.05)  to[out=90,in=90] node[above]{\tiny$g'_1$} (.2,.05)  to[out=-90,in=90] (.15,-.1) to[out=-90,in=80] (0,-1);
  \draw[thin, looseness=1.2, postaction={decorate}] (0,-1) to[out=76,in=-90] (.4,.7)
  to[out=90,in=-90] (.35,.8) to[out=90,in=90] node[above]{\tiny$a'_2$} (.49,.8) to[out=-90,in=90] (.45,.7) to[out=-90,in=74]  (0,-1);
\end{scope}
\begin{scope}[shift={(2,0)}]
  \draw[dashed,->] (-.05,0) to (1.1,0);
  \draw[dashed,->] (0,-1.1) to (0,1.1);
  \node at (0,-1) {$*$};
  \draw[pattern=horizontal lines] (0,-1) to (0,1.1) to node[above]{\tiny$\bS_{0,1/2}$} (.5,1.1)
  to (.5,0) to (0,-1);
  \draw[dotted, fill=gray, opacity=.5] (0,0) rectangle (.5,1);
  \draw[very thick] (0,0) to (.5,0);
  \draw[very thick] (0,1) to (.5,1);
  \node at (.1,.3){$\bullet$}; 
  \node at (.3,1){$\bullet$}; 
  \node at (.45,.5){$\bullet$};
  \draw[thin, looseness=1.2, postaction={decorate}] (0,-1) to[out=82,in=-90]  (.05,.1) to[out=90,in=-90] (.01,.3)  to[out=90,in=90] node[above]{\tiny$a_1$} (.2,.3)  to[out=-90,in=90] (.1,.1) to[out=-90,in=80] (0,-1);
  \draw[thin, looseness=1.2, postaction={decorate}] (0,-1) to[out=79,in=-90] (.28,.9)
  to[out=90,in=-90] (.26,1) to[out=90,in=90] (.34,1) to[out=-90,in=90] node[left]{\tiny$g_2$}
  (.3,.9) to[out=-90,in=77]  (0,-1);
  \draw[thin, looseness=1.2, postaction={decorate}] (0,-1) to[out=76,in=-90] (.4,.3)
  to[out=90,in=-90] (.35,.5) to[out=90,in=90] node[above]{\tiny$a_3$} (.49,.5) to[out=-90,in=90] (.45,.3) to[out=-90,in=74]  (0,-1);
  
  \draw[pattern=horizontal lines] (0,-1) to (.5,0) to (.5,1.1) to node[above]{\tiny$\tau_{1/2}(\bS_{0,1/2})$} (1,1.1)
  to (1,0) to (0,-1);
  \draw[dotted, fill=gray, opacity=.5] (.5,0) rectangle (1,1);
  \draw[very thick] (.5,0) to (1,0);
  \draw[very thick] (.5,1) to (1,1);
  \node at (.65,.0){$\bullet$};
  \node at (.95,.8){$\bullet$};
  \draw[thin, looseness=1.2, postaction={decorate}] (0,-1) to[out=62,in=-120]  (.55,-.1) to[out=60,in=-90] (.52,.05)  to[out=90,in=90] node[above]{\tiny$g'_1$} (.7,.05)  to[out=-90,in=50] (.65,-.1) to[out=-130,in=60] (0,-1);
  \draw[thin, looseness=1.2, postaction={decorate}] (0,-1) to[out=56,in=-90] (.9,.7)
  to[out=90,in=-90] (.85,.8) to[out=90,in=90] node[above]{\tiny$a'_2$} (.99,.8) to[out=-90,in=90] (.95,.7) to[out=-90,in=54]  (0,-1);
\end{scope}
 \end{tikzpicture}
 \caption{Left: two configurations in $\Hur^{\bS_{0,1/2}}(\bcR_{1/2},\bdel)\cong\Hur(\bcR_{1/2},\bdel)\subset\bHurm$. Right: their product in
 $\Hur^{\bS_{0,1}}(\bcR,\bdel)\cong\Hur(\bcR,\bdel)\subset\bHurm$.
 }
\label{fig:Hurmtopmon}
\end{figure}

Recall the notion of total monodromy from \cite[Definitions 6.1 and 6.3]{Bianchi:Hur2}: for a generic nice couple $(\cX,\cY)$ we have a map $\totmon\colon\Hur(\cX,\cY)\to G$ sending a configuration $(P,\psi,\phi)$ to the value of the monodromy $\psi$ at the ``large loop'', i.e. the element of $\fG(P)$ represented by a simple loop spinning clockwise around all points of $P$.

If $\cY=\emptyset$, one can lift this to a total monodromy $\hat\totmon\colon \Hur(\cX)\to \hQ$, where $\hQ$ is the completion of the PMQ $\Q$, as in \cite[Definition 2.19]{Bianchi:Hur1}. Concretely, $\hQ$ can be defined as the free, \emph{non-unital} monoid generated by elements $\hat a$ for $a\in \Q$, satisfying $\hat a\hat b=\hat b\widehat{a^b}$ for all $a,b\in\Q$, and satisfying $\hat a\hat b=\widehat{ab}$ for all $a,b\in\Q$ such that the product $ab$ is already defined in $\Q$. The non-unital monoid $\hQ$ happens to have a unit, namely $\hat\one$, and a natural binary operation of conjugation can be defined on it, so that it becomes a PMQ with complete product; there is a natural inclusion of PMQs $\Q\hookrightarrow\hQ$, which is the universal map from $\Q$ to a complete PMQ.
In the lift $\hat\totmon$ of $\totmon$ we need $\hQ$ rather than $\Q$ as target because the large loop in $\cG(P)$ is not, in general, an element of the fundamental PMQ $\fQ(P)$ (unless $P$ is a singleton), so we cannot directly evaluate $\psi$ on it; but we can factor the large loop as a product of elements in $\fQ(P)$, evaluate $\psi$ on the factors, and compute in $\hQ$ the corresponding product of elements of $\Q$.
\begin{nota}
 \label{nota:muHurmcdotomega}
 For $(t,\fc)$ and $(t',\fc')$ in $\mHurm$ (in $\bHurm$) we denote by
 $(t,\fc)\cdot(t',\fc')$ the configuration $\mu((t,\fc),(t',\fc'))$.
  
 We denote by $\hat\totmon\colon\mHurm\to \hQ$ (respectively, $\totmon\colon\bHurm\to G$)
 the composition
 \[
 \begin{tikzcd}
 \mHurm\subset[0,\infty)\times\Hur(\mcR_{\infty})\ar[r] & \Hur(\mcR_{\infty})\ar[r,"\hat\totmon"] & \hQ
 \end{tikzcd}
 \]
 \[
 \begin{tikzcd}
 \Big(\mbox{respectively, }\bHurm\subset[0,\infty)\times\Hur(\bcR_{\infty},\bdel)\ar[r]& \Hur(\bcR_{\infty},\bdel)\ar[r,"\totmon"] & G\ \Big),
 \end{tikzcd}
 \]
 where the first map is the projection on the second component, and $\hQ$ denotes the completion
 of the PMQ $\Q$.
\end{nota}

\subsection{Computation of \texorpdfstring{$\pi_0(\mHurm)$}{pi0(mHurM)}}
\label{subsec:groupcompletionpinot}
In this subsection we study the discrete monoid of path components of $\mHurm$.
We will prove the following theorem, which is similar to \cite[Proposition 6.4]{Bianchi:Hur2}.
\begin{thm}
 \label{thm:pi0mHurm}
 Recall Notations \ref{nota:Hurmplus} and \ref{nota:muHurmcdotomega}.
 The map $\hat\totmon\colon\pi_0(\mHurm)\to\hQ$ is a map of unital monoids, and it restricts
 to a bijection $\pi_0(\mHurm_+)\cong\hQ$.
\end{thm}

\begin{nota}
 \label{nota:zcentre}
 We denote by $\zcentre=\frac 12+\frac{\sqrt{-1}}2\in\C$ the centre of $\mcR$.
\end{nota}

\begin{defn}
 \label{defn:fca}
 For all $a\in\Q$ we define a configuration $\fc_a=(\set{\zcentre},\psi_a)\in\Hur(\mcR)$,
 where $\psi_a$ sends the (unique) element $\gencentre$ in $\fQ(\set{\zcentre})\setminus\set{\one}$ to $a$.
\end{defn}
For a space $X$ we denote by $\pi_0\colon X\to\pi_0(X)$ the map assigning to each point of $X$ its path component.
We denote by $\cdot$ the product of the discrete monoid $\pi_0(\mHurm)$.
\begin{lem}
 \label{lem:fcagenerate}
 The monoid $\pi_0(\mHurm)$ is generated by $\pi_0(0,(\emptyset,\one,\one))$, which is the unit, and by the elements of the form $\pi_0(1,\fc_a)$, for $a\in\Q$. Moreover
 the following equalities hold in $\pi_0(\mHurm)$:
 \begin{itemize}
  \item if $a,b\in\Q$, then $\pi_0(1,\fc_a)\cdot\pi_0(1,\fc_b)=\pi_0(1,\fc_{b})\cdot\pi_0(1,\fc_{a^b})$;
  \item if $a,b\in\Q$ and the product $ab$ is defined in $\Q$, then
  $\pi_0(1,\fc_a)\cdot\pi_0(1,\fc_b)=\pi_0(1,\fc_{ab})$.
 \end{itemize}
\end{lem}
The proof of Lemma \ref{lem:fcagenerate} is in Appendix \ref{subsec:fcagenerate}.
\begin{proof}[Proof of Theorem \ref{thm:pi0mHurm}]
 First we prove that $\hat\totmon\colon\mHurm\to\hQ$ is a map of monoids.
 Let $(t,\fc),(t',\fc')\in\mHurm$, and use Notation \ref{nota:fc}:
we can choose simple loops $\gamma\subset\bS_{-\infty,t}$ and $\gamma'\subset\bS_{t,+\infty}$,
 spinning clockwise around $P$ and $P'+t$, respectively; the product
 $[\gamma]\cdot[\gamma']\in\fG(P\cup (P'+t))$ is represented by a simple loop spinning clockwise around $P\cup (P'+t)$.
 Denoting $(t,\fc)\cdot(t',\fc')=(t+t',(P'',\psi''))$,
 by definition of $\psi''$ we have
 \[
 \hat\totmon\Big((t,\fc)\cdot(t',\fc')\Big)=
 \psi''([\gamma]\cdot[\gamma'])=\psi([\gamma])\cdot\psi'([\gamma'])
 =\hat\totmon(t,\fc)\cdot\hat\totmon(t'\fc')\in\hQ.
 \]
Note that $\hat\totmon((0,(\emptyset,\one,\one)))=\one$, so $\hat\totmon$ is a map of unital monoids;
moreover $\hat\totmon(1,\fc_a)=\hat a\in\hQ$ for all $a\in\Q$,
so that $\hat\totmon\colon\pi_0(\mHurm)\to\hQ$ hits
the generators of $\hQ$ and is thus surjective.

By Lemma \ref{lem:fcagenerate}, the corresponding relations
among the elements $\pi_0(1,\fc_a)\in \pi_0(\mHurm)$ hold, so that
the assignment $\hat a\mapsto \pi_0(1,\fc_a)$ defines
a map of non-unital monoids $\Omega\colon\hQ\to\pi_0(\mHurm)$;
note that, though both the source and the target of $\Omega$ are indeed unital monoids,
$\pi_0(1,\fc_\one)=\Omega(\one)$ is not the unit of $\pi_0(\mHurm)$, so that $\Omega$
is not a map of unital monoids.

In fact $\Omega$, as a map of sets, is a right inverse
of $\hat\totmon$, i.e. $\hat\totmon\circ\Omega$ is the identity of $\hQ$.
Moreover $\Omega$ hits all elements of $\pi_0(\mHurm)$ of the form $\pi_0(1,\fc_a)$, and by Lemma \ref{lem:fcagenerate}
every element of $\pi_0(\mHurm_+)$
can be written as a product of one or more elements of the form $\pi_0(1,\fc_a)$.
It follows that $\Omega$ is a bijection between $\hQ$ and $\pi_0(\mHurm_+)$, and this concludes the proof.
\end{proof}

\subsection{Computation of \texorpdfstring{$\pi_0(\bHurm)$}{pi0(bHurM)}}
\label{subsec:pinotbHurm}
We conclude the section by computing $\pi_0(\bHurm)$.
Recalling Notations \ref{nota:emptysetoneone} and \ref{nota:Hurmplus}, it suffices
to compute $\pi_0(\bHurm_+)$. The canonical structure we have on this last set is that of \emph{non-unital} monoid,
since the multiplication $\mu$ of $\bHurm$ restricts to a map $\bHurm_+\times\bHurm_+\to\bHurm_+$.
The total monodromy gives again a morphism of non-unital monoids
\[
 \totmon\colon\pi_0(\bHurm_+)\to G.
\]

\begin{thm}
 \label{thm:pi0bHurm}
Recall the map of PMQs $\fe\colon\Q\to G$, which is part of the PMQ-group pair structure on $(\Q,G)$.
Suppose that the image $\fe(\Q)\subset G$ generates $G$ as a group.
Then the map $\totmon\colon\pi_0(\bHurm_+)\to G$ is bijective.
\end{thm}
In other words, the unital monoid $\pi_0(\bHurm)$ is isomorphic to $G\sqcup\set{\one}$, where the extra element $\one$ plays the role of the monoid unit, and the old unit $\one_G\in G$ still satisfies $\one_G\cdot g=g\cdot \one_G=g$ for all $g\in G$, but $\one\cdot\one_G=\one_G\cdot \one=\one_G$.

We observe that the hypothesis that $G$ is generated by $\fe(\Q)$ is necessary in Theorem \ref{thm:pi0bHurm}: if, for instance, $\Q=\set{\one}$ and $G$ is any non-trivial group, then $\pi_0(\bHurm_+)$ can rather be identified (as a set) with $G\times G$, and $\totmon$ with the product map $G\times G\to G$.

The rough idea of the proof of Theorem \ref{thm:pi0bHurm} is the following: given a configuration $(t,\fc)$, we can shrink or stretch it until we have $t=1$; we can move points of $\fc$ to either horizontal side of $\bcR$, reducing to a configuration $\fc$ supported on $\bdel$; we can let all points on either component of $\bdel$ collide with each other, reducing to a configuration $\fc$ supported on at most two points lying on $\bdel$; finally, we can use that $\fe(Q)$ generates $G$ to ``trade'' factors of the total monodromy from one component of $\bdel$ to the other, reaching a configuration $\fc$ supported on a single point.

The rest of the subsection is devoted to the proof of Theorem \ref{thm:pi0bHurm}.
We replace $\bHurm$ by the homotopy equivalent space $\Hur(\bcR,\bdel)$, see Lemma \ref{lem:HurmvsHur}.
\begin{nota}
 \label{nota:bcRlr}
 We denote by $\bcRlr$ the horizontally closed square $[0,1]\times(0,1)\subset\bH$, and by $\bdel\bcRlr=\set{0,1}\times(0,1)$ the union of the vertical sides of $\bcRlr$.
 We abbreviate the nice couple $(\bcRlr,\bdel\bcRlr)$ as $(\bcRlr,\bdel)$. See Figure \ref{fig:diamolr}.
 
 We fix once and for all a semialgebraic homeomorphism $\xi^{\rot}\colon\C\to\C$ which fixes the basepoint $*=-\sqrt{-1}$
 and restricts to the homeomorphism $\bcRlr\to\bcR$ given by the $90^{\circ}$ clockwise rotation around $\zcentre$ (see Notation \ref{nota:zcentre}).
\end{nota}
By functoriality we have a homeomorphism $\xi^{\rot}_*\colon\Hur(\bcRlr,\bdel)\to \Hur(\bcR,\bdel)$.
We will prove Theorem \ref{thm:pi0bHurm} by classifying connected components of $\Hur_+(\bcRlr,\bdel)$;
from now on we will focus on the latter space.

\begin{lem}
 \label{lem:pathtobdel}
 Let $\fc\in\Hur(\bcRlr,\bdel)$; then $\fc$ is connected to a configuration $\fc'$
 supported in $\bdel\bcRlr$.
\end{lem}
To prove Lemma \ref{lem:pathtobdel} we will use the following family of homotopies of $\C$.
\begin{defn}
 \label{defn:cHtleftright}
 For all $0<t<1$ we define homotopies $\cHleft_t,\cHright_t\colon \colon\C\times[0,1]\to\C$ by the following formulas:
 \[ \def\arraystretch{1.4}
 \cHleft_t(z,s)=\left\{
 \begin{array}{ll}
  z & \mbox{if } \Re(z)\leq0 \mbox{ or }\Re(z)\geq 1\\
  z-s\Re(z)&\mbox{if } 0\leq \Re(z)\leq t\\
  z - (\frac1{1-st} -1)(1-\Re(z)) &\mbox{if } t\leq \Re(z)\leq 1;
 \end{array}
 \right.
 \]
 \[ \def\arraystretch{1.4}
 \cHright_t(z,s)=\left\{
 \begin{array}{ll}
  z & \mbox{if } \Re(z)\leq0 \mbox{ or }\Re(z)\geq 1\\
  z + (\frac1{1-s+st} -1)\Re(z) &\mbox{if } 0\leq \Re(z)\leq t\\
  z+s(1-\Re(z)) &\mbox{if } t\leq \Re(z)\leq 1.
 \end{array}
 \right.
 \]
\end{defn}
Roughly speaking, $\cHleft_t$ collapses the vertical strip $[0,t]\times\R$ to the vertical line $\C_{\Re=0}$
and expands the vertical strip $[t,1]\times\R$ to the vertical strip $[0,1]\times\R$; similarly
$\cHright_t$ collapses $[t,1]\times\R$ to $\C_{\Re=1}$ and expands $[0,t]\times\R$ to $[0,1]\times\R$.
Both homotopies restrict at each time $s$ to an endomorphism of the nice couple $(\bcRlr,\bdel)$,
so they induce homotopies
\[
 (\cHleft_t)_*\ ,\ (\cHright_t)_*\ \colon\Hur(\bcRlr,\bdel)\times[0,1]\to\Hur(\bcRlr,\bdel).
\]

\begin{proof}[Proof of Lemma \ref{lem:pathtobdel}]
 Let $\fc\in\Hur(\bcRlr,\bdel)$ and use Notation \ref{nota:fc}.
 Let $0<t<1$ be close enough to 1, so that for all $z\in P$ we have $\Re(z)=1$ or $\Re(z)\leq t$.
 Then $\cHleft_t(-;1)$ sends $P$ inside $\bdel\bcRlr$ and, therefore, $(\cHleft_t)_*$ induces a path in $\Hur(\bcRlr,\bdel)$ from $\fc$ to a configuration
 $\fc':=(\cHleft_t)_*(\fc,1)$ which is supported in $\bdel\bcRlr$.
\end{proof}

The following rhombus will help us to define a homotopy of $\C$ that squeezes the two segments in $\bdel$ to the two central points.
\begin{defn}
 \label{defn:diamo}
 We define $\diamo$ as the closed subspace of $\bH$ given by
 \[
  \diamo=\set{z\in\bH\,\colon\,\abs{\Re(z)-\frac 12}+\abs{\Im(z)-\frac 12}\leq\frac 12}.
 \]
 Geometrically, $\diamo$ is a closed rhombus centred at the point $\zcentre$ (see Notation \ref{nota:zcentre}).
 The boundary $\del\diamo$ contains points $z$ for which equality holds in the formula above.
 The corners of $\diamo$ are denoted
 $\zdiamleft=\frac{\sqrt{-1}}2$, $\zdiamright=1+\frac{\sqrt{-1}}2$,
 $\zdiamup=\frac 12+ \sqrt{-1}$ and $\zdiamdown=\frac 12$.
 We denote by $\bdiamolr$  the subspace of $\diamo$ given
 by
 \[
  \bdiamolr=\pa{\diamo\setminus\del\diamo}\cup \set{\zdiamleft,\zdiamright};
 \]
 we use the notation $\bdel\bdiamolr=\set{\zdiamleft,\zdiamright}=\bdiamolr\cap \bdel\bcRlr$,
 and we abbreviate the nice couple $(\bdiamolr,\bdel\bdiamolr)$
 as $(\bdiamolr,\bdel)$; compare with Notation \ref{nota:cRt}, and see Figure \ref{fig:diamolr}.
\end{defn}

\begin{figure}[ht]
 \begin{tikzpicture}[scale=4,decoration={markings,mark=at position 0.38 with {\arrow{>}}}]
  \draw[dashed,->] (-.1,0) to (1.1,0);
  \draw[dashed,->] (0,-.1) to (0,1.1);
  \draw[dotted, fill=gray, opacity=.5] (0,0) rectangle (1,1);
  \draw[very thick] (0,0) to (0,1);
  \draw[very thick] (1,0) to (1,1);
\begin{scope}[shift={(1.5,0)}]
  \draw[dashed,->] (-.1,0) to (1.1,0);
  \draw[dashed,->] (0,-.1) to (0,1.1);
  \draw[dotted, fill=gray, opacity=.5] (0,.5) to (.5,1) to (1,.5) to (.5,0) to (0,.5);
  \draw[very thick] (0,.49) rectangle (0.01,.51);
  \draw[very thick] (.99,.49) rectangle (1,.51);
\end{scope}
 \end{tikzpicture}
 \caption{The nice couples $(\bcRlr,\bdel)$ and $(\bdiamolr,\bdel)$.}
 \label{fig:diamolr}
\end{figure}

We have an inclusion 
of nice couples $(\bdiamolr,\bdel)\subset(\bcRlr,\bdel)$, inducing an inclusion
$\Hur(\bdiamolr,\bdel)\subset\Hur(\bcRlr,\bdel)$. 

\begin{lem}
 \label{lem:bcRvsdiamond}
 The inclusion $\Hur(\bdiamolr,\bdel)\subset \Hur(\bcRlr,\bdel)$ is a homotopy equivalence.
\end{lem}
Before proving Lemma \ref{lem:bcRvsdiamond} we define a suitable homotopy of $\C$.
\begin{defn}
 \label{defn:cHdiamo}
 For $z\in\C$ let $\fd^{\diamo}(z)=\min\set{\abs{\Re(z)-\frac 12} ;\frac 12}$.
We define a homotopy $\cHdiamo\colon\C\times[0,1]\to\C$ by the following formula.
 \[
 \def\arraystretch{1.4}
 \cHdiamo(z,s)=\left\{
 \begin{array}{ll}
  z-s \fd^{\diamo}(z)\sqrt{-1}  & \mbox{if } \Im(z)\geq 1\\
  z- 2s\fd^{\diamo}(z)(\Im(z)-\frac 12)\sqrt{-1}  & \mbox{if } 0 \leq \Im(z)\leq 1\\
  z+s\big(\frac{\Im(z)}2 +\fd^\diamo(z)\big)\sqrt{-1} & \mbox{if } \Im(z)\leq 0.
 \end{array}
 \right.  
 \]
\end{defn}
The homotopy $\cHdiamo$ satisfies the following properties:
\begin{itemize}
 \item for all $0\leq s\leq 1$, the map $\cHdiamo(-;s)\colon\C\to\C$
 induces an endomorphism of the nice couple $(\bcRlr,\bdel)$ and an endomorphism
 of the nice couple $(\bdiamolr,\bdel)$;
 \item $\cHdiamo(-;0)$ is the identity of $\C$;
 \item $\cHdiamo(-;1)$ sends $\bcRlr$ onto $\bdiamolr$ and $\bdel\bcRlr$
 onto $ \bdel\bdiamolr$.
\end{itemize}

\begin{proof}[Proof of Lemma \ref{lem:bcRvsdiamond}]
 By \cite[Proposition 4.4]{Bianchi:Hur2} the homotopy $\cHdiamo$ induces a homotopy
 $\cHdiamo_*\colon\Hur(\bcRlr,\bdel)\times[0,1]\to\Hur(\bcRlr,\bdel)$ starting from the identity and ending with a map $\Hur(\bcRlr,\bdel)\to\Hur(\bdiamolr,\bdel)$.
 The homotopy $\cHdiamo_*$ preserves the subspace $\Hur(\bdiamolr,\bdel)$ at all times
 and, thus, witnesses that the inclusion of $\Hur(\bdiamolr,\bdel)$ in
 $\Hur(\bcRlr,\bdel)$ is a homotopy equivalence.
\end{proof}

Note also that if $\fc\in\Hur(\bcRlr,\bdel)$ is supported in $\bdel\bcRlr$,
then the entire path $\cHdiamo_*(\fc,-)$ consists of configurations supported in $\bdel\bcRlr$.
Using Lemmas \ref{lem:pathtobdel} and \ref{lem:bcRvsdiamond} together, we can therefore
connect any $\fc\in\Hur(\bcRlr,\bdel)$ to a configuration $\fc'\in\Hur(\bcRlr,\bdel)$ supported in
$\bdel\bdiamolr=\set{\zdiamleft,\zdiamright}$.

We next define auxiliary configurations, supported on the three points $\zcentre,\zdiamleft,\zdiamright$: by moving $\zcentre$ towards $\zdiamleft$ or towards $\zdiamright$, we can construct paths between configurations supported on $\bdel\bdiamolr=\set{\zdiamleft,\zdiamright}$.
\begin{defn}
 \label{defn:fcgah}
  Recall Definition \ref{defn:fca}. For all $g,h\in G$ and $a\in\Q$ we define a configuration
  $\fc_{g,a,h}=(P,\psi,\phi)\in\Hur(\bcRlr,\bdel)$ as follows:
  \begin{itemize}
   \item $P=\set{\zcentre,\zdiamleft,\zdiamright}$; let $\gencentre,\gendiamleft,\gendiamright$ be an admissible
   generating set for $\fG(P)$, where $\gencentre$ is represented by a loop in $\bS_{0,1}\setminus P$,
   $\gendiamleft$ by a loop in $\bS_{-\infty,1/2}\setminus P$ and
   $\gendiamright$ by a loop in $\bS_{1/2,\infty}\setminus P$;
   \item $\psi$ maps $\gencentre\mapsto a$;
   \item $\phi$ maps $\gencentre\mapsto\fe(a)$, $\gendiamleft\mapsto g$ and $\gendiamright\mapsto h$.
  \end{itemize}
  We also define configurations $\fc_{\emptyset,a,h}$, $\fc_{g,\emptyset,h}$, $\fc_{g,a,\emptyset}$,
  $\fc_{g,\emptyset,\emptyset}$,  $\fc_{\emptyset,a,\emptyset}$ and   $\fc_{\emptyset,\emptyset, h}$
  in a similar way: for every occurrence of ``$\emptyset$'' we remove the corresponding point from $P$,
  and we define $\psi$ and $\psi$ on the relevant elements of the admissible generating set by the same formulas.
\end{defn}

\begin{proof}[Proof of Theorem \ref{thm:pi0bHurm}]
Note first that $\totmon(\fc_{\one_G,\emptyset, h})=h\in G$: this shows surjectivity of
$\totmon\colon\pi_0(\Hur_+(\bcRlr,\bdel))\to G$.

Lemma \ref{lem:pathtobdel} and the proof of Lemma \ref{lem:bcRvsdiamond} imply that every configuration $\fc\in\Hur_+(\bcRlr,\bdel)$
can be connected to a configuration supported
on $\bdel\bdiamolr$, i.e. of the form $\fc_{g,\emptyset,h}$, $\fc_{g,\emptyset,\emptyset}$ or $\fc_{\emptyset,\emptyset, h}$.

For all $g,h\in G$ the homotopies $\cHleft_{1/2}$ and $\cHright_{1/2}$
give paths joining the configuration $\fc_{\emptyset,\one,h}$ to $\fc_{\one_G,\emptyset,h}$ and $\fc_{\emptyset,\emptyset,h}$, respectively; the same homotopies
give paths joining the configuration $\fc_{g,\one,\emptyset}$
to $\fc_{g,\emptyset,\emptyset}$ and $\fc_{g,\emptyset,\one_G}$,
respectively. Thus, $\fc_{\emptyset,\emptyset,h}$ is connected to
$\fc_{\one_G,\emptyset,h}$, and
$\fc_{g,\emptyset,\emptyset}$ is connected to $\fc_{g,\emptyset,\one_G}$: we conclude that every
configuration in $\Hur(\bcRlr,\bdel)$ can be connected to a configuration of the form $\fc_{g,\emptyset,h}$.

Similarly, for all $g,h\in G$ and $a\in \Q$ the homotopies $\cHleft_{1/2}$ and $\cHright_{1/2}$
give paths joining the configuration $\fc_{g,a,h}$ to $\fc_{g\fe(a),\emptyset,h}$ and $\fc_{g,\emptyset,\fe(a)h}$ respectively. Thus $\fc_{g\fe(a),\emptyset,h}$ is connected to $\fc_{g,\emptyset,\fe(a)h}$.

Since we assumed that $\fe(\Q)$ generates $G$, we can write $g=\fe(a_1)^{\pm1}\cdot\dots\cdot\fe(a_r)^{\pm1}$. Using $r$ instances of the paths described above, or their inverses,
we can connect any configuration of the form $\fc_{g,\emptyset,h}$ to the corresponding
configuration $\fc_{\one_G,\emptyset,gh}$.
We have thus proved that every configuration in $\Hur_+(\bcRlr,\bdel)$ can be connected to a configuration of the form $\fc_{\one_G,\emptyset, h}$,
and these are sent bijectively to $G$ along $\totmon$.
\end{proof}

\section{Bar constructions of Hurwitz spaces}
\label{sec:barconstruction}
In this section we study the bar constructions of the topological monoids $\mHurm$
and $\bHurm$.
Many arguments of this and the next section are adapted from
\cite{Hatcher:Mum}, so familiarity with this paper may be valuable.

Recall that a topological monoid $M$ is \emph{group-like} if the monoid
$\pi_0(M)$ is a group; a standard argument ensures, in this case,
that for every $m\in M$ the maps given by left multiplication $\mu(m,-)\colon M\to M$
and right multiplication $\mu(-,m)$ are self-homotopy equivalences of $M$.
For left multiplication, for instance,
one chooses an element $m'\in M$ with $\mu(m',m)$ and $\mu(m,m')$ contained in the same component of
the neutral element $e$;
then a homotopy inverse of $\mu(m,-)$ is given by $\mu(m',-)$.
Note that this argument strongly relies on $M$ having a \emph{strict} neutral element $e$.

Unfortunately $\bHurm$ is a unital, but not group-like topological monoid; on the other hand its subspace
$\bHurm_+$ (see Notation \ref{nota:Hurmplus}) is a non-unital, but group-like topological monoid: see Theorem
\ref{thm:pi0bHurm}.
We will consider the space $\bHurm_+$ as a left module over $\bHurm$ in order to exploit the good properties
of both spaces.

\subsection{Bar constructions}
\label{subsec:barconstructions}
We recall the classical definition of bar construction with respect to a topological monoid $M$ and a left $M$-module $X$.
\begin{defn}
 \label{defn:BM}
Let $M$ be a topological monoid, let $X$ be a left $M$-module, and denote by
$\mu$ both multiplication maps $M\times M\to M$ and $M\times X\to X$.
We define a \emph{semisimplicial} space $B_{\bullet}(M,X)$. For $p\geq 0$, the space $B_p(M,X)$ of $p$-simplices is $M^p\times X$.
The face maps $d_i\colon B_p(M,X)\to B_{p-1}(M,X)$ are defined as follows:
\begin{itemize}
 \item $d_0\colon (m_1,\dots,m_p,x)\mapsto (m_2,\dots,m_p,x)$;
 \item $d_i\colon (m_1,\dots,m_p)\mapsto(m_1,\dots,\mu(m_i,m_{i+1}),\dots,m_p,x)$, for $1\leq i\leq p-1$;
 \item $d_p\colon (m_1,\dots,m_p)\mapsto (m_1,\dots, m_{p-1},\mu(m_p,x))$.
\end{itemize}
The space $B(M,X)$ is the thick geometric realisation of the semisimplicial space $B_{\bullet}(M,X)$, i.e. it is
the quotient of $\coprod_{p\geq0}\Delta^p\times M^p\times X$ by the equivalence relation $\sim$
generated by $(d^i(\uw),\um,x)\sim (\uw,d_i(\um,x))$, for all choices of the following data:
\begin{itemize}
 \item $p\geq0$ and $0\leq i\leq p$;
 \item a point $\uw=(w_0,\dots,w_{p-1})$ in $\Delta^{p-1}$, represented by its barycentric coordinates $w_0,\dots,w_{p-1}\ge0$ with $w_0+\dots+w_{p-1}=1$;
 \item a point $(\um,x)=(m_1,\dots,m_p,x)\in M^p\times X$
\end{itemize}
 Here $d^i\colon\Delta^{p-1}\to\Delta^p$ denotes the standard $i$\textsuperscript{th} face inclusion.

When $X=*$ is a point,
we also write $BM$ for $B(M,*)$; when $X=M$ with left multiplication coming from the monoid structure, we also write $EM$ for $B(M,M)$.
\end{defn}
In fact Definition \ref{defn:BM} only uses that $M$ is an associative non-unital monoid;
in Subsection \ref{subsec:thinbar} we will recall the \emph{thin} bar construction, which is
a simplicial space whose degeneracy maps are defined using the unit $e\in M$.

It is a standard fact that if $M$ is a topological monoid, $X$ is a left $M$-module and for all $m\in M$
the map $\mu(m;-)\colon X\to X$ is a self-homotopy equivalence of $X$, then the natural projection map
$\pr_X\colon B(M,X)\to BM=B(M,*)$ induced by the constant, $M$-equivariant map $X\to *$
is a quasi-fibration with fibres homeomorphic to $X$. See for instance \cite[Lemma D.1]{Hatcher:Mum}.

\begin{lem}
\label{lem:bHurmweonbHurm+}
Recall Definitions \ref{defn:Hurmoore} and \ref{defn:Hurmtopmon}
and let $(t,\fc)\in\bHurm$; then the left multiplication $\mu((t,\fc),-)$
restricts to a self-homotopy equivalence of $\bHurm_+$; moreover, if $\totmon(\fc)=\one\in G$,
then $\mu((t,\fc),-)|_{\bHurm_+}$ is homotopic to the identity of $\bHurm_+$.
It follows that
\[
\pr_{\bHurm_+}\colon B(\bHurm,\bHurm_+)\to B\bHurm
\]
is a quasifibration with fibre $\bHurm_+$.
\end{lem}
\begin{proof}
It suffices to prove the statement for one configuration $(t,\fc)$ in each connected component of $\bHurm$:
the statement is obvious for $(t,\fc)=(0,(\emptyset, \one,\one))\in\bHurm$, which is the neutral element
of $\bHurm$. Using Theorem \ref{thm:pi0bHurm}
we can then assume that $t=1$ and $\fc$ has the form $\fc^{\mathrm{d}}_g:=(P,\psi,\phi)$ for some $g\in G$, where
\begin{itemize}
 \item $P=\set{\zdiamdown}$ consists of the only point $\zdiamdown$ (see Definition \ref{defn:diamo});
 \item $\psi\colon\fQ(P)=\set{\one}\to\Q$ is the trivial map of PMQs;
 \item $\phi\colon\fG(P)\to G$ sends the unique standard generator of $\fG(P)$ to $g$.
\end{itemize}
We start with the case $g=\one$. We claim that $\mu((1,\fc^{\mathrm{d}}_{\one}),-)|_{\bHurm_+}$
is homotopic to the identity of $\bHurm_+$;
by Lemma \ref{lem:HurmvsHur} it suffices to prove that the restriction
\[
\mu((1,\fc_{\one}^{\mathrm{d}}),-)\colon\Hur_+(\bcR,\bdel)\to\bHurm_+
\]
is homotopic to the natural inclusion $\Hur_+(\bcR,\bdel)\hookrightarrow\bHurm_+$.

First we prove that the maps $\mu((1,\fc_{\one}^{\mathrm{d}}),-)$
and $\mu((1,(\emptyset,\one,\one)),-)$ are homotopic maps
$\Hur_+(\bcR,\bdel)\to\Hur_+(\bcR_2,\bdel)$.
We use an argument similar to the proof \cite[Proposition 7.10]{Bianchi:Hur2}. Recall from \cite[Definition 3.1]{Bianchi:Hur2} that the Ran space $\Ran_+(\bcR_2)$ is the space of non-empty finite subsets of $\bcR_2$; it is weakly contractible \cite[Theorem 5.5.1.6]{LurieHA}, and using the notion of standard explosion from \cite[Subsection 7.2]{Bianchi:Hur2}, one can find a homotopy
$\expl^{\zdiamdown}\colon\Ran_+(\bcR_2)\times[0,1]\to \Ran_+(\bcR_2)$
contracting $\Ran_+(\bcR_2)$ onto the configuration $\set{\zdiamdown}$. Recall also that there is an \emph{external product} $-\times-\colon\Hur_+(\bcR_2,\bdel)\times \Ran_+(\bcR_2) \to\Hur_+(\bcR_2,\bdel)$, which essentially superposes to a configuration in $\Hur_+(\bcR_2,\bdel)$ another configuration with trivial monodromies (i.e. a configuration in $\Ran_+(\bcR_2)$); see 
\cite[Definition 5.7 and Notation 5.9]{Bianchi:Hur2}.

We consider the following homotopy $\cH^{\zdiamdown}\colon\Hur_+(\bcR_2,\bdel)\times[0,1]\to\Hur_+(\bcR_2,\bdel)$
 \[
 \begin{tikzcd}[column sep =2cm]
  \Hur_+(\bcR_2,\bdel)\times[0,1] \ar[r,"{(\Id,\epsilon)\times\Id}"] &
  \Hur_+(\bcR_2,\bdel)\times \Ran_+(\bcR_2) \times[0,1] \ar[dl,"\Id\times\expl^{\zdiamdown}",swap]\\ 
  \Hur_+(\bcR_2,\bdel)\times \Ran_+(\bcR_2) \ar[r,"-\times-"]&  \Hur_+(\bcR_2,\bdel),
 \end{tikzcd}
 \]
where $\epsilon\colon\Hur_+(\bcR_2,\bdel)\to\Ran_+(\bcR_2)$ is the canonical map $(P,\psi,\phi)\mapsto P$. Roughly speaking, each point in the support of a configuration in $\Hur_+(\bcR_2,\bdel)$ is split at time 0 into two points: the first keeps the original local monodromy and does not move; the second carries a trivial local monodromy and moves straightly to $\zdiamdown$; at time 1 all the second points have merged at $\zdiamdown$. We observe the following:
\begin{itemize}
 \item the composition $\cH^{\zdiamdown}(-,0)\circ\mu((1,(\emptyset,\one,\one)),-)\colon \Hur_+(\bcR,\bdel)\to\Hur_+(\bcR_2,\bdel)$ is equal to $\mu((1,(\emptyset,\one,\one)),-)$, since $\cH^{\zdiamdown}(-,0)$ is the identity of $\Hur_+(\bcR_2,\bdel)$;
 \item the composition $\cH^{\zdiamdown}(-,1)\circ\mu((1,(\emptyset,\one,\one)),-)\colon \Hur_+(\bcR,\bdel)\to\Hur_+(\bcR_2,\bdel)$ is equal to $\mu((1,\fc_{\one}^{\mathrm{d}}),-)$.
\end{itemize}
We thus obtain that
$\mu((1,\fc_{\one}^{\mathrm{d}}),-)$
and $\mu((1,(\emptyset,\one,\one)),-)$ are homotopic as maps
$\Hur_+(\bcR,\bdel)\to\Hur_+(\bcR_2,\bdel)\subset\bHurm_+$.

We then note that $(1,(\emptyset,\one,\one))$ is connected by a path to $(0,(\emptyset,\one,\one))$ in $\bHurm$;
as a consequence $\mu((1,(\emptyset,\one,\one)),-)$ and $\mu((0,(\emptyset,\one,\one)),-)$
are homotopic as maps $\Hur_+(\bcR,\bdel)\to\bHurm_+$, and the second map is the natural inclusion.
This concludes the case $g=\one$.

Now let $g\neq \one$; by Theorem \ref{thm:pi0bHurm}
the three configurations $\mu((1,\fc_g^{\mathrm{d}}),(1,\fc_{g^{-1}}^{\mathrm{d}}))$, $\mu((1,\fc_{g^{-1}}^{\mathrm{d}}),(1,\fc_g^{\mathrm{d}}))$ and $(1,\fc_{\one}^{\mathrm{d}})$ are in the same connected component of $\bHurm_+$: hence, $\mu((1,\fc_g^{\mathrm{d}}),-)$ and $\mu((1,\fc_{g^{-1}}^{\mathrm{d}}),-)$ are homotopy inverses as maps $\bHurm_+\to\bHurm_+$.
\end{proof}

It is a classical fact that if $M$ is a \emph{unital}, topological monoid,
then $EM$ is contractible.
In the following proposition we prove an analogous statement for $B(\bHurm,\bHurm_+)$.
\begin{prop}
 \label{prop:BbHurmbHurm+contractible}
 The space $B(\bHurm,\bHurm_+)$ is weakly contractible.
\end{prop}
The proof of Proposition \ref{prop:BbHurmbHurm+contractible} is in Appendix \ref{subsec:BbHurmbHurm+contractible}.
As a consequence of Lemma \ref{lem:bHurmweonbHurm+} and Proposition \ref{prop:BbHurmbHurm+contractible} we obtain the following theorem.
\begin{thm}
 \label{thm:bHurmloop}
 There is a weak equivalence $\bHurm_+\simeq \Omega B\bHurm$.
\end{thm}

\subsection{Pontryagin ring and group completion}
If $M$ is a unital topological monoid, $H_*(M)$ is an associative, graded ring with unit, called Pontryagin ring.
We usually denote by $x\cdot y\in H_*(M)$ the Pontryagin product of two homology classes $x,y\in H_*(M)$.
The unit $1\in H_0(M)$ is the homology class corresponding to the connected component of $e$ in $\pi_0(M)$.
The subset $\pi_0(M)\subset H_0(M)\subset H_*(M)$ is closed under multiplication.

\begin{defn}
 \label{defn:weaklybraided}
A topological monoid $M$ is \emph{weakly braided} if there is a homeomorphism $\braiding\colon M\times M\to M\times M$
such that
\begin{itemize}
 \item if $\pr_1,\pr_2\colon M\times M\to M$ are the two natural
 projections, then $\pr_1\circ\braiding=\pr_2$ as maps $M\times M\to M$;
 \item $\mu$ and $\mu\circ\braiding$ are homotopic as maps $M\times M\to M$.
\end{itemize}
\end{defn}
Note that if $M$ is weakly braided, then the ring localisation $H_*(M)[\pi_0(M)^{-1}]$ can be constructed
by right fractions:
for all $x\in H_*(M)$ and $a\in\pi_0(M)$
there exist $y\in H_*(M)$ and $b\in\pi_0(M)$ with $x\cdot b=a\cdot y$. This follows from setting $b=a$ and
$y=(\pr_2)_*\circ\braiding_*(x\times a)$, where $\times$ denotes the homology cross-product,
and $\pr_2\colon M\times M\to M$ is as in Definition \ref{defn:weaklybraided}.

\begin{lem}
 \label{lem:mHurmweaklybraided}
The topological monoid $\mHurm$ is weakly braided.
\end{lem}
\begin{proof}
 We define $\braiding\colon\mHurm\times\mHurm\to\mHurm\times\mHurm$ by the formula
 \[
 \braiding((t,\fc),(t',\fc'))=((t',\fc'),(t,\fc^{\hat\totmon(\fc')})),
 \]
 where $\hat\totmon$ is the $\hQ$-valued total monodromy (see Notation \ref{nota:muHurmcdotomega})
 and we use the action by global conjugation \cite[Definition 6.6]{Bianchi:Hur2}.
 It is clear that $\braiding$ is a homeomorphism and that the first property in Definition \ref{defn:weaklybraided} holds.
 To check the second property, note that $\braiding$ restricts to a map
 \[
  \braiding\colon \Hur(\mcR)\times\Hur(\mcR)\to \Hur(\mcR)\times\Hur(\mcR).
 \]
 By Lemma \ref{lem:HurmvsHur} it suffices to prove that $\mu$ and $\mu\circ\braiding$
 are homotopic when considered as maps $\Hur(\mcR)\times\Hur(\mcR)\to \Hur(\mcR_2)$.
 Let $\mcR^{1/2}\subset\mcR$ be the open unit square $(1/4,3/4)\times(1/4,3/4)$ of side length $1/2$
 centred at $\zcentre\in\mcR$; we can regard $\Hur(\mcR^{1/2})$
 as an open subspace of $\Hur(\mcR)$, containing all configurations supported in $\mcR^{1/2}$. Note that $\braiding$ restricts to a map 
 \[
  \braiding\colon \Hur(\mcR^{1/2})\times\Hur(\mcR^{1/2})\to \Hur(\mcR^{1/2})\times\Hur(\mcR^{1/2}).
 \]
 
 Let $\cH^{1/2}\colon\C\times[0,1]\to\C$ be a semi-algebraic isotopy of $\C$
 fixing $*$ at all times,
 such that $\cH^{1/2}(-,0)=\Id_{\C}$ and
 $\cH^{1/2}(-,1)$ restricts to a homeomorphism $\mcR\to\mcR^{1/2}$.
 Then by functoriality there is a deformation of
 $\Hur(\mcR)$ into the subspace $\Hur(\mcR^{1/2})$. Thus it suffices to prove that the following restricted maps are homotopic:
 \[
  \mu\ ,\ \mu\circ\braiding\ \colon\Hur(\mcR^{1/2})\times\Hur(\mcR^{1/2})\to\Hur(\mcR_2).
 \]
 Let $\cH^{\braiding}\colon\C\times[0,1]\to\C$ be a semi-algebraic isotopy of $\C$ fixing pointwise
 $\C\setminus\mcR_2$ at all times, such that $\cH^{\braiding}(-,0)=\Id_{\C}$ and
 $\cH^{\braiding}(-,1)\colon\C\to\C$ has the following properties:
 \begin{itemize}
  \item $\cH^{\braiding}(-,1)$ restricts to $\tau_1\colon \mcR^{1/2}\to\tau_1(\mcR^{1/2})$ (see Definition \ref{defn:taut});
  \item $\cH^{\braiding}(-,1)$ restricts to $\tau_{-1}\colon\tau_1(\mcR^{1/2})\to  \mcR^{1/2}$;
  \item $\cH^{\braiding}(-,1)$ restricts to a self-homeomorphism of
  $\C\setminus\pa{\mcR^{1/2}\cup\tau_1(\mcR^{1/2})}$ representing a clockwise half Dehn twist, for instance we may assume that
  there is a simple loop $\gamma\subset\bS_{1,2}\setminus \tau_1(\mcR^{1/2})$ spinning clockwise around
  $\tau_1(\mcR^{1/2})$, such that $\cH^{\braiding}(-,1)\circ\gamma$ is a simple loop contained
  in $\bS_{0,1}\setminus \mcR^{1/2}$ and spinning clockwise around $\mcR^{1/2}$.
 \end{itemize}
 Then the composition of $\mu$ with $\cH^{\braiding}_*$ gives a homotopy from $\mu$ to $\mu\circ\braiding$
 as maps $\Hur(\mcR^{1/2})\times\Hur(\mcR^{1/2})\to\Hur(\mcR_2)$.
\end{proof}

Recall that for any topological monoid $M$ there is a canonical map
$M\to\Omega BM$: the induced map in homology
$H_*(M)\to H_*(\Omega BM)$ sends the multiplicative subset $\pi_0(M)\subset H_*(M)$
to the set of invertible elements of the Pontryagin ring $H_*(\Omega BM)$.
Therefore, there is an induced map of rings
\[
 H_*(M)[\pi_0(M)^{-1}]\to H_*(\Omega BM).
\]
We recall the group completion theorem (see \cite{McDuffSegal} and \cite[Theorem Q.4]{FM94}).
\begin{thm}[Group completion theorem]
\label{thm:groupcompletion}
Let $M$ be a topological monoid and suppose that the localisation $H_*(M)[\pi_0(M)^{-1}]$
can be constructed by right fractions. Then the canonical
map
\[
  H_*(M)[\pi_0(M)^{-1}]\to H_*(\Omega BM)
\]
is an isomorphism of rings.
\end{thm}
Using Theorem \ref{thm:groupcompletion} together with Lemma \ref{lem:mHurmweaklybraided} we obtain an isomorphism of rings
\[
   H_*(\mHurm)[\pi_0(\mHurm)^{-1}]\cong H_*(\Omega B\mHurm).
\]

\subsection{Thin bar construction}
\label{subsec:thinbar}
Recall Definition \ref{defn:BM}: the semisimplicial space $B_{\bullet} (M,X)$ can be enhanced to a simplicial space by defining
the degeneracy map $s_i\colon B_k(M,X)\to B_{k+1}(M,X)$, for $0\leq i\leq k$, by the following formula, where $e$ denotes the neutral element of $M$.:
\[
 s_i\colon(m_1,\dots,m_k,x)\mapsto (m_1,\dots,m_i,e,m_{i+1},\dots,m_k,x).
\]
\begin{defn}
 \label{defn:reBM}
 The simplicial space defined above is denoted by $\reB_{\bullet}(M,X)$; its geometric realisation \emph{as a simplicial space}
 is denoted by $\reB (M,X)$ and called the \emph{thin bar construction}. It is
 the quotient of $BM$ by the equivalence relation $\bar\sim$
 generated by $[s^i(\uw),\um,x]\bar\sim [\uw,s_i(\um,x)]$
 for all choices of the following data:
\begin{itemize}
 \item $p\geq0$ and $0\leq i\leq p$;
 \item a point $\uw=(w_0,\dots,w_{p+1})$ in $\Delta^{p+1}$, represented by its barycentric coordinates;
 \item a point $(\um,x)=(m_1,\dots,m_p,x)\in M^p\times X$.
\end{itemize}
 Here $s^i\colon\Delta^{p+1}\to\Delta^p$ denotes the $i$\textsuperscript{th} degeneracy.
 The natural projection map is denoted by $\pr_{\reB}\colon B(M,X)\to \reB (M,X)$.
 In the case $X=*$ we also write $\reB M$ for $\reB(M,*)$.
\end{defn}
It is a classical fact that if $M$ is well-pointed, then $\pr_{\reB}\colon BM\to \reB M$ is a weak homotopy equivalence,
as $\reB_{\bullet}(M)$ is a \emph{good} simplicial space in the sense of \cite[Appendix 2]{Segal73}.
The monoids $\mHurm$ and $\bHurm$ are well-pointed, as the connected component of the unit $(0,(\emptyset,\one,\one))$ is contractible in both cases.

\section{Deloopings of Hurwitz spaces}
\label{sec:deloopings}
In this section we describe the weak homotopy types of $B\mHurm$ and $B\bHurm$ using suitable, relative Hurwitz spaces. Recall from \cite[Definition 6.9]{Bianchi:Hur2} that a left-right-based (lr-based)
nice couple $(\zleft,\fT,\zright)$ is a nice couple $\fT=(\cX,\cY)$, together with a choice of two points $\zleft,\zright\in\cY$
satisfying
\[
\Re(\zleft)=\min\set{\Re(z)\,|\, z\in\cX}<\max\set{\Re(z)\,|\,z\in\cX}=\Re(\zright).
\]
We denote by $\Hur(\fT)_{\zleft,\zright}=\Hur(\fT;\Q,G)_{\zleft,\zright}$ the subspace of $\Hur(\fT)$ of configurations whose support contains $\set{\zleft,\zright}$;
recall from \cite[Definition 6.12]{Bianchi:Hur2} that there is an action of $G\times G^{\op}$ on the Hurwitz space $\Hur(\fT)_{\zleft,\zright}$,
i.e. there are compatible actions of $G$ on left and on right on this space;
by \cite[Lemma 6.16]{Bianchi:Hur2} the quotient map
\[
 \pr_{G,G^{\op}}\colon \Hur(\fT)_{\zleft,\zright}\to \Hur(\fT)_{G,G^{\op}}:=\Hur(\fT)_{\zleft,\zright}/G\times G^{\op}
\]
is a covering map, with $G\times G^{\op}$ as the group of deck transformations.
\begin{thm}
 \label{thm:delooping}
 Recall Definitions \ref{defn:Hurmoore}, \ref{defn:diamo} and \ref{defn:BM}.
 Let $(\Q,G)$ be a PMQ-group pair and assume that $\fe(\Q)$ generates $G$ as a group; then there are weak homotopy equivalences
 \[
 \sigma\colon B\mHurm\to \Hur(\bdiamolr,\bdel)_{G,G^{\op}};
\quad\quad\quad 
 \sigma \colon B\bHurm\to \Hur(\diamo,\del)_{G,G^{\op}}.
 \]
 Here we consider the lr-based nice couples $(\zdiamleft,(\bdiamolr,\bdel),\zdiamright)$
 and $(\zdiamleft,(\diamo,\del),\zdiamright)$.
\end{thm}

We will use a classical approach, going back to Segal \cite{Segal73}, which allows us to model the classifying space of a monoid $M$, arising from configuration spaces,
with another, \emph{relative} configuration space.
We will follow tightly the strategy of the proof of
\cite[Proposition 3.1]{Hatcher:Mum}, and to some extent we will use the same notation:
we do this for convenience of the reader.
We focus on the case of $\mHurm$ and write in parentheses the necessary changes for $\bHurm$.

We will first define the comparison
map $\sigma$ in the two cases, and then show that it induces isomorphisms on all homotopy groups.

\subsection{Definition of the comparison map}
\label{subsec:defnsigma}
By Definition \ref{defn:BM} the space $B\mHurm$ (respectively, $B\bHurm$) arises as a
quotient of the disjoint union $\coprod_{p\geq 0}\Delta^p\times\mHurm^p$ (respectively, $\coprod_{p\geq 0}\Delta^p\times\bHurm^p$). We will first define $\sigma$ on this disjoint union, and then prove
that the given assignment induces a map on the quotient.

\begin{nota}
 \label{nota:uwutufc}
 We usually denote by $(\uw;\ut,\ufc)$ a point in $\coprod_{p\geq 0}\Delta^p\times\mHurm^p$ (respectively, $\coprod_{p\geq 0}\Delta^p\times\bHurm^p$), where
 \begin{itemize}
  \item  $\uw=(w_0,\dots,w_p)$ is a system of barycentric
 coordinates in $\Delta^p$, i.e. $w_0,\dots,w_p\geq 0$ and $w_0+\dots+w_p=1$;
 \item $\ut=(t_1,\dots,t_p)$ and $\ufc=(\fc_1,\dots\fc_p)$, such that
 $(t_i,\fc_i)$ is an element in $\mHurm$ (in $\bHurm$) for all $1\leq i\leq p$.
 \end{itemize}
We usually present $\fc_i$ as $(P_i,\psi_i)$ (respectively, as $(P_i,\psi_i,\phi_i)$).
\end{nota}
Given $(\uw,\ut,\ufc)$ as in Notation \ref{nota:uwutufc}, note that
the product $(t_1,\fc_1)\dots(t_p,\fc_p)$ has the form $(t_1+\dots+t_p,\fc)$, with
$\fc$ supported on the set
\[
P=P_1\cup(P_2+t_1)\cup(P_3+t_1+t_2)\cup\dots\cup(P_p+t_1+\dots+t_{p-1}).
\]
By Definition \ref{defn:Hurmoore} we have $\fc\in\Hur(\mcR_{\infty})$ (respectively, $\fc\in\Hur(\bcR_{\infty},\bdel)$), but in the following we will consider $\fc$ as a configuration in $\Hur(\mcR_{\R})$ (in $\Hur(\bcR_{\R},\bdel)$, see Notation \ref{nota:cRt}).

\begin{defn}
 \label{defn:hmu}
The above assignment $(\uw;\ut,\ufc)\mapsto \fc$ gives a continuous map
\[
 \hmu\colon \coprod_{p\geq 0} \Delta^p\times(\mHurm)^p \to \Hur(\mcR_{\R})
 \quad
 \pa{ \mbox{respectively, } \hmu\colon \coprod_{p\geq 0} \Delta^p\times(\bHurm)^p \to \Hur(\bcR_{\R},\bdel) }.
\]
\end{defn}
Note that $\hmu$ factors, on each subspace $\Delta^p\times(\mHurm)^p$ (respectively, $\Delta^p\times(\bHurm)^p$),
through the projection on the factor $(\mHurm)^p$ (respectively, $(\bHurm)^p$). The first subspace
$\Delta^0$ is sent to the empty product in $\mHurm$ (in $\bHurm$), i.e. to the neutral element $(0,(\emptyset,\one,\one))$.
\begin{defn}
 \label{defn:barycentres}
For $(\uw;\ut,\ufc)$ as in Notation \ref{nota:uwutufc}, define $a_0=0$
and $a_i=\sum_{j=1}^i t_j$ for all $1\leq i\leq p$. Define the \emph{barycentre} of
$(\uw;\ut,\ufc)$ as
$b=\sum_{i=0}^pw_ia_i$. Set $a_i^+=\max\set{a_i,b}$ and $a_i^-=\min\set{a_i,b}$ for all $0\leq i\leq p$,
and define the \emph{upper barycentre} and the \emph{lower barycentre} as
$b^+=\sum_{i=0}^pw_ia_i^+$ and $b^-=\sum_{i=0}^pw_ia_i^-$.
\end{defn}
See Figure \ref{fig:sigma}(left).
Note that the barycentres
$b,b^+,b^-$ vary continuously on
$\coprod_{p\geq 0} \Delta^p\times(\mHurm)^p$ (on $\coprod_{p\geq 0} \Delta^p\times(\bHurm)^p$), but do not factor to continuous functions on $B\mHurm$ (respectively, $B\bHurm$): indeed, if $w_0=0$, the triple $(\uw;\ut,\ufc)$ is equivalent to the triple $(\uw';\ut',\ufc')$ obtained by removing $w_0$, $t_1$ and $\ufc_1$; all barycentres $b,b^+,b^-$ drop by $t_1$ when passing from the first to the second triple. Nevertheless, the \emph{differences} $b^+-b$ and $b-b^-$ factor to continuous functions defined on  $B\mHurm$ (respectively, $B\bHurm$).

Note also that for all $(\uw;\ut,\ufc)$ we have
$a_0\leq b^-\leq b\leq b^+\leq a_p$. More precisely, let $i_{\min}\ge0$ be minimal with $w_{i_{\min}}>0$ and let $i_{\max}\le p$ be maximal with $w_{i_{\max}}>0$; then $a_{i_{\min}}\leq b^-\leq b\leq b^+\leq a_{i_{\max}}$,
with all these inequalities
strict unless they are all equalities: in this case all $t_i$ with $i_{\min}<i\le i_{\max}$ are equal to $0$ and all corresponding $\fc_i$ are
equal to $(\emptyset,\one,\one)$, so that
$\hmu(\uw,\ut,\ufc)$ is equal to the product $(t_1,\fc_1)\dots(t_{i_{\min}},\fc_{i_{\min}})\cdot (t_{i_{\max}+1},\fc_{i_{\max}+1})\dots(t_p,\fc_p)$. In particular, if $b^-=b=b^+$ then we have that $\hmu(\uw,\ut,\ufc)$ is supported away from $\bS_{b,b}$, in fact it is supported away from $\bS_{b-\epsilon,b+\epsilon}$ for $\epsilon>0$ small enough.
\begin{defn}
\label{defn:epsilon}
 We define a continuous function $\epsilon\colon \coprod_{p\geq 0} \Delta^p\times(\mHurm)^p\to[0,1]$ (respectively, $\epsilon\colon \coprod_{p\geq 0} \Delta^p\times(\bHurm)^p\to[0,1]$): for $(\uw;\ut,\ufc)$ as in Notation \ref{nota:uwutufc}, we denote by $P\subset\R\times[0,1]$ the support of $\hmu(\uw;\ut,\ufc)$ and set
 \[
  \epsilon\colon(\uw;\ut,\ufc)=\frac 12\sup\set{t\in[0,1]\,|\,P\cap\bS_{b^--t,b^++t}=\emptyset},
 \]
where upper and lower barycentres are computed with respect to $(\uw;\ut,\ufc)$. We denote $b^-_\epsilon=b^--\epsilon$ and $b^+_\epsilon=b^++\epsilon$.
\end{defn}
We observe that $\epsilon$ satisfies the following properties:
 \begin{itemize}
  \item for all $(\uw,\ut,\ufc)$ satisfying $b^-=b^+$, we have $\epsilon(\uw,\ut,\ufc)>0$;
  \item for all $(\uw,\ut,\ufc)$ with $\epsilon(\uw,\ut,\ufc)>0$, the configuration $\hmu(\uw,\ut,\ufc)$ is supported away from
  $\bS_{b^--\epsilon,b^++\epsilon}$.
\end{itemize}

The advantage of replacing $b^-$ and $b^+$ by $b^-_\epsilon$ and $b^+_\epsilon$ is that we now have strict inequalities $b^-_\epsilon<b<b^+_\epsilon$ for all $(\uw;\ut,\ufc)$. As we will see, the disadvantage that $b^+_\epsilon-b^-_\epsilon$ does not factor through a function defined on $B\mHurm$ (respectively, on $B\bHurm$) will be inessential.

Recall the proof of Lemma \ref{lem:HurmvsHur}: for $s>0$ the map $\Lambda_s\colon\C\to\C$ is
an endomorphism of the nice couple $(\mcR_{\R},\emptyset)$ (respectively, $(\bcR_{\R},\bdel)$), depending continuously on $s$. We obtain a continuous map
\[
 \Lambda_* \!\colon\! \Hur(\mcR_{\R})\times (0,\infty)\to\!\Hur(\mcR_{\R})
 \quad
 \pa{\!\mbox{respectively, }\Lambda_*\! \colon\! \Hur(\bcR_{\R},\bdel)\times (0,\infty)\to\!\Hur(\bcR_{\R},\bdel)\!}.
\]
Similarly, recall Definition \ref{defn:taut}: for all $t\in\R$ the map $\tau_t$ is an
endomorphism of the nice couple $(\mcR_{\R},\emptyset)$ (respectively, $(\bcR_{\R},\bdel)$), depending continuously on $t$. We obtain a continuous map
\[
 \tau_* \colon \Hur(\mcR_{\R})\times \R \to\Hur(\mcR_{\R})\quad\quad
 \pa{\mbox{respectively, }\tau_* \colon \Hur(\bcR_{\R},\bdel)\times \R \to\Hur(\bcR_{\R},\bdel)}.
\]

\begin{defn}
 \label{defn:hmub}
We define a map
\[
 \hmu^b\colon \coprod_{p\geq 0} \Delta^p\times(\mHurm)^p \to \Hur(\mcR_{\R})\quad\quad
 \pa{\mbox{respectively, } \hmu^b\colon \coprod_{p\geq 0} \Delta^p\times(\bHurm)^p \to \Hur(\bcR_{\R})}
\]
by the following assignment:
\[
(\uw;\ut,\ufc)\mapsto
\Lambda_*\pa{\tau_*\pa{\hmu(\uw;\ut,\ufc)\ ;\ -b^-_\epsilon(\uw;\ut,\ufc)}\ ;\ \frac{1}{b^+_\epsilon(\uw;\ut,\ufc)-b^-_\epsilon(\uw;\ut,\ufc)}} 
\] 
\end{defn}

Roughly speaking, the map $\hmu^b$ has the effect of a horizontal translation and a dilation of
the configuration $\hmu(\uw;\ut,\ufc)$: the effect of the translation and dilation
is to map the rectangle $[b^-_\epsilon,b^+_\epsilon]\times[0,1]$ homeomorphically onto the unit square $\cR$.

\begin{defn}
 \label{defn:mfTbox}
Let $\mathring{\diamo}$ denote the interior of $\diamo$ (see Definition \ref{defn:diamo}),
and denote $\R+\frac{\sqrt{-1}}{2}=\set{t+\frac{\sqrt{-1}}{2}\,|\,t\in\R}\subset\bH$. We introduce several nice couples:
\[
\begin{array}{ll}
\bullet\ \bfTbox=\pa{\mcR_{\R},\mcR_{\R}\setminus\mcR}; &
\bullet\ \bfTdiamo=
\pa{\bdiamolr\cup (\R+\frac{\sqrt{-1}}{2}),\pa{\bdiamolr\cup (\R+\frac{\sqrt{-1}}{2})}
\setminus\mathring{\diamo}};\\[.9em]
\bullet\ \fTbox=\pa{\bcR_{\R},\bcR_{\R}\setminus\mcR};&
\bullet\ \fTdiamo=
\pa{\diamo\cup (\R+\frac{\sqrt{-1}}{2}),\pa{\diamo\cup (\R+\frac{\sqrt{-1}}{2})}
\setminus\mathring{\diamo}}.
\end{array}
\]
\end{defn}
Since $\Id_\C$ is a map of nice couples $(\mcR_{\R},\emptyset)\to\bfTbox$ (respectively,  $(\bcR_{\R},\bdel)\to\fTbox$), it induces a
map $\Hur(\mcR_{\R})\to\Hur(\bfTbox)$ (respectively, $\Hur(\bcR_{\R},\bdel)\to\Hur(\fTbox)$).
See Figure \ref{fig:mfTbox}.

\begin{figure}[ht]
 \begin{tikzpicture}[scale=3,decoration={markings,mark=at position 0.38 with {\arrow{>}}}]
  \draw[dashed,->] (-.4,0) to (1.4,0);
  \draw[dashed,->] (0,-.2) to (0,1.2);
  \fill[gray, opacity=.5] (-.4,0) rectangle (1.4,1);
  \fill[opacity=.4] (-.4,0) rectangle (0,1);
  \fill[opacity=.4] (1,0) rectangle (1.4,1);
  \draw[dotted] (-.4,0) to (1.4, 0);
  \draw[dotted] (-.4,1) to (1.4, 1);

\begin{scope}[shift={(2.1,0)}]
  \draw[dashed,->] (-.4,0) to (1.4,0);
  \draw[dashed,->] (0,-.2) to (0,1.2);
  \draw[dotted, fill=gray, opacity=.5] (0,.5) to (.5,1) to (1,.5) to (.5,0) to (0,.5);
  \draw[very thick] (-.4,.5) to (0, .5);
  \draw[very thick] (1,.5) to (1.4, .5);
\end{scope}

\begin{scope}[shift={(0,-1.6)}]
  \draw[dashed,->] (-.4,0) to (1.4,0);
  \draw[dashed,->] (0,-.2) to (0,1.2);
  \fill[gray, opacity=.5] (-.4,0) rectangle (1.4,1);
  \fill[opacity=.4] (-.4,0) rectangle (0,1);
  \fill[opacity=.4] (1,0) rectangle (1.4,1);
  \draw[very thick] (-.4,0) to (1.4, 0);
  \draw[very thick] (-.4,1) to (1.4, 1);
\end{scope}

\begin{scope}[shift={(2.1,-1.6)}]
  \draw[dashed,->] (-.4,0) to (1.4,0);
  \draw[dashed,->] (0,-.2) to (0,1.2);
  \draw[very thick] (0,.5) to (.5,1) to (1,.5) to (.5,0) to (0,.5);
  \fill[gray, opacity=.5] (0,.5) to (.5,1) to (1,.5) to (.5,0) to (0,.5);
  \draw[very thick] (-.4,.5) to (0, .5);
  \draw[very thick] (1,.5) to (1.4, .5);
\end{scope}
\end{tikzpicture}
 \caption{Top: the nice couples $\bfTbox$ and $\bfTdiamo$.;Bottom: the nice couples $\fTbox$ and $\fTdiamo$.}
 \label{fig:mfTbox}
\end{figure}

\begin{nota}
 \label{nota:hmub}
By abuse of notation we will denote by $\hmu^b$ also the following composition:
\[
 \begin{tikzcd}
  \coprod_{p\geq 0} \Delta^p\times(\mHurm)^p\ar[r,"\hmu^b"] & \Hur(\mcR_{\R})
  \ar[r,"(\Id_\C)_*"] & \Hur(\bfTbox)
 \end{tikzcd}
\]
\[
 \begin{tikzcd}
  \Big(\mbox{respectively, }\coprod_{p\geq 0} \Delta^p\times(\bHurm)^p\ar[r,"\hmu^b"] & \Hur(\bcR_{\R})
  \ar[r,"(\Id_\C)_*"] & \Hur(\fTbox)\Big).
 \end{tikzcd}
\]
\end{nota}

\begin{defn}
 \label{defn:kappa+-}
We define maps $\kappa^-$ and $\kappa^+ \colon\C\times[0,\infty)\to\C$ by the formulas
\[
\begin{array}{ll}
 \kappa^-(z,s)= &
 \left\{
 \begin{array}{ll}
 z &\mbox{if } \Re(z)\geq 0\\
 z-\Re(z) &\mbox{if } -s\leq \Re(z)\leq 0\\
 z+s&\mbox{if }\Re(z)\leq -s ;\\
 \end{array}
 \right.\\[2em]
 \kappa^+(z,s)=&
 \left\{
 \begin{array}{ll}
 z &\mbox{if } \Re(z)\leq 1\\
 z- \Re(z)+1 &\mbox{if } 1\leq \Re(z)\leq 1+s\\
 z-s&\mbox{if }\Re(z)\geq 1+s .\\
 \end{array}
 \right.
\end{array}
\]
\end{defn}
Roughly speaking, both $\kappa^-$ and $\kappa^+$ fix the vertical strip $[0,1]\times\R$
for all $s\geq0$;
the map $\kappa^-(-,s)$ collapses the strip $[-s,0]\times\R$ 
to the vertical line $\C_{\Re=0}$, and translates $(-\infty,s]\times\R$ to the right;
instead $\kappa^+(-,s)$
collapses the strip $[1,1+s]\times\R$ to the vertical line $\C_{\Re=1}$, and translates
$[1+s,\infty)\times\R$ to the left.

Both $\kappa^-(-,s)$ and $\kappa^+(-,s)$ are morphisms of nice couples $\bfTbox\to\bfTbox$
(respectively, $\fTbox\to\fTbox$) for all $s\geq0$. We obtain continous maps
\[
 \kappa^-_*\,,\,\kappa^+_*\,\colon\Hur(\bfTbox)\times[0,\infty)\to \Hur(\bfTbox)
\]
\[
 \pa{\mbox{respectively, }\kappa^-_*\,,\,\kappa^+_*\,\colon\Hur(\fTbox)\times[0,\infty)\to \Hur(\fTbox)}.
\]
\begin{nota}
 \label{nota:cHdiamo1}
Recall Definition \ref{defn:cHdiamo}; we use the notation $\cHdiamo_1:=\cHdiamo(-;1)\colon\C\to\C$. 
\end{nota}
Note that $\cHdiamo_1$ is a morphism
of nice couples $\bfTbox\to\bfTdiamo$ (respectively, $\fTbox\to\fTdiamo$).
\begin{defn}
 \label{defn:barsigma}
We define a map
\[
\hmu^\Box\colon \coprod_{p\geq 0} \Delta^p\times(\mHurm)^p \to \Hur(\bfTbox)\quad\quad
\pa{\mbox{respectively, }\hmu^\Box\colon \coprod_{p\geq 0} \Delta^p\times(\bHurm)^p \to \Hur(\fTbox)}
\]
by the following assignment, where $p$, $a_p$, $b^-_\epsilon$ and $b^+_\epsilon$
depend on $(\uw;\ut,\ufc)$:
\[
 \hmu^\Box(\uw;\ut,\ufc)=
      \kappa^-_*\pa{
         \kappa^+_*\pa{
         \hmu^b(\uw;\ut,\ufc)
             \ ;\ \frac{a_p-b^+_\epsilon}{b^+_\epsilon-b^-_\epsilon}
             }
         \ ;\ \frac{b^-_\epsilon-a_0}{b^+_\epsilon-b^-_\epsilon}
         }
\]
We further define
\[
\hmu^\diamo\colon \coprod_{p\geq 0} \Delta^p\times(\mHurm)^p \to \Hur(\bfTdiamo)\quad\quad
\pa{\mbox{respectively, }\hmu^\diamo\colon \coprod_{p\geq 0} \Delta^p\times(\bHurm)^p \to \Hur(\fTdiamo)}
\]
as the composition $(\cHdiamo_1)_*\circ \hmu^\Box$.
\end{defn}
Roughly speaking, $\hmu^\Box$ improves the effect of $\hmu^b$ as follows: $\hmu^b(\uw;\ut,\ufc)$
is a configuration supported in the rectangle $\Big[\frac{b_\epsilon^--a_0}{b_\epsilon^+-b_\epsilon^-},1+\frac{a_p-b_\epsilon^+}{b_\epsilon^+-b_\epsilon^-}\Big]\times[0,1]$,
and the further application of $\kappa^-_*$ and $\kappa^+_*$ collapse $\hmu^b(\uw;\ut,\ufc)$
to a configuration supported in $\cR$.
The further composition $\hmu^\diamo$ changes the configuration $\hmu^\Box(\uw;\ut,\ufc)$ to
a configuration $\hmu^\diamo(\uw;\ut,\ufc)$ supported in $\diamo$.

Note that there is a natural inclusion of spaces
$\Hur(\bdiamolr,\bdel)\subset\Hur(\bfTdiamo)$ (respectively, $\Hur(\diamo,\del)\subset\Hur(\fTdiamo)$).
The following lemma summarises the previous discussion.
\begin{lem}
 \label{lem:checksigmaindiamo}
 The map $\hmu^\diamo$ has values inside 
 $\Hur(\bdiamolr,\bdel)$ (inside $\Hur(\diamo,\del)$).
\end{lem}
\begin{proof}
 Let $(\uw;\ut,\ufc)$ be as in Notation \ref{nota:uwutufc}.
 Then $\hmu^b(\uw;\ut,\ufc)$
 is supported in the rectangle $\Big(\frac{-b_\epsilon^--a_0}{b_\epsilon^+-b_\epsilon^-},1+\frac{a_p-b_\epsilon^+}{b_\epsilon^+-b_\epsilon^-}\Big)\times(0,1)$
 (in the rectangle $\Big(\frac{-b_\epsilon^--a_0}{b_\epsilon^+-b_\epsilon^-},1+\frac{a_p-b_\epsilon^+}{b_\epsilon^+-b_\epsilon^-}\Big)\times[0,1]$).
 This rectangle is mapped to $\bcRlr$ (to $\cR$) by the composition
 $\kappa^-_*\Big(\kappa^+_*\Big(-;\frac{a_p-b_\epsilon^+}{b_\epsilon^+-b_\epsilon^-}\Big) ;\frac{b_\epsilon^--a_0}{b_\epsilon^+-b_\epsilon^-}\Big)$, and the map $\cHdiamo_1$
 sends $\bcRlr$ to $\bdiamolr$ (respectively, $\cR$ to $\diamo$).
 
\end{proof}
We consider now the external product
\[
 -\times-\colon \Hur(\bdiamolr,\bdel)\times\Ran(\bdiamolr)\to \Hur(\bdiamolr,\bdel)
 \]
 \[
\pa{\mbox{respectively, }-\times-\colon \Hur(\diamo,\del)\times\Ran(\diamo)\to \Hur(\diamo,\del)}
\]
from \cite[Definition 5.7 and Notation 5.9]{Bianchi:Hur2}, and evaluate the second
component at $\bdel\bdiamolr=\set{\zdiamleft,\zdiamright}$, thus obtaining a map
$-\times\bdel\bdiamolr\colon\Hur(\bdiamolr,\bdel)\to\Hur(\bdiamolr,\bdel)_{\bdel\bdiamolr}$
(respectively, $-\times\bdel\bdiamolr\colon\Hur(\diamo,\del)\to\Hur(\diamo,\del)_{\bdel\bdiamolr}$).
 
\begin{defn}
  \label{defn:sigma}
 We denote by $\hmu^{\diamo}_{\bdel\bdiamolr}$ the following composition:
 \[
  \begin{tikzcd}[column sep=2cm]
   \coprod_{p\geq 0} \Delta^p\times(\mHurm)^p \ar[r,"\hmu^\diamo"] & \Hur(\bdiamolr,\bdel) \ar[r,"-\times\bdel\bdiamolr"]
   & \Hur(\bdiamolr,\bdel)_{\bdel\bdiamolr}
  \end{tikzcd}
 \]
 \[
  \begin{tikzcd}[column sep=2cm]
   \Big(\mbox{respectively, }\coprod_{p\geq 0} \Delta^p\times(\bHurm)^p \ar[r,"\hmu^\diamo"] & \Hur(\diamo,\del) \ar[r,"-\times\bdel\bdiamolr"]
   & \Hur(\diamo,\del)_{\bdel\bdiamolr}\Big).
  \end{tikzcd}
 \]
  We denote by $\check\sigma$ the composition of $\hmu^{\diamo}_{\bdel\bdiamolr}$ with the covering projection
  \[
  \pr_{G,G^{\op}}\colon \Hur(\bdiamolr,\bdel)_{\bdel\bdiamolr}\to \Hur(\bdiamolr,\bdel)_{G,G^{\op}}
  \]
  \[
 \pa{\mbox{respectively, }\pr_{G,G^{\op}}\colon \Hur(\diamo,\del)_{\bdel\bdiamolr}\to \Hur(\diamo,\del)_{G,G^{\op}}}.
 \]
  Here we regard $(\bdiamolr,\bdel)$ (respectively, $(\diamo,\del)$) as an lr-based
  nice couple, using the two points $\zdiamleft$ and $\zdiamright$ of $\bdel\bdiamolr$.
  See Figure \ref{fig:sigma}.
 \end{defn}
 Roughly speaking, $\hmu^{\diamo}_{\bdel\bdiamolr}$ improves the effect of $\hmu^\diamo$ by forcing the presence of $\zdiamleft$ and $\zdiamright$ in the support of the configuration $\hmu^\diamo(\uw;\ut,\ufc)$, which is already supported in $\bdiamolr$ (in $\diamo$); if either point $\zdiamleft,\zdiamright$ is already in the support of $\hmu^\diamo(\uw;\ut,\ufc)$, then its local monodromy does not change when passing to
 $\hmu^{\diamo}_{\bdel\bdiamolr}(\uw;\ut,\ufc)$. The further composition $\check\sigma$ forgets the monodromy information
 around the two points $\zdiamleft$ and $\zdiamright$ of the support of $\hmu^{\diamo}_{\bdel\bdiamolr}(\uw;\ut,\ufc)$.
 
 \begin{figure}[ht]
 \begin{tikzpicture}[scale=4,decoration={markings,mark=at position 0.38 with {\arrow{>}}}]
  \draw[dashed,->] (-.05,0) to (1.6,0);
  \draw[dashed,->] (0,-1.1) to (0,1.1);
  \node at (0,-1) {$*$};
  \draw[dotted, fill=gray, opacity=.5] (0,0) rectangle (1.5,1);
  \draw[very thick] (0,0) to (1.5,0);
  \draw[very thick] (0,1) to (1.5,1);
  \draw[dotted, thick] (0,0) node[below]{\tiny$a_0$} to ++(0,1);
  \draw[dotted, thick] (.5,0) node[below]{\tiny$a_1$} to ++(0,1);
  \draw[dotted, thick] (1,0) node[below]{\tiny$a_2$} to ++(0,1);
  \draw[dotted, thick] (1.5,0) node[below]{\tiny$a_3$} to ++(0,1);
  \draw[dotted] (.33,0) node[below]{\tiny$b_\epsilon^-$} to ++(0,1);
  \draw[dotted] (1.3,0) node[below]{\tiny$b_\epsilon^+$} to ++(0,1);
  \draw[dotted] (.9,0) node[below]{\tiny$b$} to ++(0,1);
  \node at (.1,.3){$\bullet$}; 
  \node at (.3,1){$\bullet$}; 
  \node at (.45,.5){$\bullet$};
  \draw[thin, looseness=1.2, postaction={decorate}] (0,-1) to[out=82,in=-90]  (.05,.1) to[out=90,in=-90] (.01,.3)  to[out=90,in=90] node[above]{\tiny$q_{1,1}$} (.2,.3)  to[out=-90,in=90] (.1,.1) to[out=-90,in=80] (0,-1);
  \draw[thin, looseness=1.2, postaction={decorate}] (0,-1) to[out=79,in=-90] (.28,.9)
  to[out=90,in=-90] (.26,1) to[out=90,in=90] (.34,1) to[out=-90,in=90] node[left]{\tiny$g_{1,2}$}
  (.3,.9) to[out=-90,in=77]  (0,-1);
  \draw[thin, looseness=1.2, postaction={decorate}] (0,-1) to[out=76,in=-90] (.4,.3)
  to[out=90,in=-90] (.35,.5) to[out=90,in=90] node[above]{\tiny$q_{1,3}$} (.49,.5) to[out=-90,in=90] (.45,.3) to[out=-90,in=74]  (0,-1);
  \node at (.65,.0){$\bullet$};
  \node at (.95,.8){$\bullet$};
  \draw[thin, looseness=1.2, postaction={decorate}] (0,-1) to[out=62,in=-120]  (.55,-.1) to[out=60,in=-90] (.52,.05)  to[out=90,in=90] node[above]{\tiny$g_{2,1}$} (.7,.05)  to[out=-90,in=50] (.65,-.1) to[out=-130,in=60] (0,-1);
  \draw[thin, looseness=1.2, postaction={decorate}] (0,-1) to[out=56,in=-90] (.9,.7)
  to[out=90,in=-90] (.85,.8) to[out=90,in=90] node[above]{\tiny$q_{2,2}$} (.99,.8) to[out=-90,in=90] (.95,.7) to[out=-90,in=54]  (0,-1);
  \node at (1.2,.5){$\bullet$};
  \node at (1.45,.2){$\bullet$};
  \draw[thin, looseness=1.2, postaction={decorate}] (0,-1) to[out=50,in=-90]  (1.15,.4) to[out=90,in=-90] (1.12,.55)  to[out=90,in=90] node[above]{\tiny$q_{3,1}$} (1.3,.55)  to[out=-90,in=90] (1.2,.4) to[out=-90,in=44] (0,-1);
  \draw[thin, looseness=1.2, postaction={decorate}] (0,-1) to[out=46,in=-90] (1.4,.1)
  to[out=90,in=-90] (1.35,.2) to[out=90,in=90] node[above]{\tiny$q_{3,2}$} (1.49,.2) to[out=-90,in=90] (1.45,.1) to[out=-90,in=44]  (0,-1);
  \begin{scope}[shift={(2,0)}]
  \draw[dashed,->] (-.1,0) to (1.1,0);
  \draw[dashed,->] (0,-1.1) to (0,1.1);
  \node at (0,-1) {$*$};
  \draw[very thick] (0,.5) to (.5,1) to (1,.5) to (.5,0) to (0,.5);
  \fill[gray, opacity=.5] (0,.5) to (.5,1) to (1,.5) to (.5,0) to (0,.5);
  \node at (0,.5) {$\bullet$};
  \node at (1,.5) {$\bullet$};
  \draw[thin, looseness=1.2, postaction={decorate}] (0,-1) to[out=46,in=-90]  (.97,.4) to[out=90,in=-90] (.95,.55)  to[out=90,in=90] node[above]{\tiny ?} (1.03,.55)  to[out=-90,in=90] (1,.4) to[out=-90,in=44] (0,-1);
  \draw[thin, looseness=1.2, postaction={decorate}] (0,-1) to[out=91,in=-90]  (-.03,.4) to[out=90,in=-90] (-.07,.55)  to[out=90,in=90] node[above]{\tiny ?} (.03,.55)  to[out=-90,in=90] 
  (.01,.4) to[out=-90,in=89] (0,-1);

  \draw[dotted, thick] (.2,0)  to ++(0,1);
  \draw[dotted, thick] (.7,0)  to ++(0,1);
  \draw[dotted] (.6,0) to ++(0,1);
  \node at (.15,.5){$\bullet$};
  \draw[thin, looseness=1.2, postaction={decorate}] (0,-1) to[out=76,in=-90] (.12,.3)
  to[out=90,in=-90] (.1,.5) to[out=90,in=90] node[above right]{\tiny$q_{1,3}$} (.24,.5) to[out=-90,in=90] (.15,.3) to[out=-90,in=74]  (0,-1);
  \node at (.35,.15){$\bullet$};
  \node at (.65,.72){$\bullet$};
  \draw[thin, looseness=1.2, postaction={decorate}] (0,-1) to[out=62,in=-90]  (.3,.05) to[out=90,in=-90] (.22,.2)  to[out=90,in=90] node[above]{\tiny$g_{2,1}$} (.4,.2)  to[out=-90,in=90] (.35,.05) to[out=-90,in=60] (0,-1);
  \draw[thin, looseness=1.2, postaction={decorate}] (0,-1) to[out=56,in=-90] (.6,.62)
  to[out=90,in=-90] (.55,.72) to[out=90,in=90] node[above left]{\tiny$q_{2,2}$} (.69,.72) to[out=-90,in=90] (.65,.62) to[out=-90,in=54]  (0,-1);
  \node at (.9,.5){$\bullet$};
  \draw[thin, looseness=1.2, postaction={decorate}] (0,-1) to[out=50,in=-90]  (.87,.4) to[out=90,in=-90] (.85,.55)  to[out=90,in=90] node[left]{\tiny$q_{3,1}$} (.93,.55)  to[out=-90,in=90] (.9,.4) to[out=-90,in=48] (0,-1);
  \end{scope}
 \end{tikzpicture}
 \caption{Left: the product $(t_1,\fc_1)(t_2,\fc_2)(t_3,\fc_3)$ of three configurations in $\bHurm$; for a given $\uw\in\mDelta^3$ the barycentres $b,b_\epsilon^+,b_\epsilon^-$ are shown as dotted, vertical lines. We use the letter ``$q$'' to represent the $\Q$-valued monodromies around points in $\mcR_{\R}$. Right: the image of $(\uw;\ut,\ufc)$ along $\check\sigma$; the question marks suggest that, because of the quotient by the $G\times G^{\op}$-action, the monodromies of the loops spinning around 
 $\zdiamleft$ and $\zdiamright$ are just not defined.
 }
\label{fig:sigma}
\end{figure}
 
\begin{nota}
\label{nota:fclr}
We endow $B\mHurm$ and $\reB\mHurm$ (respectively, $B\bHurm$ and $\reB\bHurm$) with the basepoint corresponding to the (unique) $0$-simplex in $B_{\bullet}(\mHurm,*)$ (in $B_{\bullet}(\bHurm,*)$). Similarly
$\Hur(\bdiamolr,\bdel)_{G,G^{\op}}$ (respectively, $\Hur(\diamo,\del)_{G,G^{\op}}$)
is endowed with the basepoint given by
$\pr_{G,G^{\op}}((\emptyset,\one,\one)\times\bdel\bdiamolr)$.
We denote this basepoint by $\fclr\in\Hur(\bdiamolr,\bdel)_{G,G^{\op}}$ (respectively,
$\fclr\in\Hur(\diamo,\del)_{G,G^{\op}}$).
\end{nota}

\begin{prop}
 \label{prop:sigmaquotient}
 The map $\check\sigma$ from Definition \ref{defn:sigma} sends every sequence $(\uw;\ut,\ufc)$ satisfying $\epsilon(\uw;\ut,\ufc)>0$ to the basepoint $\fclr$.
 Moreover $\check\sigma$ descends to a pointed map
 $\sigma\colon B\mHurm\to\Hur(\bdiamolr,\bdel)_{G,G^{\op}}$
 (respectively, $\sigma\colon B\bHurm\to\Hur(\diamo,\del)_{G,G^{\op}}$); the map $\sigma$ descends further to a pointed map $\bar\sigma\colon\reB\mHurm\to\Hur(\bdiamolr,\bdel)_{G,G^{\op}}$ (respectively,
 $\bar\sigma\colon\reB\bHurm\to\Hur(\diamo,\del)_{G,G^{\op}}$).
\end{prop}
The proof of Proposition \ref{prop:sigmaquotient} is in Appendix \ref{subsec:sigmaquotient}. Since the quotient map $B\mHurm\to\reB\mHurm$ (respectively, $B\bHurm\to \reB\bHurm$)
is a weak equivalence, Theorem \ref{thm:delooping} reduces to proving that the map $\bar\sigma$ is a weak equivalence. In fact, we will prove surjectivity of $\sigma_*$ and injectivity of $\bar\sigma_*$ on homotopy groups.

\subsection{Surjectivity on homotopy groups}
\label{subsec:sigmasurjective}
We fix $q\geq 0$ and want to show that $\sigma_*\colon \pi_q( B\mHurm)\to\pi_q(\Hur(\bdiamolr,\bdel)_{G,G^{\op}})$ (respectively,
$\sigma_*\colon \pi_q(B\bHurm)\to\pi_q(\Hur(\diamo,\del)_{G,G^{\op}})$) is surjective. 
For $q=0$ this will imply, in particular, that $\Hur(\bdiamolr,\bdel)_{G,G^{\op}}$
(respectively, $\Hur(\diamo,\del)_{G,G^{\op}}$) is connected.

\begin{nota}
 \label{nota:Hurempty}
 We denote by $\mHurm_\emptyset\subset\mHurm$ (respectively, $\bHurm_\emptyset\subset\bHurm$) the component
 of the neutral element: it contains all couples $(t,(\emptyset,\one,\one))$ for $t\ge0$.
\end{nota}
Note that $\mHurm_\emptyset$ (respectively, $\bHurm_\emptyset$) is a contractible topological monoid, hence also
the subspace $B\mHurm_\emptyset\subset B\mHurm$ (respectively, $B\bHurm_\emptyset\subset B\bHurm$) is also contractible.
Moreover, the map $\sigma$ sends $B\mHurm_\emptyset$ (respectively, $B\bHurm_\emptyset$) constantly to the basepoint $\fclr$. 

In order to prove that $\sigma$ induces an isomorphism on $q$-homotopy groups, it suffices to prove that
the map $\sigma_*\colon \pi_q(B\mHurm,B\mHurm_\emptyset)\to \pi_q(\Hur(\bdiamolr,\bdel)_{G,G^{\op}})$
(respectively, $\sigma_*\colon \pi_q(B\bHurm,B\bHurm_\emptyset)\to \pi_q(\Hur(\diamo,\del)_{G,G^{\op}})$) is an isomorphism, i.e. we can consider relative homotopy groups.

To show that $\sigma_*$ is surjective, represent an element of
$\pi_q(\Hur(\bdiamolr,\bdel)_{G,G^{\op}})$ (of $\pi_q(\Hur(\diamo,\del)_{G,G^{\op}})$)
by a map
\[
f\colon D^q\to\Hur(\bdiamolr,\bdel)_{G,G^{\op}}\quad\quad
\Big(\mbox{respectively, }f\colon D^q\to\Hur(\diamo,\del)_{G,G^{\op}}\Big)
\]
sending $\del D^q$ to the basepoint $\fclr$.
Thus for all $v\in D^q$ we have an orbit $f(v)=[\fc_v]_{G,G^{\op}}$ of the
action of $G\times G^{\op}$ on $\Hur(\bdiamolr,\bdel)_{\bdel\bdiamolr}$
(on $\Hur(\diamo,\del)_{\bdel\bdiamolr}$). We choose for all $v\in D^q$ a representative
$\fc_v=(P_v,\psi_v,\phi_v)$ of $f(v)$; note that the sets $P_v\subset\bdiamolr$ (respectively, $P_v\subset\diamo$)
do not depend on this choice; similarly the evaluation of $\psi_v$ and $\phi_v$ is independent of the choice of $\fc_v$
on those elements of $\fQ(P_v)$ and $\fG(P_v)$ that can be represented by a loop contained in $[0,1]\times\R\,\setminus P_v$.

For all $v\in D^q$ we can find an open
interval $J_v\subset(0,1)$ and a
neighbourhood $v\in V_v\subseteq D^q$ such that for all $v'\in V_v$ the finite set $\Re(P_{v'})$ is disjoint from $J_v$.
Using the compactness of $D^q$ we can then choose a cover of $D^q$ by finitely many open sets $V_i$ with corresponding
open intervals $J_i\subset(0,1)$, satisfying $\Re(P_v)\cap J_i=\emptyset$ for all $v\in V_i$. After shrinking the intervals $J_i$
appropriately, we can assume they are disjoint. To fix notation, we assume that we have $r$ open sets $V_1,\dots, V_r$,
such that the corresponding intervals $J_1,\dots,J_r$
appear in this order, from left to right, on $(0,1)$.

We choose numbers $A_i\in J_i$: note that $\bS_{A_i,A_i}$ is disjoint from $P_v$ for all $v\in V_i$,
and that the numbers $A_1,\dots,A_r$ are all distinct.
We fix weights $W_i\colon D^q\to[0,1]$ giving a partition of unity on $D^q$
subordinate to the covering $\set{V_1,\dots,V_r}$.
We also assume that for each $v\in D^q$ there are at least \emph{two} distinct indices $1\le i\le r$ such that $W_i(v)>0$.

\begin{nota}
 \label{nota:diamottprime}
 Let $0< t<t'< 1$. We denote by
 $\bdiamolr_{t,t'}$ (respectively, $(\diamo_{t,t'},\del)$) the space
 $\bdiamolr\cap (t,t')\times\R$ (the nice couple $(\diamo\cap (t,t')\times\R,\del\diamo\cap(t,t')\times\R)$).

 Recall Notation \ref{nota:cRt}:
 we denote by $\mcR_{t,t'}$ (by $(\bcR_{t,t'},\bdel)$)
 the space $(t,t')\times(0,1)$ (the nice couple $((t,t')\times[0,1],(t,t')\times\set{0,1})$).
\end{nota}

\begin{lem}
\label{lem:vstriphomeo}
 For all $0<t<t'<1$ the map $\cHdiamo_1$ induces a homeomorphism
 $\Hur(\mcR_{t,t'})\cong\Hur(\bdiamolr_{t,t'})$
 (respectively, $\Hur(\bcR_{t,t'},\bdel)\cong\Hur(\diamo_{t,t'},\del)$).
\end{lem}
\begin{proof}
 Note that $\cHdiamo_1$ restricts to a semialgebraic homeomorphism of the subspace
 \[
 \bT:=\big((0,1)\times\R\big)\cup\set{*}\subset\C.
 \]
 The space $\bT$ is contractible, and
 the interior $\mathring{\bT}$ contains
 the spaces $\mcR_{t,t'}$
 and $\bdiamolr_{t,t'}$ (the spaces $\bcR_{t,t'}$ and $\diamo_{t,t'}$).
 Moreover the space $\mcR_{t,t'}$ (the nice couple $(\bcR_{t,t'},\bdel)$) is mapped along $\cHdiamo_1$ homeomorphically
 to the space $\bdiamolr_{t,t'}$ (to the nice couple  $(\diamo_{t,t'},\del)$).
It follows that $\cHdiamo_1$ induces a homeomorphism
$\Hur^{\bT}(\mcR_{t,t'})\cong\Hur^{\bT}(\bdiamolr_{t,t'})$ (respectively,
$\Hur^{\bT}(\bcR_{t,t'},\bdel)\cong\Hur^{\bT}(\diamo_{t,t'},\del)$),
and the statement is a consequence of the natural homeomorphism $\fri^{\C}_{\bT}\colon\Hur(\fT)\cong\Hur^{\bT}(\fT)$
holding for all nice couples $\fT$ contained in the interior of $\bT$.
 \end{proof}

\begin{nota}
 \label{nota:Apv}
For all $v\in D^q$ we list the indices $1\leq i_0<\dots<i_{p_v}\leq r$
satisfying $W_{i_j}(v)>0$, for some $p_v\ge1$ depending on $v$: here, recall
our assumption that for each $v\in D^q$ there are at least two indices $i$ with $W_i(v)>0$.
We denote by $A_0^v,\dots, A_{p_v}^v$ the list $A_{i_0},\dots,A_{i_{p_v}}$.
We set $B^v=\sum_{j=0}^{p_v} W_{i_j}(v)A_{i_j}^v$. For all $0\le j\le p_v$ we set
$A_i^{v,+}=\max\set{A_i^v,B^v}$ and $A_i^{v,-}=\min\set{A_i^v,B^v}$. We set
$B^{v,+}=\sum_{j=0}^{p_v} W_{i_j}(v)A_{i_j}^{v,+}$ and
$B^{v,-}=\sum_{j=0}^{p_v} W_{i_j}(v)A_{i_j}^{v,-}$.
\end{nota}
Note that the numbers $B^v$, $B^{v,+}$ and $B^{v,-}$ vary continously in $v\in\ D^q$,
and attain values in $(0,1)$. Note also that we have a sequence of strict inequalities
$0<A_0^v<B^{v,-}<B^v<B^{v,+}<A_{p_v}^v<1$.
In the following, for all $1\leq j\leq p_v$ we
construct a configuration $\fc_{v,j}=(P_{v,j},\psi_{v,j},\phi_{v,j})$ in $\Hur(\bdiamolr_{A_{j-1}^v,A_j^v},\bdel)$
(in $\Hur(\diamo_{A_{j-1}^v,A_j^v},\del)$): roughly speaking, $\fc_{v,j}$
will be the part of $\fc_v$ contained in the vertical strip $(A_{j-1}^v,A_j^v)\times\R$. See Figure \ref{fig:cvj}.

 \begin{figure}[ht]
 \begin{tikzpicture}[scale=4,decoration={markings,mark=at position 0.38 with {\arrow{>}}}]
  \draw[dashed,->] (-.1,0) to (1.1,0);
  \draw[dashed,->] (0,-1.1) to (0,1.1);
  \node at (0,-1) {$*$};
  \draw[very thick] (0,.5) to (.5,1) to (1,.5) to (.5,0) to (0,.5);
  \fill[gray, opacity=.5] (0,.5) to (.5,1) to (1,.5) to (.5,0) to (0,.5);
  \node at (0,.5) {$\bullet$};
  \node at (1,.5) {$\bullet$};
  \draw[thin, looseness=1.2, postaction={decorate}] (0,-1) to[out=46,in=-90]  (.97,.4) to[out=90,in=-90] (.95,.55)  to[out=90,in=90] node[above right]{\tiny $g^{\mathrm{r}}_{v,\diamo}$} (1.03,.55)  to[out=-90,in=90] (1,.4) to[out=-90,in=44] (0,-1);
  \draw[thin, looseness=1.2, postaction={decorate}] (0,-1) to[out=91,in=-90]  (-.03,.4) to[out=90,in=-90] (-.07,.55)  to[out=90,in=90] node[above left]{\tiny $g^{\mathrm{l}}_{v,\diamo}$} (.03,.55)  to[out=-90,in=90] 
  (.01,.4) to[out=-90,in=89] (0,-1);

  \draw[dotted, thick] (.2,0) node[below]{\tiny$A^v_0$} to ++(0,1);
  \draw[dotted, thick] (.7,0) node[below]{\tiny$A^v_1$} to ++(0,1);
  \node at (.15,.5){$\bullet$};
  \draw[thin, looseness=1.2, postaction={decorate}] (0,-1) to[out=76,in=-90] (.12,.3)
  to[out=90,in=-90] (.1,.5) to[out=90,in=90] node[above right]{\tiny$q_{v,1}$} (.24,.5) to[out=-90,in=90] (.15,.3) to[out=-90,in=74]  (0,-1);
  \node at (.35,.15){$\bullet$};
  \node at (.65,.72){$\bullet$};
  \draw[thin, looseness=1.2, postaction={decorate}] (0,-1) to[out=62,in=-90]  (.3,.05) to[out=90,in=-90] (.22,.2)  to[out=90,in=90] node[above]{\tiny$g_{v,2}$} (.4,.2)  to[out=-90,in=90] (.35,.05) to[out=-90,in=60] (0,-1);
  \draw[thin, looseness=1.2, postaction={decorate}] (0,-1) to[out=56,in=-90] (.6,.62)
  to[out=90,in=-90] (.55,.72) to[out=90,in=90] node[above left]{\tiny$q_{v,3}$} (.69,.72) to[out=-90,in=90] (.65,.62) to[out=-90,in=54]  (0,-1);
  \node at (.9,.5){$\bullet$};
  \draw[thin, looseness=1.2, postaction={decorate}] (0,-1) to[out=50,in=-90]  (.87,.4) to[out=90,in=-90] (.85,.55)  to[out=90,in=90] node[left]{\tiny$q_{v,4}$} (.93,.55)  to[out=-90,in=90] (.9,.4) to[out=-90,in=48] (0,-1);
  \begin{scope}[shift={(1.5,0)}]
  \draw[dashed,->] (-.1,0) to (1.1,0);
  \draw[dashed,->] (0,-1.1) to (0,1.1);
  \node at (0,-1) {$*$};
  \begin{scope}
   \clip (.2,-.1) rectangle (.7,1.1);
  \draw[very thick] (0,.5) to (.5,1) to (1,.5) to (.5,0) to (0,.5);
  \fill[gray, opacity=.5] (0,.5) to (.5,1) to (1,.5) to (.5,0) to (0,.5);
  \end{scope}

  \draw[dotted, thick] (.2,0) node[below]{\tiny$A^v_0$} to ++(0,1);
  \draw[dotted, thick] (.7,0) node[below]{\tiny$A^v_1$} to ++(0,1);
  \node at (.65,.72){$\bullet$};
  \node at (.35,.15){$\bullet$};
  \draw[thin, looseness=1.2, postaction={decorate}] (0,-1) to[out=62,in=-90]  (.3,.05) to[out=90,in=-90] (.22,.2)  to[out=90,in=90] node[above]{\tiny$g_{v,2}$} (.4,.2)  to[out=-90,in=90] (.35,.05) to[out=-90,in=60] (0,-1);
  \draw[thin, looseness=1.2, postaction={decorate}] (0,-1) to[out=56,in=-90] (.6,.62)
  to[out=90,in=-90] (.55,.72) to[out=90,in=90] node[above left]{\tiny$q_{v,3}$} (.69,.72) to[out=-90,in=90] (.65,.62) to[out=-90,in=54]  (0,-1);
  \end{scope}
 \end{tikzpicture}
 \caption{Left: a configuration $\fc_v\in\Hur(\diamo,\del)_{\bdel\bdiamolr}$, in the case $p_v=1$. Right: the clipped configuration $\fc_{v,1}\in\Hur(\diamo_{A_0^v,A_1^v},\del)$.
 }
\label{fig:cvj}
\end{figure}

To define $\fc_{v,j}$, note that $P_v$ is disjoint from the vertical lines $\set{A_{j-1}^v,A_j^v}\times\R$; as a consequence $P_v$ is contained in the disjoint union $\bdiamolr_{A^v_{j-1},A^v_j} \sqcup (\bdiamolr\setminus\bS_{A^v_{j-1},A^v_j})$ (respectively,
$\diamo_{A^v_{j-1},A^v_j} \sqcup (\diamo\setminus\bS_{A^v_{j-1},A^v_j})$);
note that the first space $\bdiamolr_{A^v_{j-1},A^v_j}$ (respectively, $\diamo_{A^v_{j-1},A^v_j}$)
is contained in the interior of $\bS_{A^v_{j-1},A^v_j}$, whereas the second space
$\bdiamolr\setminus\bS_{A^v_{j-1},A^v_j}$
(respectively, $\diamo\setminus\bS_{A^v_{j-1},A^v_j}$) is contained
in the interior of $\C\setminus\bS_{A^v_{j-1},A^v_j}$.
We use the restriction map
\[
 \fri_{\bS_{A^v_{j-1},A^v_j}}^{\C}\colon \Hur\pa{\bdiamolr_{A^v_{j-1},A^v_j} \sqcup \pa{\bdiamolr\setminus\bS_{A^v_{j-1},A^v_j}},\bdel\bdiamolr}
 \to \Hur^{\bS_{A^v_{j-1},A^v_j}}\pa{\bdiamolr_{A^v_{j-1},A^v_j}}
\]
\[
 \Big(\mbox{respectively, }\fri_{\bS_{A^v_{j-1},A^v_j}}^{\C}\colon\!\Hur\!\pa{\diamo_{A^v_{j-1},A^v_j} \sqcup \pa{\diamo\setminus\bS_{A^v_{j-1},A^v_j}},\del\!}
 \to \Hur^{\bS_{A^v_{j-1},A^v_j}}\pa{\diamo_{A^v_{j-1},A^v_j},\del}\Big)
\]
from \cite[Definition 3.15]{Bianchi:Hur2} and define $\fc_{v,j}$ as the image of $\fc_v$ along this map. We can then use the canonical identification
\[
 \fri^{\C}_{\bS_{A^v_{j-1},A^v_j}}\colon \Hur\pa{\bdiamolr_{A^v_{j-1},A^v_j}}\cong \Hur^{\bS_{A^v_{j-1},A^v_j}}\pa{\bdiamolr_{A^v_{j-1},A^v_j}}
\]
\[
 \Big(\mbox{respectively, }\fri_{\bS_{A^v_{j-1},A^v_j}}^{\C} \colon \Hur\pa{\diamo_{A^v_{j-1},A^v_j},\del}\to\Hur^{\bS_{A^v_{j-1},A^v_j}}\pa{\diamo_{A^v_{j-1},A^v_j},\del}\Big)
\]
to regard $\fc_{v,j}$ as a configuration in $\Hur(\bdiamolr_{A^v_{j-1},A^v_j})$
(in $\Hur(\diamo_{A^v_{j-1},A^v_j},\del)$).

Note that $\fc_{v,j}$ does not depend on the choice of a representative
$\fc_v$ of $f(v)=[\fc_v]_{G,G^{\op}}$.
Let $\bT$ be as in the proof of Lemma \ref{lem:vstriphomeo};
then we can regard $\fc_{v,j}$ as a configuration in $\Hur^{\bT}(\bdiamolr_{A^v_{j-1},A^v_j})$
(in $\Hur^{\bT}(\diamo_{A^v_{j-1},A^v_j},\del)$) using the canonical identification
$\fri^{\bT}_{\bS_{A^v_{j-1},A^v_j}}$.

By Lemma \ref{lem:vstriphomeo}
we can consider for all $1\le j\le p_v$ the configuration
$(\cHdiamo_1)_*^{-1}(\fc_{v,j})$ lying in the space
$\Hur^{\bT}(\mcR_{A_{j-1}^v,A_j^v})$
(in $\Hur^{\bT}(\bcR_{A_{j-1}^v,A_j^v},\bdel)$);
using the identification $\fri^{\C}_{\bT}$
we can regard $(\cHdiamo_1)_*^{-1}(\fc_{v,j})$ as lying in
$\Hur(\mcR_{A_{j-1}^v,A_j^v})$ (in $\Hur(\bcR_{A_{j-1}^v,A_j^v},\bdel)$).
Composing further with the map
$\tau_{-A_{j-1}^v}$ we obtain configurations
\[
\fc'_{v,j}:=(\tau_{-A_{j-1}^v})_*\pa{(\cHdiamo_1)_*^{-1}(\fc_{v,j})}
\]
lying in $\Hur(\mcR_{A_j^v-A_{j-1}^v})$
(in $\Hur(\bcR_{A_j^v-A_{j-1}^v},\bdel)$), for $1\leq j\leq p_v$.

In other words we obtain couples
$(A_1^v-A_0^v,\fc'_{v,1}),\dots,(A_{p_v}^v-A_{p_v-1}^v,\fc'_{v,p_v})$ in $\mHurm$ (in $\bHurm$);
the product $(A_1^v-A_0^v,\fc'_{v,1})\dots(A_{p_v}^v-A_{p_v-1}^v,\fc'_{v,p_v})$ of these configurations has the form $(A_{p_v}^v-A_0^v,\fc')$,
for some $\fc'\in\Hur(\mcR_{A_{p_v}^v-A_0^v})$ (respectively, $\fc'\in\Hur(\bcR_{A_{p_v}^v-A_0^v},\bdel)$): we then have, roughly speaking, that
$\cHdiamo_1(\tau_{A_0^v}(\fc'))$ recovers the part of $\fc_v$ in the vertical
strip $(A_0^v,A_{p_v}^v)\times\R$.

We can define a map $g\colon D^q\to B\mHurm$ (respectively, $g\colon D^q\to B\bHurm$) by the formula
\[
 v\mapsto \Big(W_{i_0}(v),\dots,W_{i_{p_v}}(v);(A_1^v-A_0^v,\fc'_{v,1}),\dots,(A_{p_v}^v-A_{p_v-1}^v,\fc'_{v,p_v})\Big).
\]
To see that $g$ is continuous, note that if a weight $W_{i_j}(v)$ goes to 0,
then the number $A^v_j$ is dropped from the list $A_0^v,\dots,A_{p_v}^v$ and the following happens:
\begin{itemize}
 \item if $1\leq j\leq p_v-1$, then
the configurations $(A_j^v-A_{j-1}^v,\fc'_{v,j})$ and $(A_{j+1}^v-A_j^v,\fc'_{v,j+1})$ are replaced in the formula above by their product in $\mHurm$ (in $\bHurm$), according to the identifications defining $B\mHurm$ (respectively, $B\bHurm$); this is compatible with the fact that the configurations $\fc_{v,j}$ and $\fc_{v,j+1}$
are ``adjacent'' in $\bdiamolr$ (in $\diamo$), and if $A^v_j$ is dropped these two configurations are replaced in the construction by their ``concatenation''
in $\Hur(\bdiamolr_{A^v_{j-1},A^v_{j+1}})$
(respectively, $\Hur(\diamo{A^v_{j-1},A^v_{j+1}},\del)$), which is up to canonical identifications the configuration $\fri^{\C}_{\bS_{A^v_{j-1},A^v_{j+1}}}(\fc_v)$;
\item if $j=0$ or $j=p_v$, then the configuration $(A_1^v-A_0^v,\fc'_{v,1})$ or $(A_{p_v}^v-A_{p_v-1}^v,\fc'_{v,p_v})$
is dropped in the formula above, according to the identifications defining
of $B\mHurm$ (respectively, $B\bHurm$); this is compatible with the fact that
the configuration $\fc_{v,0}$ or $\fc_{v,p_v}$ is also dropped in the construction, as soon as $A^v_0$ or $A^v_{p_v}$ is dropped.
\end{itemize}
Note also that $g$ sends $\del D^q$ inside $B\mHurm_\emptyset$ (inside $B\bHurm_\emptyset$): in fact,
if $f(v)=\fclr$, then all configurations $\fc'_{v,j}$ are supported on the empty set.
For $v\in D^q$ we remark, moreover, the equalities $B^{v,-}-A_0^v=b^-(g(v))$ and $B^{v,+}-A_0^v=b^+(g(v))$, by virtue of which the lower and upper barycentres of $g(v)$ can be recovered from the numbers $B^{v,-}$ and $B^{v,+}$ and vice versa, once the ``translation parameter'' $A_0^v$ is known; in particular,
$B^{v,+}-B^{v,-}=b^+(g(v))-b^-(g(v))$.

We are left to prove that $\sigma g$ is homotopic to $f$, relative to $\del D^q$, i.e. they represent the same element in $\pi_q(\Hur(\bdiamolr,\bdel)_{G,G^{\op}})$
(in $\pi_q(\Hur(\diamo,\del)_{G,G^{\op}})$).

\begin{defn}
 \label{defn:cHdiamoleftright}
Recall Definitions \ref{defn:cHtleftright} and \ref{defn:cHdiamo}.
For all $0<t<1$ we define homotopies $\cHdiamoleft_t,\cHdiamoright_t\colon\C\times[0,1]\to\C$ by the following formula.
 \[ \def\arraystretch{2}
 \cH^{\diamo,\bullet}_t(z,s)=\left\{
 \begin{array}{ll}
  \cH^{\bullet}_t(z,s)+\big(\fd^{\diamo}(z)-\fd^{\diamo}(\cH^{\bullet}_t(z,s))\big)\sqrt{-1}
  & \mbox{if } z\notin\mathring{\diamo},\ \Im(z)\geq \frac 12\\
  \cH^{\bullet}_t(z,s)-\big(\fd^{\diamo}(z)-\fd^{\diamo}(\cH^{\bullet}_t(z,s))\big)\sqrt{-1}
  & \mbox{if } z\notin\mathring{\diamo},\ \Im(z)\leq \frac 12\\
  \cH^{\bullet}_t(z,s)-\pa{\Im(z)-\frac 12}\frac{\fd^{\diamo}(\cH^{\bullet}_t(z,s))-\fd^{\diamo}(z)}{\frac 12-\fd^{\diamo}(z)}\sqrt{-1} 
  & \mbox{if } z\in\diamo,\ z\neq \zdiamleft,\zdiamright.
 \end{array}
 \right.
 \]
Here $\bullet$ is either $\mathrm{l}$ or $\mathrm{r}$.
\end{defn}
Roughly speaking, $\cHdiamoleft_t$ collapses the part of $\diamo$ contained in $[0,t]\times\R$
to $\zdiamleft$, and expands the other part $\diamo\setminus[0,t]\times\R$ inside $\diamo$;
similar remarks hold for $\cHdiamoright_t$.

Note that for $z\in\diamo$ close to $\zdiamleft$ we have the equalities $\frac{\fd^{\diamo}(\cHleft_t(z,s))-\fd^{\diamo}(z)}{\frac 12-\fd^{\diamo}(z)}=s$
and $\frac{\fd^{\diamo}(\cHright_t(z,s))-\fd^{\diamo}(z)}{\frac 12-\fd^{\diamo}(z)}=\frac {-s+st}{1-s+st}$.
Similarly, for $z\in\diamo$ close to $\zdiamright$ we have $\frac{\fd^{\diamo}(\cHleft_t(z,s))-\fd^{\diamo}(z)}{\frac 12-\fd^{\diamo}(z)}=\frac {st}{1-st}$
and $\frac{\fd^{\diamo}(\cHright_t(z,s))-\fd^{\diamo}(z)}{\frac 12-\fd^{\diamo}(z)}=s$.
In particular, the homotopies $\cHdiamoleft_t,\cHdiamoright_t$ are continuous; note also that they depend continuously on $t\in(0,1)$, so that
we can define continuous maps
\[
\cHdiamoleft\,,\,\cHdiamoright\,\colon\C\times[0,1]\times(0,1)\to\C
\]
by $\cHdiamoleft(z,s,t)=\cHdiamoleft_t(z,s)$
and $\cHdiamoright(z,s,t)=\cHdiamoright_t(z,s)$.
Note also that, for $\bullet=\mathrm{l},\mathrm{r}$, the following properties hold:
\begin{enumerate}
 \item $\cH^{\diamo,\bullet}_t(-,s)$ is an endomorphism of the nice couple $(\bdiamolr,\bdel)$
 (respectively, $(\diamo,\del)$) for all $(s,t)\in[0,1]\times(0,1)$;
 \item $\cH^{\diamo,\bullet}_t(-,s)$
 fixes pointwise the subspaces
 $\C_{\Re\leq0}$ and $\C_{\Re\geq 1}$, and preserves the subspaces $\diamo$,
 $[0,1]\times\left(-\infty,\frac 12\right]\setminus\mathring{\diamo}$
 and $[0,1]\times\left[\frac 12, \infty\right)\setminus\mathring{\diamo}$;
 in particular it fixes the points $\zdiamleft,\zdiamright$;
 \item $\cH^{\diamo,\bullet}_t(-,0)=\Id_\C$ for all $t\in(0,1)$.
\end{enumerate}
It follows from property (1) and \cite[Proposition 4.4]{Bianchi:Hur2} that for $\bullet=\mathrm{l},\mathrm{r}$ there
is an induced map
\[
\cH^{\diamo,\bullet}_*\colon\Hur(\bdiamolr,\bdel)\times[0,1]\times(0,1)\to\Hur(\bdiamolr,\bdel)
\]
\[
\pa{\mbox{respectively,}\cH^{\diamo,\bullet}_*\colon\Hur(\diamo,\del)\times[0,1]\times(0,1)\to\Hur(\diamo,\del)\ }.
\]
Property (2) ensures that for all $(s,t)\in[0,1]\times(0,1)$ the map $\cH^{\diamo,\bullet}_*(-,s,t)$
restricts to a self-map of the subspace $\Hur(\bdiamolr,\bdel)_{\bdel\bdiamolr}$
(respectively,  $\Hur(\diamo,\del)_{\bdel\bdiamolr}$), and is equivariant with respect to the $G\times G^{\op}$ action
on this subspace. In particular there is an induced map
\[
\cH^{\diamo,\bullet}_*\colon\Hur(\bdiamolr,\bdel)_{G,G^{\op}}\times[0,1]\times(0,1)\to\Hur(\bdiamolr,\bdel)_{G,G^{\op}}
\]
\[
\pa{\mbox{respectively, }\cH^{\diamo,\bullet}_*\colon\Hur(\diamo,\del)_{G,G^{\op}}\times[0,1]\times(0,1)\to\Hur(\diamo,\del)_{G,G^{\op}}\ }.
\]

Property (3) ensures that $\cH^{\diamo,\bullet}_*(-,0,t)$ is the identity
of $\Hur(\bdiamolr,\bdel)_{G,G^{\op}}$ (of $\Hur(\diamo,\del)_{G,G^{\op}}$)
for all $t\in(0,1)$. We can now define a map
$H\colon D^q \times[0,1] \to \Hur(\bdiamolr,\bdel)_{G,G^{\op}}$ (respectively, $H\colon D^q \times[0,1] \to \Hur(\diamo,\del)_{G,G^{\op}}$)
by setting
\[
 H(v,s)=
 \cHdiamoleft_*\pa{\cHdiamoright_* (f(v),s,B^{v,+}),s,\frac{B^{v,-}}{B^{v,+}}},
\]
where we recall that $0<B^{v,-}<B^{v,+}<1$.
By the construction of $g$ we have $H(-,1)=\sigma g$. This concludes the proof that
$\sigma_*$ is surjective on $\pi_q$.

In the particular case $q=0$ we obtain that, since
$B\mHurm$ and $B\bHurm$ are connected, then $\Hur(\bdiamolr,\bdel)_{G,G^{\op}}$ and $\Hur(\diamo,\del)_{G,G^{\op}}$
are also connected.
The fact that $\Hur(\bdiamolr,\bdel)_{G,G^{\op}}$ is connected can also be proved by combining \cite[Lemma 6.16 and Proposition 7.10]{Bianchi:Hur2}, Theorem \ref{thm:pi0bHurm} and Lemma \ref{lem:bcRvsdiamond}: the space $\Hur(\bdiamolr,\bdel)_{G,G^{\op}}$ is the quotient of the space $\Hur(\bdiamolr,\bdel)_{\bdel\bdiamolr}$ by the action of $G\times G^{\op}$, and on $\pi_0(\Hur(\bdiamolr,\bdel)_{\bdel\bdiamolr})\cong G$ this action can be identified with the action by left and right multiplication, which is transitive. The fact that $\Hur(\diamo,\del)_{G,G^{\op}}$ is connected is instead new in our discussion, although it could have been proved directly using simpler arguments.

\subsection{Injectivity on homotopy groups}
\label{subsec:sigmainjective}
For $q\geq 0$ we want now to prove that $\sigma_*\colon \pi_q(B\mHurm)\to\pi_q(\Hur(\bdiamolr,\bdel)_{G,G^{\op}})$ (respectively,
$\sigma_*\colon \pi_q(B\bHurm)\to\pi_q(\Hur(\diamo,\del)_{G,G^{\op}})$) is injective.
We fix a basepoint $*\in S^q\subset D^{q+1}$ and start with a pointed map
$\tilde f\colon S^q\to B\mHurm$ (respectively, $\tilde f\colon S^q\to B\bHurm$) and a map $f\colon D^{q+1}\to\Hur(\bdiamolr,\bdel)_{G,G^{\op}}$
(respectively, $f\colon D^{q+1}\to\Hur(\diamo,\del)_{G,G^{\op}}$), such that the restriction of $f$ on $S^q=\del D^{q+1}$ is equal to $\sigma \tilde f$.

We can construct a map $g\colon D^{q+1}\to B\mHurm$ (respectively, $g\colon D^{q+1}\to B\bHurm$) in the same way
as we constructed $g\colon D^q\to B\mHurm$ (respectively, $g\colon D^q\to B\bHurm$) in the previous subsection:
using compactness of $D^{q+1}$, we can find a suitable cover $V_1,\dots,V_r$ of $D^{q+1}$ and
disjoint intervals $J_1,\dots,J_r\subset(0,1)$, ordered from left to right, such that
for all $v\in V_i$, if $\fc_v$ is a representative of $f(v)=[\fc_v]_{G,G^{\op}}$, then $\fc_v$ is supported on a set $P_v$ with $\Re(P_v)\cap J_i=\emptyset$.
We also fix $A_i\in J_i$ for all $1\leq i\leq r$, and a partition of unity $W_1,\dots W_r$
subordinate to the covering $V_i$; again we assume that for all $v\in D^{q+1}$ there are at least
two indices $1\le i\le r$ such that $W_i(v)>0$.
The rest of the construction is the same as in the previous subsection; note that,
in general, $g(*)$ is a point in the contractible subspace
$B\mHurm_\emptyset$ (respectively, $B\bHurm_\emptyset$), but $g(*)$ is not necessarily the basepoint
of $B\mHurm$ (of $B\bHurm$).
For our scopes it suffices to prove that $g|_{S^q}$ is homotopic to $\tilde f$
as maps $S^q\to B\mHurm$ (as maps $S^q\to B\bHurm$), through a homotopy sending $*\in S^q$
inside $B\mHurm_\emptyset$ (respectively, $B\bHurm_\emptyset$) at all times. We are thus replacing $\pi_q(B\mHurm)$ (respectively, $\pi_q(B\bHurm)$) with the set of homotopy classes of maps of pairs from $(S^q,*)$ to $(B\mHurm,B\mHurm_\emptyset)$ (respectively, to $(B\bHurm,B\bHurm_\emptyset$).

At this point of the discussion it becomes convenient to switch our focus to the thin bar construction. Recall Definition \ref{defn:reBM}:
the projection $\pr_{\reB}\colon B\mHurm\to \reB\mHurm$ (respectively, $\pr_{\reB}\colon B\bHurm\to \reB\bHurm$) is a weak homotopy equivalence; similarly, note that the subspace $\reB\mHurm_\emptyset\subset\reB\mHurm$ (respectively,
$\reB\bHurm_\emptyset\subset\reB\bHurm$) is contractible.
Therefore, it suffices to
prove that $\pr_{\reB}\circ g|_{S^q}$ and $\pr_{\reB}\circ\tilde f$ are homotopic as maps
$S^q\to \reB\mHurm$ (as maps $S^q\to \reB\bHurm$), by a homotopy sending $*\in S^q$ inside
$\reB\mHurm_\emptyset$ (inside $\reB\bHurm_\emptyset$) at all times.
Exhibiting such homotopy will be much easier than comparing $g|_{S^q}$ and $\tilde f$ directly.

One of the advantages of the thin bar construction occurs
already in the construction of the map $g\colon D^{q+1}\to B\mHurm$
(respectively, $g\colon D^{q+1}\to B\bHurm$).
Recall that we had to shrink the
intervals $J_1,\dots,J_r$ to make them disjoint, in order to ensure that
the numbers $A_1,\dots,A_r$ are all distinct.
This was crucial when defining $g(v)$ for $v\in D^{q+1}$ (or, in the previous
subsection, for $v\in D^q$): we started from the list of indices $1<i_0<\dots<i_{p_v}<l$ satisfying $W_{i_j}(v)>0$; we
denoted by $A_0^v,\dots,\dots A_{p_v}^v$ the list $A_{i_0},\dots,A_{i_{p_v}}$, and used the portions of $f(v)$ contained in
the slices $\bdiamolr_{A_{j-1}^v,A_j^v}$ (in $(\diamo_{A_{j-1}^v,A_j^v},\del)$) to define the configurations $\fc'_{v,1},\dots,\fc'_{v,{p_v}}$. The fact that the numbers $A_1,\dots,A_r$
are all distinct was crucial in ensuring that all slices have strictly positive width and,
most important, the numbers $W_{i_0}(v),\dots,W_{i_{p_v}}(v)$ naturally form an \emph{ordered}
list of $p+1$ numbers: this is crucial, as we want to use these numbers to define the barycentric coordinates of a point in $\Delta^p$.

Suppose instead that we repeat the above construction of $g$, but using directly the thin bar construction; in other words, consider the composition of $g$ with the projection $\pr_{\reB}$: then the fact that $A_1,\dots,A_r$ are all distinct
ceases to be important. Indeed, suppose that for some $v\in D^{q+1}$
we write a list $A_0^v,\dots,\dots A_{p_v}^v$ as above, and suppose that for some
$1\leq j\leq p_v$ we have $A^v_{j-1}=A^v_j$:
this means, in particular, that $\fc'_{v,j}=(\emptyset,\one,\one)$.
On the one hand, we cannot unequivocally determine which of the barycentric coordinates
$W_{i_{j-1}}(v)$ and $W_{i_j}(v)$ should come first in the list of barycentric
coordinates for $g(v)$; on the other hand, the two possibilities
give rise to the same configuration in $\reB\mHurm$ (in $\reB\bHurm$).
For simplicity, in the following we keep assuming that the numbers
$0<A_1<\dots<A_r<1$ are distinct, and we keep considering $g$
and $\tilde f$ as maps with values in $B\mHurm$ (in $B\bHurm$).

Fix $v\in S^q$ and let $\tilde f(v)\in B\mHurm$ (respectively, $\tilde f(v)\in B\bHurm$)
be represented, for some $\tp_v\geq 0$,
by the $\tp_v$-tuple $(t^v_1,\fc^v_1),\dots,(t^v_{\tp_v},\fc^v_{\tp_v})$
of configurations in $\mHurm$ (in $\bHurm$), with barycentric coordinates $w^v_0,\dots,w^v_{\tp_v}$.
Let the numbers $a^v_0,\dots,a^v_{\tp_v}$ and the barycentres $b_v$, $b_{\epsilon,v}^-$ and $b_{\epsilon,v}^+$ be computed
as in Definitions \ref{defn:barycentres} and \ref{defn:epsilon} with
respect to $(t^v_1,\fc^v_1),\dots,(t^v_{\tp_v},\fc^v_{\tp_v})$ and $w^v_0,\dots,w^v_{\tp_v}$.

Using the notation
from Subsection \ref{subsec:sigmasurjective}, for $v\in S^q$
let $1\leq i_0<\dots<i_{p_v}\leq r$ be the list of all indices
$i_j$ satisfying $W_{i_j}(v)>0$; again let $A_0^v,\dots,\dots A_{p_v}^v$ 
denote the corresponding list of numbers $0<A_{i_0}<\dots<A_{i_{p_v}}<1$.
Recall that $g(v)\in B\mHurm$ (respectively, $g(v)\in B\bHurm$) is constructed using
the numbers $W_{i_0}(v),\dots,W_{i_{p'}}(v)$ as barycentric coordinates, and using
the portions of $\fc_v$ contained in the slices $[A_{i_0},A_{i_1}]\times\R,\dots,[A_{i_{p'-1}},A_{i_{p'}}]\times\R$ to obtain configurations $(A_{i_j}-A_{i_{j-1}},\fc'_{v,j})$ in $\mHurm$ (in $\bHurm$).

For $0\leq j\leq p_v$ define $\alpha^v_j=b^-_{\epsilon,v}+(b^+_{\epsilon,v}-b^-_{\epsilon,v})A^v_j$.
Then $b^-_{\epsilon,v}\leq\alpha^v_0\leq\dots\leq\alpha^v_{p'}\leq b^+_{\epsilon,v}$, and
the inequality $b^-_{\epsilon,v}<b^+_{\epsilon,v}$ is strict.
To fix notation, let
\[
\beta^v_0\leq\dots\leq\beta^v_{\tp_v+p_v+1}
\]
be the union of the lists of numbers
$a^v_0,\dots,a^v_{\tp_v}$ and $\alpha^v_0,\dots,\alpha^v_{p_v}$: all numbers $\beta^v_j$,
as well as the numbers $b_v$, $b^-_{\epsilon,v}$ and $b^+_{\epsilon,v}$, belong to the interval $[a^v_0-\epsilon,a^v_{\tp_v}+\epsilon]=[-\epsilon,a^v_{\tp_v}+\epsilon]$.

In writing the list $\beta^v_0\leq\dots\leq\beta^v_{\tp_v+p_v+1}$ we \emph{choose} a shuffle of the sets
$\set{0,\dots,\tp_v}$ and $\set{0,\dots,p_v}$ into $\set{0,\dots,\tp_v+p_v+1}$, i.e. a pair of strictly increasing and commonly surjective maps
\[
\tilde\eta\colon\set{0,\dots,\tp_v}\to \set{0,\dots,\tp_v+p_v+1},\quad
\eta\colon\set{0,\dots,p_v}\to \set{0,\dots,\tp_v+p_v+1}.
\]
In the generic case the numbers $a^v_0,\dots,a^v_{\tp_v},\alpha^v_0,\dots,\alpha^v_{p_v}$ are all distinct
and the choice of the shuffle is unique,
but nothing prevents that, for some $v\in S^q$, some of these numbers become equal.

Let $\hat{w}^v_0,\dots,\hat{w}^v_{\tp_v+p_v+1}$
denote the corresponding shuffle of the lists of barycentric coordinates,
i.e. $\hat{w}^v_{\tilde\eta(j)}=w^v_j$ and $\hat{w}^v_{\eta(j)}=W_{i_j}(v)$. Define
also
\[
\bfe\colon\set{0,\dots,\tp_v+p_v+1}\times[0,1]\to[0,1],\quad\quad \bfe\colon(\tilde\eta(j),s)\mapsto s,\quad\bfe\colon(\eta(j),s)\mapsto 1-s.
\]

Let $(a^v_{\tp_v},\fc)$ be the product $(t^v_1,\fc^v_1)\dots(t^v_{\tp_v},\fc^v_{\tp_v})$ in $\mHurm$ (in $\bHurm$), and use Notation
\ref{nota:fc}. Then for all $0\leq j\leq \tp_v+p_v+1$ the vertical line $\C_{\Re=\beta_j}$ is disjoint from $P$.
We can then cut the rectangle $\mcR_{a^v_{\tp_v}}$ (respectively, $\bcR_{a^v_{\tp_v}}$) along these vertical lines, and define configurations
$\bar{\fc}^v_j$ in $\Hur(\mcR_{\beta_{j-1},\beta_j})$ (in $\Hur(\bcR_{\beta_{j-1},\beta_j},\bdel)$)
as the parts of $\fc$ lying in the regions $\bS_{\beta_{j-1},\beta_j}$, for all $1\leq j\leq \tp_v+p_v+1$: formally we evaluate the restriction maps $\fri^{\C}_{\bS_{\beta_{j-1},\beta_j}}$ on $\fc$.
Let $\hat{\fc}^v_j$ be the configuration in $\Hur(\mcR_{\beta_j-\beta_{j-1}})$ (in $\Hur(\bcR_{\beta_j-\beta_{j-1}},\bdel)$) given by $(\tau_{-\beta_{j-1}})_*(\bar{\fc}^{v}_j)$, for all $1\leq j\leq \tp_v+p_v+1$.

We define a homotopy $H\colon S^q\times[0,1]\to \reB\mHurm$ (respectively, $H\colon S^q\times[0,1]\to \reB\bHurm$)
by the formula
\[
\begin{split}
H(v,s)= & \Big(\ \bfe(0,s)\hat{w}^v_0\,,\,\dots\,,\,\bfe(\tp_v+p_v+1,s)\hat{w}^v_{\tp_v+p_v+1}\,;\\
    & \pa{\beta_1-\beta_0,\hat{\fc}^v_1}\,,\,\dots\,,\,\pa{\beta_{\tp_v+p_v+1}-\beta_{\tp_v+p_v},\hat{\fc}^v_{\tp_v+p_v+1}}\Big).
\end{split}
\]
The continuity of the formula relies on the fact that we are using the thin bar construction: if
varying $v\in S^q$ two consecutive values $\beta_{j-1}$ and $\beta_j$ become equal, then the corresponding
configuration $(\beta_j-\beta_{j-1},\hat{\fc}^v_{j-1})$ becomes equal to $(0,(\emptyset,\one,\one))$ and
can, thus, be dropped from the list: the weights $\bfe(j-1,s)\hat{w}^v_{j-1}$ and
$\bfe(j,s)\hat{w}^v_j$ are replaced by their sum $\bfe(j-1,s)\hat{w}^v_{j-1}+\bfe(j,s)\hat{w}^v_j$
and we obtain a description of the same configuration in $\reB\mHurm$ (in $\reB\bHurm$) which
is formally \emph{symmetric} in the indices $j-1$ and $j$.

For $s=1$ the list of weights $\bfe(0,s)\hat{w}^v_0,\dots,\bfe(p+p'+1,s)\hat{w}^v_{\tp_v+p_v+1}$
reduces to the list of weights $w^v_0,\dots,w^v_{\tp_v}$, shuffled with $p_v+1$ occurrences of $0$; if we drop the zeros and perform the corresponding products of consecutive elements in the list
$\pa{\beta_1-\beta_0,\hat{\fc}^v_1},\dots,\pa{\beta_{\tp_v+p_v+1}-\beta_{\tp_v+p_v},\hat{\fc}^v_{\tp_v+p_v+1}}$, we recover $\tilde f(v)$.

Similarly, for $s=0$ we obtain the weights $W_{i_1}(v),\dots,W_{i_{p_v}}(v)$ shuffled with $\tp_v+1$ occurrences of $0$; in particular, since $\beta_0=a_0=0$ and $\beta_{\tp_v+p_v+1}=a_p$, at least one zero at the beginning and at least one zero at the end of the list of all weights are dropped. If we perform
the corresponding products of consecutive elements in the list
$\pa{\beta_1-\beta_0,\hat{\fc}^v_1},\dots,\pa{\beta_{\tp_v+p_v+1}-\beta_{\tp_v+p_v},\hat{\fc}^v_{\tp_v+p_v+1}}$, and if we drop the corresponding elements at the two ends of the list,
we recover $g(v)$.

Finally, note that for $v=*\in S^q$ we have that
all configurations $\hat\fc^v_j$ are supported on the empty set, so that $H(v,-)$ is a path in $\reB\mHurm_\emptyset$ (in $\reB\bHurm_\emptyset$). This concludes the proof that $\sigma_*$ is injective
on homotopy groups.

\subsection{Homology of the group completion of \texorpdfstring{$\mHurm$}{mHM}}
\label{subsec:homologygroupcompletion}
The second part of Theorem \ref{thm:delooping} implies, together with Theorem \ref{thm:bHurmloop}, that there is a weak equivalence
\[
 \bHurm_+\simeq \Omega \Hur(\diamo,\del)_{G,G^{\op}}.
\]
If we select one connected component on each side, using Theorem \ref{thm:pi0bHurm}
for the left-hand side, we obtain a weak equivalence
\[
 \bHurm_{+,\one}\simeq \Omega_0 \Hur(\diamo,\del)_{G,G^{\op}}.
\]
By \cite[Lemma 6.16]{Bianchi:Hur2} we have that $\Hur(\diamo,\del)_{\bdel\bdiamolr}$ is a (disconnected) covering of
$\Hur(\diamo,\del)_{G,G^{\op}}$: more precisely, there is a free and properly discontinuous action of $G\times G^{\op}$ on the former space, and the latter is the quotient by this action. Hence also the connected component $\Hur(\diamo,\del)_{\bdel\bdiamolr;\one}\subseteq\Hur(\diamo,\del)_{\bdel\bdiamolr}$ is a covering of $\Hur(\diamo,\del)_{G,G^{\op}}$, with deck transformation group given by the stabiliser in $G\times G^{\op}$ of this component, which is the ``diagonal'' copy of $G$, consisting of pairs $(g,g^{-1,op})$ for varying $g\in G$. We obtain a weak equivalence
\[
 \bHurm_{+,\one}\simeq \Omega_0 \Hur(\diamo,\del)_{\bdel\bdiamolr;\one},
\]
and Lemmas \ref{lem:HurmvsHur} and \ref{lem:bcRvsdiamond} and \cite[Proposition 7.10]{Bianchi:Hur2} yield a weak equivalence
\[
 \Hur(\bdiamolr,\bdel)_{\bdel\bdiamolr;\one}\simeq\Hur_+(\bdiamolr,\bdel)_{\one}\simeq \bHurm_{+,\one} \simeq \Omega_0 \Hur(\diamo,\del)_{\bdel\bdiamolr;\one}\simeq \Omega_0 \Hur_+(\diamo,\del)_{\one}.
\]
Now we use the first part of Theorem \ref{thm:delooping}, together with the fact that $\Hur(\bdiamolr,\bdel)_{\bdel\bdiamolr,\one}$ is a covering 
of $\Hur(\bdiamolr,\bdel)_{G,G^{\op}}$, again by \cite[Lemma 6.16]{Bianchi:Hur2}; taking one component of loop spaces we obtain a weak equivalence
 \[
 \Omega_0 B\mHurm\simeq \Omega_0\Hur(\bdiamolr,\bdel)_{G,G^{\op}}\simeq  \Omega_0\Hur(\bdiamolr,\bdel)_{\bdel\bdiamolr,\one}.
\]
Putting the above weak equivalences together, we obtain
 \[
 \Omega_0 B\mHurm\simeq \Omega_0\Hur(\bdiamolr,\bdel)_{\bdel\bdiamolr,\one} \simeq \Omega^2_0 \Hur_+(\diamo,\del)_{\one}.
 \]
Finally, Theorem \ref{thm:groupcompletion} (which is applicable thanks to Lemma \ref{lem:mHurmweaklybraided}) and Theorem \ref{thm:pi0mHurm},
imply the following homology isomorphism,
\[
 H_*(\mHurm)[\pi_0(\mHurm)^{-1}] \cong \Z[\cG(\Q)] \otimes H_*(\Omega^2_0 \Hur_+(\diamo,\del)_{\one}).
\]
Here $\cG(\Q)$ denotes the enveloping group of the PMQ $\Q$, as in \cite[Definition 2.9]{Bianchi:Hur1}: concretely, this is the group generated by elements $[a]$ for $a$ ranging in $\Q$, under the relations $[a][b]=[b][a^b]$ for all $a,b\in\Q$, and $[a]\hat b=\widehat{ab}$ for all $a,b\in\Q$ such that the product $ab$ is already defined in $\Q$. In fact $\cG(\Q)$ is the universal group receiving a map of PMQs from $\Q$; moreover, $\cG(\Q)$ coincides with the enveloping group of both the monoid $\hQ$ and its free unitalisation $\hQ\sqcup\set{\one}\cong\pi_0(\mHurm)$.

In fact, the combination of Theorems \ref{thm:pi0bHurm}, \ref{thm:bHurmloop} and \ref{thm:delooping} implies the following isomorphisms of discrete monoids (which happen to be groups):
\[
 \pi_1(\Hur(\diamo,\del)_{G,G^{\op}})\cong \pi_1(B\bHurm)\cong \pi_0(\bHurm_+)\cong G.
\]
On the other hand \cite[Lemma 6.16]{Bianchi:Hur2} implies that $\Hur(\diamo,\del)_{\bdel\bdiamolr,\one}$ is a covering of $\Hur(\diamo,\del)_{G,G^{\op}}$ with group of deck transformations $G$, and $\Hur(\diamo,\del)_{\bdel\bdiamolr,\one}$ is connected by Theorem \ref{thm:pi0bHurm}. In the next lemma we prove that $\Hur(\diamo,\del)_{\bdel\bdiamolr,\one}$ is, in fact, the universal cover of $\Hur(\diamo,\del)_{G,G^{\op}}$.
\begin{lem}
\label{lem:simplyconnected}
 The space $\Hur(\diamo,\del)_{\bdel\bdiamolr,\one}$ is simply connected.
\end{lem}
\begin{proof}
Recall the weak equivalence $\sigma\colon B\bHurm\to \Hur(\diamo,\del)_{\bdel\bdiamolr,\one}$. For $g\in G$ let $(1,\fc_g)\in\bHurm_+$ be the pair with $\fc_g$ being a configuration supported on the unique point $\frac 12$, carrying local monodromy $g\in G$. Consider the loop $\gamma_g\colon [0,1]\to B\bHurm$ sending $t\in[0,1]$ to the class of $(t,1-t;(1,\fc_g))\in \Delta^1\times\bHurm\subset\coprod_{p\ge0}\Delta^p\times\bHurm^p$ in the quotient. Observe that the class of $\gamma_g$ in $\pi_1(B\bHurm)$ corresponds to the class of $(1,\fc_g)\in\pi_0(\bHurm_+)$, which along the total monodromy corresponds to the element $g\in G$; in particular every element of $\pi_1(B\bHurm)$ is represented by a loop $\gamma_g$.

Recall now Definition \ref{defn:cHtleftright} and the proof of Theorem \ref{thm:pi0bHurm}, and let $\hat\fc_{g^{-1},g,\one_G}\in\Hur(\cR,\del)_{0,1;\one}$ be the configuration supported on the set $\set{0,\frac 12,1}$, such that the $G$-valued total monodromy sends small loops spinning clockwise around $0,\frac 12,1$ to $g^{-1},g,\one$, respectively.
Let $\hat\gamma_g\colon[0,1]\to\Hur(\cR,\del)_{0,1;\one}$ be the path defined by
\[
\hat\gamma_g(t)=\left\{
\begin{array}{cl}
\cHleft_{1/2}\pa{\hat\fc_{g^{-1},g,\one_G},1}&\mbox{if }0\le t\le  1-\frac{1}{\sqrt{2}};\\[.2cm]
\cHleft_{1/2}\pa{\hat\fc_{g^{-1},g,\one_G},1-\frac{1-2(1-t)^2}{2t(1-t)}}&\mbox{if }1-\frac{1}{\sqrt{2}}\le t\le \frac 12;\\[.2cm]
\cHright_{1/2}\pa{\hat\fc_{g^{-1},g,\one_G},\frac{1-2(1-t)^2}{2t(1-t)}-1}&\mbox{if }\frac12\le t\le \frac{1}{\sqrt{2}};\\[.2cm]
\cHright_{1/2}\pa{\hat\fc_{g^{-1},g,\one_G},1}&\mbox{if }\frac{1}{\sqrt{2}}\le t\le 1.
\end{array}
\right.
\]
Consider also the map $(\cHdiamo_1)_*\colon\Hur(\cR,\del)_{0,1;\one}\to\Hur(\diamo,\del)_{\bdel\bdiamolr;\one}$ induced by $\cHdiamo_1$;
then $\tilde\gamma_g:=(\cHdiamo_1)_*\circ\hat\gamma_g\colon[0,1]\to\Hur(\diamo,\del)_{\bdel\bdiamolr;\one}$ is a path lifting
the loop $\sigma\circ\gamma_g\colon[0,1]\to\Hur(\diamo,\del)_{\bdel\bdiamolr;\one}$ along the covering $\Hur(\diamo,\del)_{\bdel\bdiamolr;\one}\to \Hur(\diamo,\del)_{G,G^{\op}}$.

Both configurations $\tilde\gamma_g(0)$ and $\tilde\gamma_g(1)$ are supported on $\bdel\bdiamolr$; the local monodromies around $\zdiamleft$ and $\zdiamright$ are $\one_G$ and $\one_G$, respectively, for the first configuration, and are $g^{-1}$ and $g$, respectively, for the second. It follows that the path $\tilde\gamma_g$ is a loop if and only if $g=\one_G$. This shows that $\Hur(\diamo,\del)_{\bdel\bdiamolr;\one}\to \Hur(\diamo,\del)_{G,G^{\op}}$ is a universal covering and, in particular, 
$\Hur(\diamo,\del)_{\bdel\bdiamolr;\one}$ is simply connected.
\end{proof}

From now on it is convenient to replace the nice couple $(\diamo,\del)$ with the nice couple $(\cR,\del)$.
For this, fix an orientation-preserving, semialgebraic homeomorphism $\xi\colon\C\to\C$
fixing $*$ and restricting to a homeomorphism of couples $(\diamo,\del)\cong(\cR,\del)$;
then $\xi$ induces a homeomorphism $\xi_*\colon \Hur(\diamo,\del)\cong\Hur(\cR,\del)$, restricting
to a homeomorphism $\Hur_+(\diamo,\del)_{\one}\cong\Hur_+(\cR,\del)_{\one}$. Using again
\cite[Proposition 7.10]{Bianchi:Hur2} we can then replace $\Hur_+(\cR,\del)_{\one}$
by the weakly equivalent space $\Hur(\cR,\del)_{0;\one}$, where $0\in\del\cR$ is the lower left
vertex.

We rephrase the last homology isomorphism, together with the discussion about simply connectedness, 
as the following theorem.
\begin{thm}
 \label{thm:mainHiso}
 In the hypotheses of Theorem \ref{thm:delooping} there is an isomorphism of graded abelian groups
\[
 H_*(\mHurm(\Q))[\pi_0(\mHurm(\Q))^{-1}] \cong \Z[\cG(\Q)] \otimes H_*(\Omega^2_0 \Hur(\cR,\del;\Q,G)_{0;\one}).
\]
Moreover, the space $\Hur(\cR,\del;\Q,G)_{0;\one}$ is simply connected.
\end{thm}
We immediately observe that the left-hand side of the isomorphism in Theorem \ref{thm:mainHiso} only depends on the PMQ $\Q$, and not on the PMQ-group pair $(\Q,G)$ (in particular, not on the group $G$).
In fact the isomorphism of Theorem \ref{thm:mainHiso} is an isomorphism of rings, if we consider on
$\Z[\cG(\Q)] \otimes H_*(\Omega^2_0 \Hur(\cR,\del)_{0;\one})$ the correct
structure of twisted tensor product of rings, which we briefly describe in the following.

The group $\cG(\Q)$ acts on the right on $\Hur(\cR,\del)_{0;\one}=\Hur(\cR,\del;\Q,G)_{0;\one}$ by global conjugation: in fact, $G$ acts on the right on this space by global conjugation,
and we consider the map of groups $\cG(\fe)\colon\cG(\Q)\to G$. Consequently, $\cG(\Q)$ acts on the right
on $H_*(\Omega^2_0 \Hur(\cR,\del)_{0;\one})$ by automorphisms of rings.

For $g_1,g_2\in\cG(\Q)$ and $x_1,x_2\in H_*(\Omega^2_0 \Hur(\cR,\del)_{0;\one})$
we define the twisted product $(g_1\otimes x_1)\cdot(g_2\otimes x_2):=(g_1g_2\otimes (x_1^{g_2}\cdot x_2)$.
This assignment extends to an associative product on $\Z[\cG(\Q)] \otimes H_*(\Omega^2_0 \Hur(\cR,\del)_{0;\one})$,
and with this ring structure the isomorphism of Theorem \ref{thm:mainHiso} is an isomorphism of rings.

The isomorphism of Theorem \ref{thm:mainHiso} is a bit surprising
at first glance, because $\mHurm(\Q)$
is \emph{not}, in general, weakly equivalent to an $E_2$-algebra.

\section{The space \texorpdfstring{$\bB(\Q_+,G)$}{bB(Q,G)}}
\label{sec:BQG}
For concrete homology computations the space $\Hur(\cR,\del;\Q,G)_{0;\one}$ is too large.
In this section we introduce a homotopy equivalent subspace
\[
\bB(\Q_+,G)\subset\Hur(\cR,\del;\Q,G)_{0;\one};
\]
under the assumption that $\Q$ is \emph{augmented}.
If, in addition, we assume that $\Q$ is \emph{normed},
then $\bB(\Q_+,G)$ admits a natural filtration by closed subspaces. In the next
section, assuming further that $\Q$ is finite and \emph{rationally Poincar\'e},
we will exploit this filtration to compute explicitly the rational cohomology
ring of $\bB(\Q_+,G)$.

Recall from \cite[Definitions 4.1 and 4.9]{Bianchi:Hur1} that a PMQ is \emph{augmented} if the set $\Q_+:=\Q\setminus\set{\one}$ is an ideal for the partial product, i.e. if for all $a,b\in\Q$ such that $ab=\one$ we have $a,b=\one$. A \emph{normed} PMQ is a PMQ $\Q$ together with a morphism of PMQs $N\colon \Q\to\N$ such that $N^{-1}(0)=\set{\one}$. Every normed PMQ is also augmented.
\begin{defn}
 \label{defn:bB}
 Let $\fc\in\Hur(\cR,\del;\Q,G)_{0;\one}$ and use Notation \ref{nota:fc}, so that the support
 $P$ of $\fc$ splits as
 $\set{z_1,\dots,z_l}\subset\mcR$
 and $\set{z_{l+1},\dots,z_k}\subset\del\cR$.
 Let $\beta\subset\mcR$ be a clockwise oriented simple closed curve in $\mcR\setminus P$
 such that the disc bounded by $\beta$ contains all points $z_1,\dots,z_l$.
 
 The configuration $\fc$ lies in the subspace $\bB(\Q,G)\subset\Hur(\cR,\del;\Q,G)_{0;\one}$ if
 the conjugacy class of $\fG(P)$ corresponding to $\beta$ is contained
 in the PMQ $\fQext(P)_{\psi}\subset\fG(P)$
 (see \cite[Definition 2.13]{Bianchi:Hur2}).
 
 If $\Q$ is augmented, we define $\bB(\Q_+,G)$ as the intersection
 \[
 \bB(\Q_+,G):=\bB(\Q,G)\cap \Hur(\cR,\del;\Q_+,G)_{0;\one}\subset\Hur(\cR,\del;\Q,G)_{0;\one}.
 \]
\end{defn}
Roughly speaking, a configuration $\fc\in\Hur(\cR,\del;\Q,G)_{0;\one}$ lies in $\bB(\Q,G)$
if, using Notation \ref{nota:fc}, the $l$ values of
the monodromy $\psi$ around the $l$ points of $P\cap\mcR$ can be multiplied in $\Q$.
See Figure \ref{fig:bB}.

 \begin{figure}[ht]
 \begin{tikzpicture}[scale=4,decoration={markings,mark=at position 0.38 with {\arrow{>}}}]
  \draw[dashed,->] (-.05,0) to (1.6,0);
  \draw[dashed,->] (0,-1.1) to (0,1.1);
  \node at (0,-1) {$*$};
  \draw[very thick] (0,0) rectangle (1,1);
  \fill[gray, opacity=.5] (0,0) rectangle (1,1);
  \draw[thick, dotted] (.04,.04) rectangle (.98,.98);
  \node at (.7,.95){\tiny$\beta$};
  \node at (0,0){$\bullet$};
  \draw[thin, looseness=1.2, postaction={decorate}] (0,-1) to[out=91,in=-90]  (-.03,-.1) to[out=90,in=-90] (-.07,.05)  to[out=90,in=90] node[left]{\tiny$g_6$} (.03,.05)  to[out=-90,in=90] 
  (.01,-.1) to[out=-90,in=89] (0,-1);
  \node at (.1,.3){$\bullet$}; 
  \node at (.3,1){$\bullet$}; 
  \node at (.45,.5){$\bullet$};
  \draw[thin, looseness=1.2, postaction={decorate}] (0,-1) to[out=82,in=-90]  (.05,.1) to[out=90,in=-90] (.01,.3)  to[out=90,in=90] node[above]{\tiny$a_1$} (.2,.3)  to[out=-90,in=90] (.1,.1) to[out=-90,in=80] (0,-1);
  \draw[thin, looseness=1.2, postaction={decorate}] (0,-1) to[out=79,in=-90] (.28,.9)
  to[out=90,in=-90] (.26,1) to[out=90,in=90] (.34,1) to[out=-90,in=90] node[left]{\tiny$g_4$}
  (.3,.9) to[out=-90,in=77]  (0,-1);
  \draw[thin, looseness=1.2, postaction={decorate}] (0,-1) to[out=76,in=-90] (.4,.3)
  to[out=90,in=-90] (.35,.5) to[out=90,in=90] node[above]{\tiny$a_2$} (.49,.5) to[out=-90,in=90] (.45,.3) to[out=-90,in=74]  (0,-1);
  \node at (.65,.0){$\bullet$};
  \node at (.95,.8){$\bullet$};
  \draw[thin, looseness=1.2, postaction={decorate}] (0,-1) to[out=62,in=-120]  (.55,-.1) to[out=60,in=-90] (.52,.05)  to[out=90,in=90] node[above]{\tiny$g_5$} (.7,.05)  to[out=-90,in=50] (.65,-.1) to[out=-130,in=60] (0,-1);
  \draw[thin, looseness=1.2, postaction={decorate}] (0,-1) to[out=56,in=-90] (.9,.7)
  to[out=90,in=-90] (.85,.8) to[out=90,in=90] node[above]{\tiny$a_3$} (.99,.8) to[out=-90,in=90] (.95,.7) to[out=-90,in=54]  (0,-1);
 \end{tikzpicture}
 \caption{The configuration $\fc$ lies in the space $\bB(\Q_+,G)$ if $g_6\fe(a_1)g_4\fe(a_2)g_5\fe(a_3)=\one\in G$, none of $a_1,a_2,a_3$ is equal to $\one\in\Q$, and the product $a_1a_2(a_3^{g_5^{-1}})$ is defined in $\Q$.
 }
\label{fig:bB}
\end{figure}

\begin{defn}
 \label{defn:inert}
 Let $\fT$ be a nice couple and let $(\Q,G)$ be a PMQ-group pair with $\Q$ augmented.
 Let $\fc=(P,\psi,\phi)\in\Hur(\fT;\Q,G)$; a point $z\in P$ is \emph{inert} for $\fc$ if
 $z\in \cX\setminus\cY$ and $\psi$ sends to $\one_{\Q}$
 each element of $\fQ(P)$ represented by a small loop spinning clockwise around $z$.
\end{defn}
If $\Q$ is augmented, a configuration $\fc$ lies in $\bB(\Q_+,G)$ if it lies in $\bB(\Q,G)$ and,
moreover, no point of the support of $\fc$ is inert: in other words, $\psi$ attains values different from $\one$ around all $l$ points of $P\cap\mcR$.
\begin{ex}
Suppose that $\Q$ is finite and normed, and let $N_{\mathrm{max}}\in\N$ be the maximal
norm of an element of $\Q$. Let $\fc\in\bB(\Q,G)$, and use Notation \ref{nota:fc}.
Then at most $N_{\mathrm{max}}$ of the $l$ points in $P\cap \mcR_1$ can be non-inert;
in particular if $\fc\in\bB(\Q_+,G)$, then $l\leq N_{\mathrm{max}}$.
There is also another evident restriction on the behaviour of $\psi$:
if $\gen_1,\dots,\gen_k$ is an admissible generating set for $\fG(P)$, then
$\sum_{i=1}^l N(\psi(\gen_i))\leq N_{\mathrm{max}}$.
\end{ex}
\begin{ex}
Let $\Q$ be the abelian PMQ $\set{\one,\bullet}$
with trivial partial multiplication, and let $G=\set{\one}$ be the trivial group.
Let $\fc\in\bB(\Q_+,G)$, and use Notation \ref{nota:fc}.
Then $P\cap\mcR$ is either empty, or it contains exactly one point. We can define
a map $\bB(\Q_+,G)\to \cR/\del\cR\cong S^2$ by looking at the position of the unique
point of $P$ in $\mcR$,
and taking the quotient point $[\del\cR]\in\cR/\del\cR$ if $P\cap\mcR=\emptyset$.
In fact, the map $\bB(\Q_+,G)\to \cR/\del\cR$ is a quasifibration with fibre
the Ran space $\Ran(\del\cR)$, so it is a weak homotopy equivalence.

We will see in Proposition \ref{prop:bBHur} that the inclusion $\bB(\Q_+,G)\subset\Hur(\cR,\del\cR;\Q,G)_{0;\one}$ is also a homotopy equivalence; hence, in this case, Theorem \ref{thm:mainHiso} reduces to a classical result of Segal \cite{Segal73} stating that the group completion of the topological monoid
$\coprod_{n\geq 0}\mathrm{Conf}_n(\R^2)$ is $\Omega^2S^2$. Passing from
$\Hur(\cR,\del;\Q_+,G)_{0;\one}$ to its subspace $\bB(\Q_+,G)$ should thus be regarded as the analogue
of passing from the relative configuration space $\mathrm{Conf}(\cR,\del)$ to the sphere
$\mcR/\del\mcR$ by a scanning argument.
\end{ex}

\subsection{Deformation retraction onto \texorpdfstring{$\bB(\Q,G)$}{bB(Q,G)}}
In this subsection we will prove that $\Hur(\cR,\del;\Q,G)_{0;\one}$ admits a deformation
retraction onto its subspace $\bB(\Q,G)$; if $\Q$ is augmented, the same argument will give by
restriction a deformation retraction
of $\Hur(\cR,\del;\Q_+,G)_{0;\one}$ onto its subspace $\bB(\Q_+,G)$.
Using \cite[Proposition 7.4]{Bianchi:Hur2}, we will therefore obtain the following proposition.
\begin{prop}
 \label{prop:bBHur}
 For any PMQ-group pair $(\Q,G)$ the inclusion of
 $\bB(\Q,G)$ into $\Hur(\cR,\del;\Q,G)_{0;\one}$ is a homotopy equivalence. If $\Q$ is augmented,
 the following is a square of inclusions which are homotopy equivalences:
 \[
 \begin{tikzcd}
\bB(\Q_+,G)\ar[r,hook,"\subset"]\ar[d,hook,"\subset"] & \bB(\Q,G)\ar[d,hook,"\subset"]\\
\Hur(\cR,\del;\Q_+,G)_{0;\one}\ar[r,hook,"\subset"] & \Hur(\cR,\del;\Q,G)_{0;\one}.
 \end{tikzcd}
\]
\end{prop}
The rough idea of the proof of Proposition \ref{prop:bBHur} is that each configuration $\fc\in\Hur(\cR,\del;\Q,G)_{0;\one}$ can be gradually magnified around the centre $\zcentre\in\cR$, letting gradually more and more points collide with $\del\cR$: such points are downgraded to points in the support of $\fc$ around which only the $G$-valued monodromy is defined and they remain fixed during further magnification. At some finite time we obtain a configuration satisfying the properties of Definition \ref{defn:bB}, and we stop the magnification.
\begin{defn}
\label{defn:cHbB}
Let $z_0\in\mcR$. For $z\in\C$ let $\fd^{\bB}_{z_0}(z)\in[1,\infty]$ be the infimum of all $s\ge1$
such that $z_0+s(z-z_0)\notin\cR$; note that $\fd^{\bB}_{z_0}(z)=\infty$ if and only if $z=z_0$.
We define a map $\cH^{\bB}_{z_0}\colon\C\times[1,\infty)\to\C$ by the formula
 \[\def\arraystretch{1.3}
 \cH^{\bB}_{z_0}(z,s)=\left\{
  \begin{array}{ll}
   z & \mbox{if } z\notin \mcR\\
   z_0+s(z-z_0) &\mbox{if }z\in\cR \mbox{ and } z_0+s(z-z_0)\in\cR\\
   z_0+\fd^{\bB}_{z_0}(z)\cdot(z-z_0) &
   \mbox{if }z\in\cR \mbox{ and } z_0+s(z-z_0)\notin\cR.\\   
  \end{array}
 \right.
 \]
\end{defn}
Roughly speaking, the map $\cH^{\bB}_{z_0}(-,s)$
expands the square $\frac{s-1}s z_0+\frac 1s \mcR$, which has side length $\frac 1s$,
to the entire $\cR$, by a homothety centred at $z_0$ of rescaling factor $s$ and collapses
$\cR\setminus (\frac{s-1}s z_0+\frac 1s \mcR)$ onto $\del\cR$.
Note that for all $s\geq 1$ the map $\cH^{\bB}_{z_0}(-,s)$ is an endomorphism of the nice couples $(\cR,\del)$;
note also that $\cH^{\bB}_{z_0}(0,s)=0$ for all $s\ge1$, and that $\cH^{\bB}_{z_0}(-,1)$ is the identity of $\C$. Thus we obtain a continuous map
\[
 (\cH^{\bB}_{z_0})_*\colon\Hur(\cR,\del;\Q,G)_{0;\one}\times[1,\infty)\to \Hur(\cR,\del;\Q,G)_{0;\one},
\]
such that $(\cH^{\bB}_{z_0})_*(-,1)$ is the identity of $\Hur(\cR,\del;\Q,G)_{0;\one}$.
\begin{defn}
 \label{defn:intwindow}
 Let $z_0\in\mcR$. For each $0<\epsilon<1$ denote by $\beta_{z_0,\epsilon}\subset\mcR$
 the simple closed curve whose support is the square
 $(1-\epsilon)z_0+\epsilon(\del\cR)$, i.e. the boundary of the square of side length $\epsilon$ obtained from $\del\cR$ by a homothety centred at $z_0$
 of rescaling factor $\epsilon$; we orient $\beta_{z_0,\epsilon}$ clockwise.
 Let $\fc\in\Hur(\cR,\del)_{\one}$ and use Notation \ref{nota:fc}.
 We denote by $\mathfrak{W}_{z_0}(\fc)\in[0,1]$ the supremum of all $0<\epsilon<1$ satisfying
 the following properties:
 \begin{itemize}
  \item $P\cap \beta_{z_0,\epsilon}=\emptyset$;
  \item every element $\fg\in\fG(P)$ in the conjugacy class corresponding to $\beta_{z_0,\epsilon}$
  belongs to  $\fQext(P)_{\psi}\subseteq\fG(P)$.
 \end{itemize}
\end{defn}
Note that $\mathfrak{W}_{z_0}(\fc)\leq 1$ for all $\fc\in\Hur(\cR,\del;\Q,G)_{\one}$; moreover,
$\mathfrak{W}_{z_0}(\fc)= 1$ if and only if $\fc\in\bB(\Q,G)$.
Note, on the other hand, that $\mathfrak{W}_{z_0}(\fc)>0$, as for a generic and very small $\epsilon>0$ the curve
$\beta_{z_0,\epsilon}$ is disjoint from $P$ and encloses at most one point of $P$, so it corresponds to a conjugacy class of $\fG(P)$
which is contained in $\fQ(P)\subset\fQext(P)_{\psi}$.
Note also that the assignment $\fc\mapsto\mathfrak{W}_{z_0}(\fc)$ is continuous in $\fc$.
\begin{nota}
Recall Notation \ref{nota:zcentre}. We simplify the notation and write
$\cH^{\bB}=\cH^{\bB}_{\zcentre}$ and $\mathfrak{W}=\mathfrak{W}_{\zcentre}$.
\end{nota}
\begin{proof}[Proof of Proposition \ref{prop:bBHur}]
We define a homotopy
\[
H^{\bB}\colon \Hur(\cR,\del;\Q,G)_{0;\one}\times[0,1]\to\Hur(\cR,\del;\Q,G)_{0;\one},
\]
\[
H^{\bB}(\fc,s)=\cH^{\bB}_*\pa{\fc\,,\,1-s+s\frac 1{\mathfrak{W}(\fc)}}.
\]
Note that $H^{\bB}(-,0)$ is the identity of $\Hur(\cR,\del;\Q,G)_{0;\one}$,
that $H^{\bB}(-,1)$ takes values in $\bB(\Q,G)$ and that
$H^{\bB}(-,s)$ restricts to the identity
of $\bB(\Q,G)\subset\Hur(\cR,\del;\Q,G)$ for all $0\leq s\leq 1$.

This proves the first homotopy equivalence in the statement. The second homotopy
equivalence is completely analogous: we have constructed the deformation retraction
of $\Hur(\cR,\del;\Q,G)_{0;\one}$ onto its subspace $\bB(\Q,G)$ using enriched
functoriality with respect to
maps of nice couples; this implies that $H^{\bB}(-,s)$
restricts at all times to a self-map of $\Hur(\cR,\del;\Q_+,G)_{0;\one}$
and, in particular, the restriction of $H^{\bB}(-,1)$ on $\Hur(\cR,\del;\Q_+,G)_{0;\one}$ takes values
in $\Hur(\cR,\del;\Q_+,G)_{0;\one}\cap\bB(\Q,G)=\bB(\Q_+,G)$.
\end{proof}

\subsection{Norm filtration}
In the rest of the section we assume that $\Q$ is endowed with a norm $N\colon\Q\to\N$. Our aim is to introduce
a filtration $\FN_{\bullet}$ on $\bB(\Q_+,G)$. It will be convenient to define
a norm filtration $\FN_{\bullet}$ more generally on the Hurwitz space
$\Hur(\fT;\Q,G)$ associated with any nice couple $\fT$.
\begin{defn}
\label{defn:FN}
 Let $\fT$ be a nice couple, let $\fc\in\Hur(\fT;\Q,G)$ and use Notation \ref{nota:fc}.
 Let $\gen_1,\dots,\gen_k$ be an admissible generating set for $\fG(P)$.
 We define a function of sets $N\colon\Hur(\fT;\Q,G)\to\N$ by
 \[
  N(\fc)=N(\psi(\gen_1))+\dots+ N(\psi(\gen_l)),
 \]
 and call $N(\fc)$ the \emph{norm} of the configuration $\fc$.
 
 For $\nu\geq0$ we define the $\nu$\textsuperscript{th} filtration layer $\FN_\nu\Hur(\fT;\Q,G)$
 as the subspace of configurations $\fc$ with $N(\fc)\leq\nu$. We also set
 $\FN_{-1}\Hur(\fT;\Q,G)=\emptyset$.
 
 For $\nu\geq0$ we denote by $\fFN_\nu\Hur(\fT;\Q,G)$ the $\nu$\textsuperscript{th} filtration stratum
 \[
 \fFN_\nu\Hur(\fT;\Q,G):= \FN_\nu\Hur(\fT;\Q,G)\setminus\FN_{\nu-1}\Hur(\fT;\Q,G).
 \]
\end{defn}
Recall from \cite[Definition 2.5]{Bianchi:Hur2} that given a nice couple $\fT=(\cX,\cY)$ and a finite subset $P\subset\cX$, an \emph{adapted covering} of $P$ is a collection $\uU$ of disjoint, semialgebraic open discs in $\C$ containing each a single point of $P$, and such that each point in $P\setminus\cY$ is surrounded by a disc disjoint from $\cY$. The topology on $\Hur(\fT;\Q,G)$ has a basis given by the open neighbourhoods $\fU(\fc;\uU)$, for varying 
$\fc=(P,\psi,\phi)\in\Hur(\fT;\Q,G)$ and varying $\uU$ among adapted covers of $P$.
\begin{lem}
 \label{lem:FNclosed}
 For all $\nu\geq 0$ we have that
 $\Hur(\fT;\Q,G)\setminus\FN_{\nu-1}\Hur(\fT;\Q,G)$ is open in $\Hur(\fT;\Q,G)$.
\end{lem}
\begin{proof}
Let $\fc=(P,\psi,\phi)\in\Hur(\fT;\Q,G)$ and assume
that $N(\fc)\geq\nu$. Let $\uU$ be an adapted covering of $P$.
We claim that the open neighbourhood $\fU(\fc,\uU)$
is contained in $\Hur(\fT;\Q,G)\setminus\FN_{\nu-1}\HurTQG$.

Let $\fc'\in\fU(\fc,\uU)$, and use Notation \ref{nota:fc}. For all $1\leq i\leq l$
let $P'_i=P'\cap U_i$, and let $P'_i=\set{z'_{i,1},\dots,z'_{i,k'_i}}$.
Choose an admissible generating set $\gen_1,\dots,\gen_k$ of $\fG(P)\cong\fG(\uU)$. We can regard $\fG(\uU)$ as a subgroup of $\fG(P')$. We can choose an admissible
generating set $\gen'_1,\dots,\gen'_{k'}$
of $\fG(P')$ with the following property: for all $1\leq i\leq l$,
if $\gen'_{i,1},\dots,\gen'_{i,k'_i}$ are the elements represented by loops spinning
around the points $z'_{i,1},\dots,z'_{i,k'_i}$ respectively, then the product
$\gen'_{i,1}\dots\gen'_{i,k'_i}$ is equal to $\gen_i$ in $\fG(P')$.
The hypothesis $\fc'\in\fU(\fc,\uU)$ implies the following equality in $\Q$, for all $1\leq i\leq l$:
\[
\psi'(\gen'_{i,1})\cdot\dots\cdot\psi'(\gen'_{i,k'_i})=\psi(\gen_i),
\]
whence, using the norm on $\Q$, we obtain
\[
N(\psi'(\gen'_{i,1}))+\dots+N(\psi'(\gen'_{i,k'_i}))=N(\psi(\gen_i)).
\]
Summing over $1\leq i\leq l$, and recalling that $P'_1\cup\dots\cup P'_l$ might
be a proper subset of $P'\setminus\cY=\set{z'_1,\dots,z'_{l'}}$, we obtain
\[
 \nu\leq N(\fc)=\sum_{i=1}^lN(\psi(\gen_i))=\sum_{i=1}^l\sum_{j=1}^{k'_i}N(\psi'(\gen'_{i,j}))\leq\sum_{i=1}^{l'}N(\psi'(\gen'_i))=N(\fc'),
\]
which shows that $\fU(\fc,\uU)$ is contained in $\Hur(\fT;\Q,G)\setminus\FN_{\nu-1}\HurTQG$.

\end{proof}
\begin{nota}
 \label{nota:FNbB}
 For every subspace $X\subset\Hur(\fT;\Q,G)$ and for $\nu\ge-1$, we use the notation
 $\FN_\nu X=\FN_\nu\Hur(\fT;\Q,G)\cap X$. For $\nu\ge0$ we use the notation
 $\fFN_\nu X=\fFN_\nu\Hur(\fT;\Q,G)\cap X$.
\end{nota}
 We will use Notation \ref{nota:FNbB} mainly in the case
 $X=\bB(\Q_+,G)\subset\Hur(\cR,\del;\Q,G)$.

 \subsection{A model for \texorpdfstring{$BG$}{BG}}
Our next goal is to analyse the strata $\fFN_\nu\bB(\Q_+,G)$.
We start with $\fFN_0\bB(\Q_+,G)$, which can be identified with
$\Hur(\del\cR,\del\cR;\Q,G)_{0;\one}$. By \cite[Lemmas 5.4 and 5.5]{Bianchi:Hur2}
we can equivalently consider the space $\Hur(\del\cR;G)_{0;\one}$,
where the group $G$ is considered as a (complete) PMQ.
In this subsection we prove the following proposition.
\begin{prop}
 \label{prop:FNzeroBG}
 The space $\Hur(\del\cR;G)_{0;\one}$ is an Eilenberg-Maclane space of type $K(G,1)$.
\end{prop}
\begin{defn}
 \label{defn:sqcap}
We denote by $\sqcap\subset\del\cR$ the union of the three closed sides of $\cR$
\[
\sqcap:=\set{0}\times[0,1]\,\cup\,[0,1]\times\set{1}\,\cup\,\set{1}\times[0,1].
\]
\end{defn}

 \begin{lem}
  \label{lem:Hursqcapcontractible}
  The spaces $\Hur(\sqcap;G)_{0;\one}$ and $\Hur(\sqcap;G)_{0,1;\one}$ are contractible.
 \end{lem}
\begin{proof}
 Note that $\sqcap$ is contractible;
 more precisely we can fix a semialgebraic homotopy $\cH^{\sqcap}\colon\C\times[0,1]\to\C$
 with the following properties:
 \begin{itemize}
  \item $\cH^{\sqcap}(-,s)$ is a \emph{lax} endomorphism of the nice couple $(\sqcap,\emptyset)$, for all $0\le s\le 1$ (see \cite[Definition 4.2]{Bianchi:Hur2});
  \item $\cH^{\sqcap}(0,s)=0\in\C$ for all $0\leq s\leq 1$;
  \item $\cH^{\sqcap}(-,0)=\Id_{\C}$;
  \item $\cH^{\sqcap}(-,1)$ maps $\sqcap$ constantly to $0$.
 \end{itemize}
 By functoriality, using that $G$ is a complete PMQ, we obtain a homotopy
 \[
 \cH^{\sqcap}_*\colon\Hur(\sqcap;G)_{0;\one}\times[0,1]\to\Hur(\sqcap;G)_{0;\one}.
 \]
 Note that the map $\cH^{\sqcap}_*(-,1)$ takes values in $\Hur(\set{0},G)_{0;\one}$,
 which is just a point. Thus the homotopy $\cH^{\sqcap}_*$ exhibits
 $\Hur(\sqcap;G)_{0;\one}$ as a contractible space.
 
 By \cite[Proposition 7.10]{Bianchi:Hur2}, the inclusion
 $\Hur(\sqcap;G)_{0,1;\one}\subset\Hur(\sqcap;G)_{0;\one}$
 is a homotopy equivalence, hence $\Hur(\sqcap;G)_{0,1;\one}$ is also contractible.
 \end{proof}
 
 Now let $\xi^{\sqcap}\colon\C\to\C$ be a semialgebraic map with the following properties:
 \begin{enumerate}
  \item $\xi^{\sqcap}$ is a lax endomorphism of the nice couple $(\del\cR,\emptyset)$, in particular it restricts to a map $\del\cR\to\del\cR$;
  \item $\xi^{\sqcap}$ fixes $\C_{\Re\le0}$ pointwise;
  \item $\xi^{\sqcap}$ maps the horizontal segment $[0,1]\times\set{0}$ constantly to $0$;
  \item $\xi^{\sqcap}$ restricts to a homeomorphism $\C\setminus([0,1]\times\set{0})\to\C\setminus\set{0}$.
 \end{enumerate}
 Note that $\xi^{\sqcap}$
 is a lax morphism of nice couples $(\sqcap,\emptyset)\to(\del\cR,\emptyset)$; we obtain by functoriality
 a map $\xi^{\sqcap}_*\colon\Hur(\sqcap;G)_{0,1;\one}\to \Hur(\del\cR;G)_{0;\one}$, see Figure \ref{fig:xi}.

 \begin{figure}[ht]
 \begin{tikzpicture}[scale=2.8,decoration={markings,mark=at position 0.38 with {\arrow{>}}}]
  \draw[dashed,->] (-.5,0) to (1.3,0);
  \draw[dashed,->] (0,-1.1) to (0,1.3);
  \node at (0,-1) {$*$};
  \draw[very thick] (0,0) rectangle (1,1);
  \node at (0,0){$\bullet$};
  \draw[thin, looseness=1.2, postaction={decorate}] (0,-1) to[out=91,in=-90]  (-.03,-.1) to[out=90,in=-90] (-.07,.0)  to[out=90,in=90] node[right]{\tiny$g_4$} (.03,.0)  to[out=-90,in=90] 
  (.01,-.1) to[out=-90,in=89] (0,-1);
  \node at (0,.3){$\bullet$}; 
  \node at (.7,1){$\bullet$}; 
  \draw[thin, looseness=1.2, postaction={decorate}] (0,-1) to[out=100,in=-120]  (-.1,.33) to[out=60,in=-180] 
  (0,.37)  to[out=0,in=0] node[right]{\tiny$g_2$} (0,.27)  to[out=180,in=60] (-.1,.28) to[out=-120,in=97] (0,-1);
  \draw[thin, looseness=1.2, postaction={decorate}] (0,-1) to[out=108, in= -135] (-.1,1.05) to [out=45, in=100]
  (.7,1.1) to[out=-80,in=90] (.75,1) to[out=-90,in=-90] (.65,1) to[out=90,in=-80] node[left]{\tiny$g_3$}
  (.67,1.1) to[out=100, in=45] (-.05,1.05) to [out=-135, in=106]  (0,-1);
  \node at (1,.1){$\bullet$};
  \draw[thin, looseness=1.2, postaction={decorate}] (0,-1) to[out=40,in=-60] (1.1,.05)
  to[out=120,in=0] (1,.03) to[out=180,in=180] node[left]{\tiny$g_1$} (1,.23) to[out=0,in=120] (1.1,.18) to[out=-60,in=20]  (0,-1); 
 
 \begin{scope}[shift={(1.9,0)}]
  \draw[dashed,->] (-.5,0) to (1.3,0);
  \draw[dashed,->] (0,-1.1) to (0,1.3);
  \node at (0,-1) {$*$};
  \draw[very thick] (0,0) to (0,1) to (1,1) to (1,0);
  \node at (0,0){$\bullet$};
  \node at (1,0){$\bullet$};
  \draw[thin, looseness=1.2, postaction={decorate}] (0,-1) to[out=91,in=-90]  (-.03,-.1) to[out=90,in=-90] (-.07,.0)  to[out=90,in=90] node[right]{\tiny$g'_4$} (.03,.0)  to[out=-90,in=90] 
  (.01,-.1) to[out=-90,in=89] (0,-1);
  \draw[thin, looseness=1.2, postaction={decorate}] (0,-1) to[out=46,in=-120] (.95,-.1)
  to[out=60,in=-90] (.93,0) to[out=90,in=90] node[left]{\tiny$g'_5$} (1.05,0) to[out=-90,in=60] (.98,-.1) to[out=-120,in=44]  (0,-1);
  \draw[dashed, looseness=1.2, postaction={decorate}] (0,-1) to[out=92, in=-160] (0,.15)
    to [out=20, in=160] node[above]{\tiny$g'_4g'_5=g_4$} (1,.15) to[out=-20, in=37] (0,-1);
  \node at (0,.3){$\bullet$}; 
  \node at (.3,1){$\bullet$}; 
  \draw[thin, looseness=1.2, postaction={decorate}] (0,-1) to[out=100,in=-120]  (-.1,.33) to[out=60,in=-180] 
  (0,.37)  to[out=0,in=0] node[right]{\tiny$g_2$} (0,.27)  to[out=180,in=60] (-.1,.28) to[out=-120,in=97] (0,-1);
  \draw[thin, looseness=1.2, postaction={decorate}] (0,-1) to[out=108, in= -135] (-.1,1.05) to [out=45, in=100] (.3,1.1)
  to[out=-80,in=90] (.35,1) to[out=-90,in=-90] (.25,1) to[out=90,in=-80] node[left]{\tiny$g_3$}
  (.27,1.1) to[out=100, in=45] (-.05,1.05) to [out=-135, in=106]  (0,-1);
  \node at (1,.8){$\bullet$};
  \draw[thin, looseness=1.2, postaction={decorate}] (0,-1) to[out=36,in=-90] (1.2,0) to[out=90,in=-60] (1.1,.8)
  to[out=120,in=0] (1,.73) to[out=180,in=180] node[left]{\tiny$g_1$} (1,.87) to[out=0,in=120] (1.1,.82) to[out=-60,in=90] (1.25,0) to [out=-90,in=33]  (0,-1);
 \end{scope}
 \end{tikzpicture}
 \caption{Left: a configuration $\fc\in\Hur(\del\cR;G)_{0;\one}$. Right: a configuration $\fc'$ in the fibre $(\xi^{\sqcap}_*)^{-1}(\fc)\subset\Hur(\sqcap;G)_{0,1;\one}$.
 }
\label{fig:xi}
\end{figure}
We will prove that $\xi^{\sqcap}_*$ is a covering map.
\begin{lem}
 \label{lem:xisqcapfibres}
For $\fc\in\Hur(\del\cR;G)_{0;\one}$, the fibre of $\xi^{\sqcap}_*$ over $\fc$ is a non-empty and discrete subspace of $\Hur(\sqcap;G)_{0,1;\one}$.
\end{lem}
\begin{proof}
Write $\fc=(P,\psi)$, where $P=\set{z_1,\dots,z_k}$ and
$\psi\colon\fQ_{(\del\cR,\emptyset)}(P)\to G$ is a map of PMQs;
without loss of generality suppose $z_k=0$.
Assume that we are given $\fc'=(P',\psi')\in(\xi^{\sqcap}_*)^{-1}(\fc)$. 
Note that if $\xi^{\sqcap}_*(\fc')=\fc$, then, in particular, $\xi^{\sqcap}(P')=P$ and by properties (3) and (4) of $\xi^{\sqcap}$ we must have $P'=(\xi^{\sqcap})^{-1}(P)\cap\sqcap\subset\C$.
To fix notation, let $z'_i=(\xi^{\sqcap})^{-1}(z_i)\in P'$ for $1\leq i\leq k-1$, and let $z'_k=0\in P'$ and $z'_{k+1}=1\in P'$.

Fix an admissible generating set $\gen_1,\dots,\gen_k$ for $\fG(P)$, and assume that $\gen_k$
is represented by a loop supported in a small neighbourhood of the vertical segment $\set{0}\times[-1,0]\subset\C$ joining $*$ to 0. Then we can consider
$(\xi^{\sqcap})^{-1}$ as a map $\CmP\to\C\setminus([0,1]\times\set{0}\cup P')\subset\CmP'$, and map the generators $\gen_1,\dots,\gen_k$
to elements of $\fG(P')$. Note that $\gen_1,\dots,\gen_{k-1}$ are mapped to simple loops spinning
clockwise around the points $z'_1,\dots,z'_{k-1}$, whereas $\gen_k$ is mapped to a simple loop
spinning clockwise around the segment $[0,1]\subset \C$. We decompose $(\xi^{\sqcap})^{-1}_*(\gen_k)$
as a product of two elements $\gen'_k,\gen'_{k+1}$, represented by simple loops in $\CmP'$ spinning
clockwise around $z'_k$ and $z'_{k+1}$
respectively, and define $\gen'_i=(\xi^{\sqcap})^{-1}_*(\gen_i)$ for $1\leq i\leq k-1$: thus we obtain
an admissible generating set $\gen'_1,\dots,\gen'_{k+1}$ for $\fG(P')$. Note that for $1\le i\le k-1$ the generator $\gen'_i$ can be represented by a simple loop in $\CmP'$ supported in $\C\setminus(P'\cup [0,1]\times \set{0})$; similarly $\gen'_k$ and $\gen'_{k+1}$ can be represented by loops supported on small neighbourhoods of the straight segments in $\C$ joining $*$ with $0$ and $1$ respectively.

Since $\xi^{\sqcap}_*(\fc')=\fc$, we have in $G$ the equalities $\psi'(\gen'_i)=\psi(\gen_i)$ for $1\leq i\leq k-1$, and
$\psi(\gen_k)=\psi'(\gen'_k)\psi'(\gen'_{k+1})$.

Vice versa, for any factorisation of $\psi(\gen_k)\in G$ as the product $gh$ of two elements in $G$,
we can define a configuration $\fc'=(P',\psi')$ by setting $P'=(\xi^{\sqcap})^{-1}(P)\cap\sqcap$,
and by defining $\phi'$ by sending $\gen'_i\mapsto\psi(\gen_i)$
for $1\leq i\leq k-1$, $\gen'_k\mapsto g$ and $\gen'_{k+1}\mapsto h$. This shows that
$(\xi^{\sqcap})^{-1}(\fc)$ is non-empty.

Now note that, for $\fc'$ as above and for any adapted covering $\uU'$ of $P'$, we have that $\fc'$ is the unique configuration
in $\fU(P';\uU')\subset\Hur(\sqcap)_{0,1;\one}$ supported on the set $P'$. In fact, the normal neighbourhoods $\fU(\bar\fc;\uU')$, for fixed $\uU'$
and varying $\bar\fc$ in $(\xi^{\sqcap})^{-1}(\fc)$, are disjoint: compare with the proof of \cite[Proposition 3.8]{Bianchi:Hur2}. This proves
that $(\xi^{\sqcap})^{-1}(\fc)$ is a discrete topological space.
\end{proof}

\begin{lem}
\label{lem:xisqcapaction}
There is a free action of $G$ on $\Hur(\sqcap;G)_{0,1;\one}$ whose orbits are precisely the fibres of $\xi^{\sqcap}_*$.
\end{lem}
\begin{proof}
Let $\fc'\in\Hur(\sqcap;G)_{0,1;\one}$, and write $\fc'=(P',\psi')$ with $P'=\set{z'_1,\dots,z'_{k+1}}$, where we assume $z'_k=0$, $z'_{k+1}=1$. For $g\in G$ we can define a new configuration $g*\fc'=(P',g*\psi')\in\Hur(\sqcap)_{0,1;\one}$
by setting $g*\psi'(\gen'_i)=\psi'(\gen_i)$ for all $1\leq i\leq k-1$,
$g*\psi'(\gen'_k)=\psi'(\gen'_k) g^{-1}$ and $g*\psi'(\gen'_{k+1})=g \psi'(\gen'_k)$, where $\gen'_1,\dots,\gen'_{k+1}$ is an admissible generating set of $\fG(P')$ as in the proof of Lemma \ref{lem:xisqcapfibres}.

This defines a left action of $G$ on the set $\Hur(\sqcap;G)_{0,1;\one}$.
This action can also be obtained by identifying $\Hur(\sqcap;G)_{0,1;\one}$ with $\Hur(\sqcap,\sqcap;G,G)_{0,1;\one}$ via \cite[Lemma 5.4]{Bianchi:Hur2}, and by considering $(0,(\sqcap,\sqcap),1)$ as a left-right-based nice couple \cite[Definition 6.9]{Bianchi:Hur2}, and by restricting the action of $G\times G^{op}$ on $\Hur(\sqcap,\sqcap;G,G)_{0,1;\one}$ to the diagonal subgroup $G\subset G\times G^{op}$, which leaves the subspace $\Hur(\sqcap,\sqcap;G,G)_{0,1;\one}$ invariant. This proves, in particular, that the action is continuous.

For $\fc'$ as above, we use the notation $\fc=(P,\psi):=\xi^{\sqcap}_*(\fc')$, with $P=\set{z_1,\dots,z_k}$, and assume $z_k=0$ and $z_i=\xi^{\sqcap}(z'_i)$ for all $1\le i\le k-1$. Choose an admissible generating set $\gen_1,\dots,\gen_k$ of $\fG(P)$ as in Lemma \ref{lem:xisqcapfibres}. Then $\psi(\gen_i)=\psi'(\gen'_i)=g*\psi'(\gen'_i)$ for all $1\le i\le k-1$, and
\[
\psi(\gen_k)=\psi'(\gen'_k)\psi'(\gen'_k)=\psi'(\gen'_k)g^{-1}g\psi'(\gen'_k)=\pa{g*\psi'(\gen'_k)}\pa{g*\psi'(\gen'_k)}.
\]
It follows that $\xi^{\sqcap}_*(\fc')=\xi^{\sqcap}_*(g*\fc')$.
\end{proof}

\begin{lem}
 \label{lem:xisqcapopen}
 The map $\xi^{\sqcap}_*$ is open.
\end{lem}
\begin{proof}
 Let $\fc$ and $\fc'$ be as in the proof of Lemma \ref{lem:xisqcapfibres}, i.e. $\xi^{\sqcap}_*\colon\fc'\mapsto \fc$, and let $\uU'$ be an adapted covering of $P'$. We want to find an adapted covering $\uU$ of $P$
 such that $\xi^{\sqcap}_*(\fU(P';\uU'))$ contains $\fU(P,\uU)$. We choose $\uU$ with the following
 properties:
 \begin{itemize}
  \item for all $1\leq i\leq k$ the intersection $U_i\cap\del$ is contractible;
  \item for all $1\leq i\leq k-1$ the intersection $(\xi^{\sqcap})^{-1}(U_i)\cap \sqcap$ is contained
  in $U'_i$;
  \item the intersection $(\xi^{\sqcap})^{-1}(U_k)\cap\sqcap$ is contained in $U'_k\cup U'_{k+1}$.
 \end{itemize} 
 Let $\tfc=(\tP,\tpsi)\in\fU(\fc,\uU)$: we want to find a configuration $\tfc'\in\fU(\fc',\uU')$
 with $\xi^{\sqcap}_*(\tfc')=\tfc$.
 Write $\tP=\set{\tz_1,\dots,\tz_{\tilde k}}$, with $\tz_{\tilde k}=0$.
 Then the finite set $\tP':=(\xi^{\sqcap})^{-1}(\tP)\cap\sqcap$ is contained in $\uU'$, and it intersects non-trivially
 every component of $\uU'$. We write $\tP'=\set{\tz'_1,\dots,\tz'_{\tilde k+1}}$,
 and assume $\tz'_{\tilde k}=0$ and $\tz'_{\tilde k+1}=1$.
 
 Let $\tilde\gen'_1,\dots,\tilde\gen'_{\tilde k+1}$ be an admissible generating set for $\fG(\tP')$
 as in the proof of Lemma \ref{lem:xisqcapfibres}. Note that we can regard $\gen'_k$ and $\gen'_{k+1}$
 as elements of $\fG(\tP')$ by the composition $\fG(P')\cong\fG(\uU')\subset\fG(\tP')$;
 moreover, the sequence of elements $\tilde\gen'_1,\dots,\tilde\gen'_{\tilde k-1},\gen'_k,\gen'_{k+1}$
 is also a free generating set for $\fG(\tP')$, although in general it is not an admissible generating set: in fact, $\gen'_k$ can be decomposed as a product of distinct elements $\tilde\gen'_i$, with one element
 equal to $\tilde\gen'_{\tilde k}$, and similarly $\gen'_{k+1}$ can be decomposed as a product with one factor
 equal to $\tilde\gen'_{\tilde k +1}$.
 Nevertheless we can define a morphism of groups $\tphi'\colon\fG(\tP')\to G$ by setting
 $\tphi'\colon\tilde\gen'_i\mapsto \tphi(\tilde\gen_i)$ for $1\leq i\leq \tilde k-1$,
 and $\tphi'\colon\gen'_i\mapsto\phi'(\gen'_i)$ for $i=k,k+1$. We can restrict $\tphi'$
 to $\fQ_{(\sqcap,\emptyset)}(P')$ and obtain a morphism of PMQs $\tpsi'\colon\fG_{(\sqcap,\emptyset)}(P')\to G$.
 
 We can use the previous to define a configuration
 $\tfc'=(\tP',\tpsi')\in\fU(\fc';\uU')$, satisfying $\xi^{\sqcap}_*(\tfc')=\tfc$.
\end{proof}
In the last step of the proof, note that $\tfc'$ is, in fact, the \emph{unique} configuration
in $\fU(\fc';\uU)$ with $\xi^{\sqcap}_*(\tfc')=\tfc$. This shows, in particular,
that for $\fc$, $\fc'$, $\uU'$ and $\uU$ as in the proof of Lemma \ref{lem:xisqcapopen},
there is a unique map of sets $\fs\colon\fU(\fc;\uU)\to\fU(\fc';\uU')$ which is a section
of $\xi^{\sqcap}_*$, i.e. such that $\xi^{\sqcap}_*\circ\fs$ is equal to the inclusion 
of $\fU(\fc;\uU)$ in $\Hur(\del\cR;G)_{0;\one}$.
  
\begin{lem}
 \label{lem:localsections}
 Let $\fc$, $\fc$, $\uU'$ and $\uU$ be as in the proof of Lemma \ref{lem:xisqcapopen},
 and let $\fs\colon\fU(\fc;\uU)\to\fU(\fc';\uU')$ be the section defined above. Then $\fs$ is continuous.
\end{lem}
\begin{proof}
Let $\tfc\in\fU(\fc;\uU)$, denote $\tfc'=\fs(\tfc)\in\fU(\fc';\uU')$, and use the notation from the proof
of Lemma \ref{lem:xisqcapopen}. By continuity of $\xi^{\sqcap}_*$ there is an adapted covering
$\tilde{\uU}'$ of $\tP'$ with $\tilde{\uU}'\subset\uU'$ and such that $\xi^{\sqcap}_*$ maps
$\fU(\tfc';\tilde{\uU})$ inside $\fU(\fc;\uU)$.

First, note that $\xi^{\sqcap}_*$
is injective on $\fU(\tfc';\tilde{\uU}')$: for a configuration $\tfc''\in\fU(\tfc';\tilde{\uU}')$
we have, in fact, $\fs(\xi^{\sqcap}_*(\tfc''))=\tfc''$.

By Lemma \ref{lem:xisqcapopen} we know that $\xi^{\sqcap}_*$ is open; it follows that
the map $\xi^{\sqcap}_*\colon\fU(\tfc';\tilde{\uU}')\to\Hur(\del\cR;G)_{0;\one}$ is a homeomorphism
onto its image and, hence, $\fs$ is continuous on the open set $\xi^{\sqcap}_*(\fU(\tfc';\tilde{\uU}'))$, which contains $\tfc$.
\end{proof}
\begin{proof}[Proof of Proposition \ref{prop:FNzeroBG}]
The combination of Lemmas \ref{lem:xisqcapfibres}, \ref{lem:xisqcapopen} and \ref{lem:localsections} shows that the map $\xi^{\sqcap}_*\colon\Hur(\sqcap;G)_{0,1;\one}\to\Hur(\del\cR;G)_{0;\one}$
is a covering. Lemma \ref{lem:Hursqcapcontractible} shows that the total space is contractible, in particular connected, and Lemma \ref{lem:xisqcapaction} exhibits $G$ as the group of deck transformations of $\xi^{\sqcap}_*$.
\end{proof}

\begin{nota}
 \label{nota:FNzeroBG}
We denote by $\cB G$ the space
 \[
 \cB G:=\Hur(\del\cR;G)_{0;\one}\cong\Hur(\del\cR,\del\cR;\Q,G)_{0;\one}=\FN_0\bB(\Q_+,G)=\fFN_0\bB(\Q_+,G).
 \]
\end{nota}
 
\subsection{Bundles over \texorpdfstring{$\cB G$}{BG}}
In this subsection we define for all $\nu\geq 0$ a bundle map
$\fp_\nu\colon\fFN_\nu\bB(\Q_+,G)\to \cB G$; the fibre of $\fp_\nu$
can be identified with 
\[
\coprod_{a\in\Q_\nu}\Hur(\mcR;\Q_+)_a,
\]
where $\Q_\nu\subset\Q$ is the subset of elements of norm $\nu$.
In the case $\nu=0$, the map $\fp_0$ is just the identity of $\cB G=\fFN_0\bB(\Q_+,G)$,
and the fibre is a point, i.e. the space $\Hur(\mcR;\Q_+)_{\one}$.

In the next section we will investigate the rational cohomology of $\bB(\Q_+,G)$ using
the Leray spectral sequence associated with the filtration
$\FN_{\bullet}=\FN_{\bullet}\bB(\Q_+,G)$: the first page of this spectral sequence
contains the \emph{relative} cohomology groups $H^*(\FN_\nu,\FN_{\nu-1})$,
rather than the cohomology groups of the strata $\fFN_\nu$.
To acquire information about these relative cohomology groups, we introduce in this subsection certain subspaces $\FNfat_{\nu-1}=\FNfat_{\nu-1}\bB(\Q_+,G)\subset\bB(\Q_+,G)$, for $\nu\geq 0$.
We will prove, between this subsection and the next section, that
$\FNfat_{\nu-1}\subset\FN_{\nu}$,
that $\FN_{\nu-1}$ is contained in the interior of $\FNfat_{\nu-1}$ when the latter is regarded as 
a subspace of $\FN_{\nu}$, and that the inclusion $\FN_{\nu-1}\subset\FNfat_{\nu-1}$ is a
homotopy equivalence.
In particular, after setting $\delfat\fFN_{\nu}:=\fFN_{\nu}\cap\FNfat_{\nu-1}$,
we will obtain in Lemma \ref{lem:excision} an isomorphism
\[
 H^*(\FN_{\nu},\FN_{\nu-1})\cong H^*(\fFN_{\nu},\delfat\fFN_{\nu}).
\]
In this subsection we will prove that $\fp_\nu$ exhibits $(\fFN_{\nu},\delfat\fFN_\nu)$ as a couple of bundles over $\cB G$,
with fibre a suitable couple of spaces
\[
\pa{\Hur(\mcR;\Q_+)_\nu,\delfat\Hur(\mcR;\Q_+)_\nu}=
\pa{\coprod_{a\in\Q_\nu}\Hur(\mcR;\Q_+)_a, \coprod_{a\in\Q_\nu}\delfat\Hur(\mcR;\Q_+)_a};
\]
the computation of the rational cohomology $H^*(\fFN_{\nu},\delfat\fFN_\nu;\bQ)$ will then be possible using the Serre spectral sequence associated with $\fp_\nu$.
\begin{defn}
 \label{defn:delfat}
 Recall Definition \ref{defn:FN}. Denote by $\delfat\cR$ the
 closed neighbourhood of $\del\cR$ in $\cR=[0,1]^2$ given by
 $\delfat\cR=\cR\setminus\pa{\frac{\zcentre}2+\frac 12\mcR}$,
 where $\frac{\zcentre}2+\frac 12\mcR$ is the image of $\mcR$
 along the homothety centred at $\zcentre$ of rescaling factor $\frac 12$;
 in other words, $\frac{\zcentre}2+\frac 12\mcR$ is the open square
 of side length $\frac 12$ centred at $\zcentre$.

 The identity of $\C$ induces a map
 \[
  (\Id_\C)_*\colon \Hur(\cR,\del;\Q,G)\to\Hur(\cR,\delfat\cR;\Q,G).
 \]
Recall that $\Hur(\cR,\delfat\cR;\Q,G)$ has a filtration by subspaces $\FN_\nu\Hur(\cR,\delfat\cR;\Q,G)$ for $\nu\ge -1$;
we define $\FNfat_\nu\bB(\Q_+,G)$ as the intersection
\[
\FNfat_\nu\bB(\Q_+,G):=(\Id_\C)_*^{-1}\pa{\FN_\nu\Hur(\cR,\delfat\cR;\Q,G)}\ \cap\  \FN_{\nu+1}\bB(\Q_+,G).
\]
\end{defn}
In particular, we have $\FNfat_{-1}\bB(\Q_+,G)=\emptyset$.
Roughly speaking, for $\fc\in \bB(\Q_+,G)$,
we can use Notation \ref{nota:fc},
and suppose without loss of generality that $\set{z_1,\dots,z_{l'}}=P\cap (1/4,3/4)^2$ for some $0\le l'\le l$;
if $\gen_1,\dots,\gen_k$ is an admissible generating set for $\fG(P)$,
then $\fc$ belongs to $\FNfat_\nu\bB(\Q_+,G)$ if the following hold: 
\begin{itemize}
 \item $\sum_{i=1}^l N(\psi(\gen_i))\le\nu+1$, that is, $\fc\in\FN_{\nu+1}\bB(\Q_+,G)$;
 \item $\sum_{i=1}^{l'} N(\psi(\gen_i))\le\nu$.
\end{itemize}
Another characterisation is the following: $\FNfat_\nu\bB(\Q_+,G)$ is the preimage of the space $\FN_\nu\bB(\Q_+,G)$ along the
restricted map $\cH^{\bB}(-,2)\colon \FN_{\nu+1}\bB(\Q_+,G)\to\bB(\Q_+,G)$, see Definition \ref{defn:cHbB}. Note that, in fact, $\cH^{\bB}(-,2)$ restricts to a self-map of $\FN_{\nu+1}\bB(\Q_+,G)$.
By construction, we have inclusions
\[
 \FN_\nu\bB(\Q_+,G)\subset\FNfat_\nu\bB(\Q_+,G)\subset\FN_{\nu+1}\bB(\Q_+,G).
\]
\begin{nota}
 \label{nota:delfat}
 For a subspace $X\subseteq \bB(\Q_+,G)$ and $\nu\ge -1$ we denote by $\FNfat_\nu X$
 the intersection $X\cap \FNfat_\nu\bB(\Q_+,G)$.
 For $\nu\ge 0$ we denote by $\delfat\fFN_\nu X$ the intersection $\fFN_\nu X\cap \FNfat_{\nu-1}X$.
\end{nota}
\begin{ex}
 Let $a\in\Q$ and let $X=\Hur(\mcR;\Q_+)_a$; the inclusion of nice couples $(\mcR,\emptyset)\subset (\cR,\del)$ induces an inclusion of spaces $X\subset\bB(\Q_+,G)$; let $\nu= N(a)\ge0$, and note that $F_{\nu-1}X=\emptyset$ and, hence,
 $X=\FN_\nu X=\fFN_\nu X$.
 The space $\FNfat_{\nu-1} X$ contains all configurations $\fc\in X$ such that, using Notation \ref{nota:fc}, $P$
 intersects non-trivially $\delfat\cR$: in fact the condition $\fc\in \Hur(\mcR;\Q_+)$ implies that the monodromy
 $\psi$ attains values of norm $\ge1$ (i.e. different from $\one\in\Q$) 
 around all points of $P$.
 The space $\delfat\fFN_\nu X$ coincides with $\FNfat_{\nu-1} X$ in this case, and by abuse of notation we will also
 write
 \[
  \delfat \Hur(\mcR;\Q_+)_a=\delfat\fFN_\nu \Hur(\mcR;\Q_+)_a.
 \]
\end{ex}
\begin{ex}
 Let $X=\bB(\Q_+,G)$; then $\FNfat_{-1} X=\delfat\fFN_0 X=\emptyset$ and $\fFN_0 X=\cB G$;
 hence the identity of $\cB G$
 can be regarded as a pair of bundles
 \[
  \fp_0\colon \pa{\,\fFN_0 \bB(\Q_+,G)\,,\,\delfat\fFN_0\bB(\Q_+,G)\,}\to \cB G
 \]
 with fibre the couple $(\Hur(\mcR;\Q_+)_{\one}, \delfat \Hur(\mcR;\Q_+)_{\one})=(\set{(\emptyset,\one,\one)},\emptyset)$.
In the rest of the subsection we generalise this example to the other strata of $\bB(\Q_+,G)$.
\end{ex}

Fix in the following $\nu\ge1$, let $\fc\in\fFN_\nu\bB(\Q_+,G)$ and use Notation \ref{nota:fc}.
Let $\arc_1,\dots,\arc_k$ be embedded arcs in $\C$ joining $*$ with the points $z_1,\dots,z_k$ of $P$, intersecting pairwise only at the endpoint $*$, and such that $\arc_i\subset\C\setminus\mcR$ for $l+1\le i\le k$.
Recall that, since $\fc\in\bB(\Q_+,G)$, the point $0$ belongs to $P$; we use the convention that $z_k=0$. Let $\gen_1,\dots,\gen_k$ be the admissible generating set of $\fG(P)$ obtained by replacing each $\zeta_i$ by a loop contained in a small neighbourhood of $\zeta_i$ and spinning clockwise only around $z_i$.

We define a new configuration $\fc'=(P',\psi')\in \cB G=\Hur(\del\cR;G)_{0;\one}$ as follows:
\begin{itemize}
 \item $P'$ is the intersection $P\cap \del\cR$;
 \item $\psi'$ sends, for $l+1\le i\le k-1$, the generator $\gen_i$ to $\phi(\gen_i)\in G$,
 and it sends $\gen_k$ to the unique element $\psi'(\gen_k)\in G$ such that the resulting configuration
 $\fc'=(P',\psi')$ satisfies $\totmon(\fc')=\one\in G$.
\end{itemize}
Note that $G$ is treated as a complete PMQ when defining $\fc'$. See Figure \ref{fig:fp}.

 \begin{figure}[ht]
 \begin{tikzpicture}[scale=4,decoration={markings,mark=at position 0.38 with {\arrow{>}}}]
  \draw[dashed,->] (-.05,0) to (1.1,0);
  \draw[dashed,->] (0,-1.1) to (0,1.1);
  \node at (0,-1) {$*$};
  \draw[very thick, fill=black!40!white] (0,0) rectangle (1,1);
  \fill[white, opacity=.4] (.25,.25) rectangle (.75, .75);
  \node at (0,0){$\bullet$};
  \draw[thin, looseness=1.2, postaction={decorate}] (0,-1) to[out=91,in=-90]  (-.03,-.1) to[out=90,in=-90] (-.07,.05)  to[out=90,in=90] node[left]{\tiny$g_6$} (.03,.05)  to[out=-90,in=90] 
  (.01,-.1) to[out=-90,in=89] (0,-1);
  \node at (.1,.3){$\bullet$}; 
  \node at (.3,1){$\bullet$}; 
  \node at (.45,.5){$\bullet$};
  \draw[thin, looseness=1.2, postaction={decorate}] (0,-1) to[out=82,in=-90]  (.05,.1) to[out=90,in=-90] (.01,.3)  to[out=90,in=90] node[above]{\tiny$a_1$} (.2,.3)  to[out=-90,in=90] (.1,.1) to[out=-90,in=80] (0,-1);
  \draw[thin, looseness=1.2, postaction={decorate}] (0,-1) to[out=79,in=-90] (.28,.9)
  to[out=90,in=-90] (.26,1) to[out=90,in=90] (.34,1) to[out=-90,in=90] node[right]{\tiny$g_4$}
  (.3,.9) to[out=-90,in=77]  (0,-1);
  \draw[thin, looseness=1.2, postaction={decorate}] (0,-1) to[out=76,in=-90] (.4,.3)
  to[out=90,in=-90] (.35,.5) to[out=90,in=90] node[above]{\tiny$a_2$} (.49,.5) to[out=-90,in=90] (.45,.3) to[out=-90,in=74]  (0,-1);
  \node at (.65,.0){$\bullet$};
  \node at (.95,.8){$\bullet$};
  \draw[thin, looseness=1.2, postaction={decorate}] (0,-1) to[out=62,in=-120]  (.55,-.1) to[out=60,in=-90] (.52,.05)  to[out=90,in=90] node[above]{\tiny$g_5$} (.7,.05)  to[out=-90,in=50] (.65,-.1) to[out=-130,in=60] (0,-1);
  \draw[thin, looseness=1.2, postaction={decorate}] (0,-1) to[out=56,in=-90] (.9,.7)
  to[out=90,in=-90] (.85,.8) to[out=90,in=90] node[above]{\tiny$a_3$} (.99,.8) to[out=-90,in=90] (.95,.7) to[out=-90,in=54]  (0,-1);
  \begin{scope}[shift={(1.4,0)}]
     \draw[dashed,->] (-.05,0) to (1.1,0);
  \draw[dashed,->] (0,-1.1) to (0,1.1);
  \node at (0,-1) {$*$};
  \draw[very thick] (0,0) rectangle (1,1);
  \node at (0,0){$\bullet$};
  \draw[thin, looseness=1.2, postaction={decorate}] (0,-1) to[out=91,in=-90]  (-.03,-.1) to[out=90,in=-90] (-.07,.05)  to[out=90,in=90] node[right]{\tiny$g'_3$} (.03,.05)  to[out=-90,in=90] 
  (.01,-.1) to[out=-90,in=89] (0,-1);
  \node at (.3,1){$\bullet$}; 
  \draw[thin, looseness=1.2, postaction={decorate}] (0,-1) to[out=79,in=-90] (.28,.9)
  to[out=90,in=-90] (.26,1) to[out=90,in=90] (.34,1) to[out=-90,in=90]
  node[right]{\tiny$g'_1$}
  (.3,.9) to[out=-90,in=77]  (0,-1);
  \node at (.65,.0){$\bullet$};
  \draw[thin, looseness=1.2, postaction={decorate}] (0,-1) to[out=62,in=-120]  (.55,-.1) to[out=60,in=-90] (.52,.05)  to[out=90,in=90] node[above]{\tiny$g'_2$} (.7,.05)  to[out=-90,in=50] (.65,-.1) to[out=-130,in=60] (0,-1);
  \end{scope}
 \end{tikzpicture}
 \caption{The configuration $\fc$ from Figure \ref{fig:bB} belongs to $\delfat\fFN_\nu$, where $\nu=N(a_1)+N(a_2)+N(a_3)$, because the two points $z_1$ and $z_3$ lie in $\delfat\cR$. On the right, the image of $\fc$ along $\fp_\nu$; note that the loop with label $g_4$ on the left is \emph{not} constructed using an arc contained in $\C\setminus\mcR$; as a consequence the monodromy on right around the ``same'' loop is changed by conjugation:
 we have in fact $g'_1=g_4^{\fe(a_2)\pa{\fe(a_3)^{g_5^{-1}}}}$, $g'_2=g_5$ and
 $g'_3=g_6\fe(a_1)\fe(a_2)g_5\pa{\fe(a_3)^{g_5^{-1}}}$.
 }
\label{fig:fp}
\end{figure}

\begin{defn}
 \label{defn:fpnu}
 The previous assignment $\fc\mapsto\fc'$ defines a map of sets
 \[
\fp_\nu\colon\fFN_\nu\bB(\Q_+,G)\to \cB G.
 \]
\end{defn}
To check that the previous is a good definition, we need to verify that the choice of the arcs $\arc_1,\dots,\arc_k$
is not relevant in computing $\fp_\nu(\fc)$. The generator $\gen_i$ is uniquely defined up to conjugation
by a power of the element $[\gamma]\in\fG(P)$ represented by a loop $\gamma$ spinning clockwise around $\cR$.
It follows that $\psi(\gen_i)$ is well-defined, as an element of $G$, up to conjugation by a power of $\psi([\gamma])$,
i.e. up to conjugation by a power of $\totmon(\fc)=\one\in G$: here we use that the total monodromy attains constantly the value
$\one\in G$ on configurations of $\bB(\Q_+,G)$; this shows that $\psi(\gen_i)\in G$ is well-defined
for $1\le i\le k-1$, and $\psi(\gen_k)$ is also uniquely determined by the values $\psi(\gen_i)$ for $1\le i\le k-1$
and by its characterising property. Therefore, $\fp_\nu(\fc)$ is well-defined.

We next check that $\fp_\nu\colon\fFN_\nu\bB(\Q_+,G)\to \cB G$ is continuous. Roughly speaking, $\fp_\nu$ splits a configuration in $\fFN_\nu\bB(\Q_+,G)$ in two parts, the part supported on $\del\cR$ and the part supported on $\mcR$, and it pushes all points in the second part to $0$, thus giving rise to a new configuration supported only on $\del\cR$. Continuity of $\fp_\nu$ depends on the fact that if we perturb a configuration \emph{staying inside the stratum $\fFN_\nu\bB(\Q_+,G)$}, then no point in $\mcR$ can collide with $\del\cR$, and vice versa no point in $\del\cR$ can move in the interior, as this would let the internal total norm jump (down and up, respectively). Here it is important that, working with $\bB(\Q_+,G)$, the local monodromy of a point lying in $\mcR$ is required to have positive norm.

Formally, let $\fc\in \fFN_\nu\bB(\Q_+,G)$ and let $\uU$ be an adapted covering of $P$: in particular,
we have $z_i\in U_i\subset\mcR$ for all $1\le i\le l$.
Denote by $\fc'=\fp_\nu(\fc)\in \cB G$, and let $\uU'$ be the restricted,
adapted covering of $P'$, i.e. $\uU'=(U_{l+1},\dots, U_k)$.
Then $\fp_\nu$ sends the intersection $\fU(\fc;\uU)\cap \fFN_\nu\bB(\Q_+,G)$ inside 
$\fU(\fc';\uU')\subset \cB G$: this essentially follows from the observation that, for $\tfc=(\tP,\tpsi,\tphi)\in\fU(\fc;\uU)\cap \fFN_\nu\bB(\Q_+,G)$, we have $\tP\subset U_1\cup\dots\cup U_l\cup \pa{\uU'\cap\del \cR}$.

\begin{nota}
 \label{nota:Hurnu}
Let $\nu\ge0$, let $\bT\subset\C$ be a contractible subspace containing $*$, let 
$\cX\subset\mathring{\bT}$ be a semialgebraic subspace
and let $X$ be a subspace of $\Hur^{\bT}(\cX;\Q)$.
We denote $X_\nu=\coprod_{a\in\Q_\nu} \pa{X\cap\Hur^{\bT}(\cX;\Q)_a}$.
\end{nota}

\begin{prop}
 \label{prop:fpnubundle}
The map $\fp_\nu\colon\fFN_\nu\bB(\Q_+,G)\to \cB G$ is a bundle map with fibre
$\Hur(\mcR;\Q_+)_\nu$.
The restricted map $\fp_\nu\colon\delfat\fFN_\nu\bB(\Q_+,G)\to \cB G$ is also a bundle map with fibre
$\delfat\Hur(\mcR;\Q_+)_\nu$.
The two bundles admit compatible local trivialisations, i.e. they form a couple of bundles.
\end{prop}
\begin{proof}
 Choose a small \emph{closed} interval $J\subset(0,1)\times\set{0}\subset \del\cR$, and choose
 an arc $\arc_J$ joining $*$ with the midpoint
 of $J$. Let $\bT$ be the union $\bT=\mcR\cup J\cup\arc_J\subset\C$, and note that $\bT$ is contractible and contains
 $\mcR$ in its interior. We can define a map of sets
 \[
 \fri_{\bT}^{\C}\colon\fFN_\nu\bB(\Q_+,G)\cap \Hur(\cR\setminus J,\del\cR\setminus J;\Q,G)
 \to\Hur^{\bT}(\mcR;\Q_+)_\nu,
 \]
 in the spirit of \cite[Definition 3.15]{Bianchi:Hur2}.
 To define this map, let $\fc\in\fFN_\nu\bB(\Q_+,G)\cap\Hur(\cR\setminus J,\del\cR\setminus J;\Q,G)$:  using Notation \ref{nota:fc},
 this means that $\fc\in\fFN_\nu\bB(\Q_+,G)$ and the support $P$ of $\fc$ does not intersect $J$.
 The inclusion $\bT\setminus P\subset\CmP$ gives rise to an inclusion of groups $\fG^{\bT}(P\cap\mcR)\subset\fG(P)$
 and, by restriction, an inclusion of PMQs $\fQ^{\bT}_{(\mcR,\emptyset)}(P\cap \mcR)\subset\fQ_{(\cR,\del)}(P)$.
 We define $\fri_{\bT}^{\C}(\fc)$ to be the configuration $\fc'=(P',\psi')$, where $P'=P\cap\mcR$ and $\psi'\colon\fQ^{\bT}_{(\mcR,\emptyset)}(P')\to\Q$
 is the composition of the above inclusion with $\psi\colon\fQ_{(\cR,\del)}(P)\to\Q$.

 To show that $\fri_{\bT}^{\C}$ is continuous at $\fc\in\fFN_\nu\bB(\Q_+,G)\cap\Hur(\cR\setminus J,\del\cR\setminus J;\Q,G)$,
 let $\uU$ be an adapted covering of $P$ with $\uU\subset\C\setminus J$, let $\fc'=\fri^{\C}_{\bT}(\fc)$ as above, and
 let $\uU'$ be the restriction of $\uU$ to $P'\subset P$, i.e. $\uU'$ consists of those components of $\uU$ that intersect
 non-trivially $P'$; equivalently, $\uU'$ consists of those components of $\uU$ that are contained in $\mcR$.
 
 We claim that $\fri^{\C}_{\bT}$ sends $\fU(\fc,\uU)\cap\fFN_\nu\bB(\Q_+,G)$ inside
 $\fU(\fc',\uU')\subset \Hur^{\bT}(\mcR;\Q_+)_\nu$; since every small enough adapted
 covering $\uU'$ of $P'$ with respect to $(\mcR,\emptyset)$ can be extended to an
 adapted covering $\uU$ of $P$ with respect to $(\cR\setminus J,\del\cR\setminus J)$,
 the claim suffices to prove continuity of $\fri^{\C}_{\bT}$.
 
 For the claim, let $\hat\fc=(\hat P,\hat\psi,\hat\phi)\in\fU(\fc,\uU)\cap\fFN_\nu\bB(\Q_+,G)$;
 fix an admissible generating set $\gen_1,\dots,\gen_k$
 of $\fG(P)$ extending an admissible generating set $\gen_1,\dots,\gen_l$ of $\fG^{\bT}(P')$, and regard
 $\gen_1,\dots,\gen_k$ as elements of $\fG(\hat P)$ by the inclusion $\fG(P)\cong\fG(\uU)\subset\fG(\hat P)$.
 Let $\hat P_i=\hat P\cap U_i$ for all $1\le i\le k$, and write $\hat P_i=\set{z_{i,1},\dots,z_{i,k_i}}$.
 Choose an admissible generating set $(\hat\gen_{i,j})_{1\le i\le k,1\le j\le k_i}$ of $\fG(\hat P)$,
 such that the equality $\gen_i=\hat\gen_{i,1}\cdot\dots\cdot\hat\gen_{i,k_i}$ holds in $\fG(\hat P)$ for all $1\le i\le k$. 
 The hypothesis on $\hat\fc$ implies that for all $1\le i\le l$ the product
 $\hat\psi(\hat\gen_{i,1})\dots\hat\psi(\hat\gen_{i,k_i})$ is defined in $\Q$ and is equal to $\psi(\gen_i)$,
 in particular $\sum_{j=1}^{k_i}N(\hat\psi(\hat\gen_{i,j}))=N(\psi(\gen_i))$. Summing over $1\le i\le l$,
 and recalling that $\hat\fc\in \fFN_\nu\bB(\Q_+,G)$, we obtain the equality
 $\sum_{i=1}^l\sum_{j=1}^{k_i}N(\hat\psi(\hat\gen_{i,j}))=\nu$.
 This implies that $\hat P$ can
 only intersect the open sets $U_{l+1},\dots,U_k$ in points of $\del\cR$ or in \emph{inert} points for $\hat\fc$;
 since $\hat\fc$ has no inert point, we have that $\hat P$ is contained in the union $\del\cR\cup \uU'$.
 It follows that $\hat\fc'=\fri^{\C}_{\bT}(\hat\fc)$ is supported on $\uU'$; the fact that
 $\hat\fc'$ is contained in $\fU(\fc',\uU')$ follows now directly from the definition of $\fri^{\C}_{\bT}$
 and from the already mentioned equalities
 $\hat\psi(\hat\gen_{i,1})\cdot\dots\cdot\hat\psi(\hat\gen_{i,k_i})=\psi(\gen_i)$ for $1\le i\le l$.
 
 Now note that the intersection $\fFN_\nu\bB(\Q_+,G)\cap \Hur(\cR\setminus J,\del\cR\setminus J;\Q,G)$ is precisely the preimage along $\fp_\nu$ of the open subspace $\Hur(\del\cR\setminus J;G)_{0;\one}\subset\cB G$. The product map $\fp_\nu\times\fri^{\C}_{\bT}$ gives a homeomorphism
 \[
  \fp_\nu\times\fri^{\C}_{\bT}\ \colon \fp_\nu^{-1}(\Hur(\del\cR\setminus J;G)_{0;\one})\cong \Hur(\del\cR\setminus J;G)_{0;\one}\times \Hur^{\bT}(\mcR;\Q_+)_\nu.
 \]
 Since the open sets $\Hur(\del\cR\setminus J;G)_{0;\one}$ form an open covering
 of $\cB G$, for varying $J$, we obtain that
 $\fp_\nu$ is a bundle map, i.e it admits local trivialisations. The fibre of the bundle
 is homeomorphic to the space $\Hur(\mcR;\Q_+)_\nu$.
 
 The local trivialisation $\fp_\nu\times\fri^{\C}_{\bT}$ restricts to a local trivialisation of 
 the restriction of $\fp_\nu$ to $\delfat\fFN_\nu\bB(\Q_+,G)\subset\fFN_\nu\bB(\Q_+,G)$,
 with restricted fibre $\delfat\Hur(\mcR;\Q_+)_\nu$.
\end{proof}

\section{Rational cohomology}
\label{sec:cohomology}
In this section we assume that $\Q$ is a finite, $\bQ$-Poincar\'e PMQ and $G$ is a finite group,
and compute the rational cohomology of $\bB(\Q_+,G)$. Recall from \cite[Definition]{Bianchi:Hur2} that a PMQ $\Q$ is $\bQ$-Poincar\'e
if $\Q$ is locally finite and for all $a\in\Q$ the space $\Hur(\mcR;\Q_+)_a$ is a
$\bQ$-homology manifold of some dimension: in this case $\Q$ admits an intrinsic norm $h\colon\Q\to\N$
and $\Hur(\mcR;\Q_+)_a$ is an orientable $\bQ$-homology manifold of dimension $2h(a)$ for all $a\in\Q$. See \cite[Proposition 9.7]{Bianchi:Hur2}.

Our interest for the space $\bB(\Q_+,G)$ and its cohomology comes from Theorem \ref{thm:mainHiso}
and Proposition \ref{prop:bBHur}, relating $\bB(\Q_+,G)$ to the topological monoid $\mHurm(\Q)$.
Note that, for a fixed PMQ $\Q$, we are free to choose a group $G$ completing $\Q$ to a PMQ-group pair $(\Q,G,\fe,\fr)$. If $\Q$ is finite we can, for instance, take $G=\cG(\Q)/\cK(\Q)$, which is a finite group; here $\cK(\Q)$ denotes the kernel of the map $\rho\colon\cG(\Q)\to\Aut_{\PMQ}(\Q)^{op}$, giving the right action of $\cG(\Q)$ on $\Q$ by conjugation: if $\Q$ is finite, then $\Aut_{\PMQ}(\Q)^{op}$ is also finite and contains a subgroup isomorphic to $\cG(\Q)/\cK(\Q)$. See also \cite[Lemma 2.13]{Bianchi:Hur1}.
Thus, if we are given a \emph{finite}, $\bQ$-Poincar\'e PMQ $\Q$,
we can complete $\Q$ to a PMQ-group pair by adjoining a suitable finite group $G$.

Recall from \cite[Definition 4.29]{Bianchi:Hur1} that $\cA(\Q)\subset\bQ[\Q]$ is defined as the subring of the
PMQ-group ring $\bQ[\Q]$ consisting of the invariants under conjugation by $G$:
 \[
  \cA(\Q)=\bQ[\Q]^G.
 \]
As a $\bQ$-vector space, $\cA(\Q)$ is spanned by elements $\sca{S}=\sum_{a\in S}\sca{a}$,
for each conjugacy class $S\subset\Q$. In this section we consider
$\bQ[\Q]^G$ as a graded, associative ring, by putting the generator $\sca{a}$ in degree $2h(a)$,
for all $a\in\Q$; similarly
$\cA(\Q)$ is a graded ring with
$\sca{S}$ in degree $2h(a)$, for any $a\in S$. By \cite[Lemma 4.31]{Bianchi:Hur1} the ring
$\cA(\Q)$ is a commutative ring, hence by our choice of degrees it is a graded-commutative
ring, supported in even degrees. We will prove the following theorem.
\begin{thm}
 \label{thm:HbBQG}
 Let $(\Q,G)$ be a PMQ-group pair with $\Q$ finite and $\bQ$-Poincar\'e and with $G$ finite.
 Then there is an isomorphism of rings
 \[
  H^*(\bB(\Q_+,G);\bQ)\cong \cA(\Q).
 \]
\end{thm}
In this entire section we use the abbreviation $\bB=\bB(\Q_+,G)$.
\subsection{Two spectral sequence arguments}
Since the space $\bB$ is equipped with a filtration by subspaces $\FN_\nu\bB$,
we can compute $H^*(\bB;\bQ)$ by the associated Leray spectral sequence, whose first page reads
$E_1^{p,\nu}=H^{p+\nu}(\FN_\nu\bB,\FN_{\nu-1}\bB)$.
\begin{lem}
\label{lem:excision}
 For $\nu\ge0$ the inclusion $\FN_{\nu-1}\bB\subset\FNfat_{\nu-1}\bB$ is a homotopy equivalence.
 Moreover we have cohomology isomorphisms
 \[
  H^*(\FN_\nu\bB,\FN_{\nu-1}\bB;\bQ)\cong
  H^*(\FN_\nu\bB,\FNfat_{\nu-1}\bB;\bQ)\cong
  H^*(\fFN_\nu\bB,\delfat\fFN_{\nu}\bB;\bQ).
 \]
\end{lem}
\begin{proof}
 Recall Definitions \ref{defn:cHbB} and \ref{defn:intwindow}.
 For $\fc\in\FN_\nu\bB$ let $\mathfrak{W}_\nu(\fc)$ be the supremum in $[1/2,1]$ of all $1/2<\epsilon<1$
 for which $\cH^{\bB}_*(\fc,1/\epsilon)\in\FN_{\nu-1}\bB$. We define a homotopy
 \[
  H^{\bB}_\nu\colon\FN_\nu\bB\times[0,1]\to\FN_\nu\bB,\quad
  (\fc,s)\mapsto \cH^{\bB}_*\pa{\fc , 1-s+\frac{s}{\mathfrak{W}_\nu(\fc)}}.
 \]
 The homotopy $H^{\bB}_\nu$ restricts to a deformation retraction of $\FNfat_{\nu-1}\bB$
 onto $\FN_{\nu-1}\bB$, whence the first cohomology isomorphism follows.
 The second cohomology isomorphism follows from excision: in fact,
 $\FN_{\nu-1}\bB=(\mathfrak{W}_\nu)^{-1}(1)$
 and the open set $(\mathfrak{W}_\nu)^{-1}((1/2,1])$ is contained in $\FNfat_{\nu-1}(\bB)$,
 so we can apply excision.
\end{proof}
We can now focus on the relative cohomology groups $H^*(\fFN_\nu\bB,\delfat\fFN_{\nu}\bB;\bQ)$.
\begin{prop}
\label{prop:Serre}
For $\nu\ge0$, the cohomology groups $H^*(\fFN_\nu\bB,\delfat\fFN_\nu\bB;\bQ)$ are concentrated in degree $*=2\nu$;
the group $H^{2\nu}(\fFN_\nu\bB,\delfat\fFN_{\nu}\bB;\bQ)$ is isomorphic to $\cA(\Q)_{2\nu}$, i.e.
the degree-$2\nu$ part of $\cA(\Q)$.
\end{prop}
\begin{proof}
We use the Serre spectral sequence $\mathcal{E}(\nu)$ associated with the couple of bundles
$\fp_\nu\colon(\fFN_\nu\bB,\delfat\fFN_{\nu}\bB)\to\cB G$: its second page reads
\[
 \mathcal{E}(\nu)_2^{p,q}=H^p\pa{\cB G;H^q\pa{\Hur(\mcR;\Q_+)_\nu, \delfat\Hur(\mcR;\Q_+)_\nu;\bQ}}.
\]
The first step is to compute $H^*(\Hur(\mcR;\Q_+)_a,\delfat\Hur(\mcR;\Q_+)_a;\bQ)$ for $a\in\Q_\nu$,
and the argument for this will be similar to the proof of Lemma \ref{lem:excision}.
Consider the closed unit square $\cR$ and
define $\delfat\Hur(\cR;\Q_+)_a:=\Hur(\cR;\Q_+)_a\setminus \Hur((\frac {\zcentre}2+\frac 12\mcR);\Q_+)_a$,
and $\del\Hur(\cR;\Q_+)_a:=\Hur(\cR;\Q_+)_a\setminus \Hur(\mcR;\Q_+)_a$.
The subspace $\delfat\Hur(\cR;\Q_+)_a\subset \Hur(\cR;\Q_+)_a$ contains configurations
whose support intersects $\delfat\cR$ (see Definition \ref{defn:delfat}), whereas $\del\Hur(\cR;\Q_+)_a\subset \Hur(\cR;\Q_+)_a$ contains configurations whose support intersects $\del\cR$.

Recall Definition \ref{defn:cHbB}: for all $s\ge1$ the map $\cH^{\bB}(-,s)\colon\C\to\C$
is a lax endomorphism of the nice couple $(\cR,\emptyset)$; if we consider
$\Hur(\cR;\Q_+)_a$ as a connected component of $\Hur(\cR;\hQ_+)$, under the inclusion $\Q\subset\hQ$,
we obtain by functoriality a homotopy
\[
 \cH^{\bB}_*\colon\Hur(\cR;\Q_+)_a\times[1,\infty)\to\Hur(\cR;\Q_+)_a.
\]
For $\fc\in\Hur(\cR;\Q_+)_a$ denote by $\mathfrak{W}_a(\fc)$ the supremum in
$[1/2,1]$ of all $1/2<\epsilon<1$ for which $\cH^{\bB}_*(\fc,1/\epsilon)\in\del\Hur(\cR;\Q_+)_a$.
We define a homotopy
\[
H^{\bB}_a\colon\Hur(\cR;\Q_+)_a\times[0,1]\to\Hur(\cR;\Q_+)_a,\quad\quad
(\fc,s)\mapsto \cH^{\bB}_*\pa{\fc\,,\,1-s+\frac{s}{\mathfrak{W}_a(\fc)}}.
\]
The homotopy $H^{\bB}_a$ restricts to a deformation retraction of $\delfat\Hur(\cR;\Q_+)_a$
onto $\del\Hur(\cR;\Q_+)_a$; moreover the subspace $\del\Hur(\cR;\Q_+)_a$ of $\Hur(\cR;\Q_+)_a$ is contained in the interior
of $\delfat\Hur(\cR;\Q_+)_a$, and
\[
\delfat\Hur(\cR;\Q_+)_a\setminus\del\Hur(\cR;\Q_+)_a=\delfat\Hur(\mcR;\Q_+)_a.
\]
We thus obtain cohomology isomorphisms
\[
\begin{split}
 H^*(\Hur(\mcR;\Q_+)_a,\delfat\Hur(\mcR;\Q_+)_a;\bQ) \cong &  H^*(\Hur(\cR;\Q_+)_a,\delfat\Hur(\cR;\Q_+)_a;\bQ)
\\[.2cm]
 \cong & H^*(\Hur(\cR;\Q_+)_a,\del\Hur(\cR;\Q_+)_a;\bQ). 
\end{split}
\]
Since $\Q$ is $\bQ$-Poincar\'e, as a consequence of \cite[Lemma 9.5, Proposition 9.7]{Bianchi:Hur2} the cohomology
$H^*(\Hur(\cR;\Q_+)_a,\del\Hur(\cR;\Q_+)_a;\bQ)$ vanishes in degrees $*\neq 2\nu=2h(a)$,
and it is equal to $\bQ$ in degree $*=2\nu$. Going back to the Serre spectral sequence,
the group $\mathcal{E}(\nu)_2^{p,q}$ vanishes for $q\neq 2\nu$, and
$\mathcal{E}(\nu)_2^{p,2\nu}$ is equal to the twisted cohomology group
$H^p\pa{\cB G;\oplus_{a\in\Q_\nu}\bQ}$; this already shows that the spectral sequence
collapses on its second page. Moreover, since $G$ is a finite group and we are considering
twisted cohomology with coefficients in a $G$-representation over $\bQ$, all cohomology
groups except possibly $H^0$ vanish, i.e. the entire page $\mathcal{E}(\nu)_2$ vanishes
except possibly $\mathcal{E}(\nu)_2^{0,2\nu}=H^0\pa{\cB G;\bigoplus_{a\in\Q_\nu}\bQ}$.

The action of $G$ on $\bigoplus_{a\in\Q_\nu}\bQ$ is the $\bQ$-linearisation of the action of $G$
on the set $\Q_\nu$ by conjugation; the invariants of the 
action of $G$ on $\bigoplus_{a\in\Q_\nu}\bQ$
are therefore the $\bQ$-vector space spanned by conjugacy classes of $\Q$ of norm $\nu$: this vector
space is isomorphic to the degree-$2\nu$ part of $\cA(\Q)$.
\end{proof}
Proposition \ref{prop:Serre} implies that the $E_1$-page of the Leray spectral sequence
associated with the filtered space $\bB$ is supported on the main diagonal, i.e. $E_1^{p,\nu}=0$
whenever $p\neq\nu$. This implies that the spectral sequence collapses on its first page
and, thus, we obtain an isomorphism of graded $\bQ$-vector spaces
 \[
  H^*(\bB(\Q_+,G);\bQ)\cong \cA(\Q).
 \]
It will be convenient to specify a particular isomorphism of graded $\bQ$-vector spaces.
Recall that, since $\Q$ is $\bQ$-Poincar\'e, it is locally finite and coconnected;
in particular, for all $\nu\ge0$ and for all $b\in\Q_\nu$ there is a canonical fundamental class
\[
[\Arr(\Q)_b,\NAdm(\Q)_b]\in H_{2\nu}\pa{|\Arr(\Q)_b|,|\NAdm(\Q)_b|;\bQ},
\]
see \cite[Definition 6.25]{Bianchi:Hur1}. Recall also the map of pairs
\[
\upsilon=\upsilon_b\colon\pa{|\Arr(\Q)_b|,|\NAdm(\Q)_b|}\to\pa{\Hur(\cR;\Q_+)_b;\del\Hur(\cR;\Q_+)_b};
\]
by \cite[Lemma 8.23]{Bianchi:Hur2} $\upsilon$ is a continuous bijection, and the hypothesis
that $\Q$ is locally finite implies that $|\Arr(\Q)_b|$ is compact, hence
$\upsilon_b$ is a homeomorphism. We thus obtain a fundamental class
\[
 [\Hur(\cR;\Q_+)_b;\del\Hur(\cR;\Q_+)_b]\in H_{2\nu}\pa{\Hur(\cR;\Q_+)_b;\del\Hur(\cR;\Q_+)_b;\bQ}.
\]
Using the homotopy equivalences of pairs
\[
(\Hur(\cR;\Q_+)_b;\del\Hur(\cR;\Q_+)_b)\simeq
(\Hur(\cR;\Q_+)_b;\delfat\Hur(\cR;\Q_+)_b)
\]
and excision we obtain a fundamental class
\[
 [\Hur(\mcR;\Q_+)_b;\delfat\Hur(\mcR;\Q_+)_b]\in H_{2\nu}\pa{\Hur(\mcR;\Q_+)_b;\delfat\Hur(\mcR;\Q_+)_b;\bQ}.
\]

\begin{nota}
 Let $\sca{S}\in\cA(\Q)_{2\nu}$ be the generator corresponding to the conjugacy class $S\subset\Q_\nu$.
 We regard $\sca{S}$ as the (unique) cohomology class
 \[
 \sca{S}\in H^{2\nu}\pa{\Hur(\mcR;\Q_+)_\nu, \delfat\Hur(\mcR;\Q_+)_\nu;\bQ}
 \]
 satisfying the following property: for all $b\in\Q_\nu$,
 the Kronecker pairing of $\sca{S}$ with the fundamental homology class
 $[\Hur(\mcR;\Q_+)_b,\delfat\Hur(\mcR;\Q_+)_b]$
 is $1$ if $b\in S$ and is $0$ if $b\notin S$.
 
 Note that $\sca{S}$ is invariant under the action of $G$ by conjugation, hence $\sca{S}$ corresponds
 to a cohomology class in $H^{2\nu}\pa{\fFN_\nu\bB,\delfat\fFN_\nu\bB;\bQ}$,
 which we also denote $\sca{S}$. Finally, we use the canonical isomorphisms
 \[
   H^{2\nu}\pa{\fFN_\nu\bB,\delfat\fFN_\nu\bB;\bQ}\overset{\cong}{\leftarrow}
   H^{2\nu}\pa{\FN_\nu\bB,\delfat\FNfat_{\nu-1}\bB;\bQ}\overset{\cong}{\to}
   H^{2\nu}\pa{\FN_\nu\bB;\bQ}\overset{\cong}{\leftarrow}
   H^{2\nu}\pa{\bB;\bQ}
 \]
 to regard $\sca{S}$ as a cohomology class in $H^{2\nu}(\bB(\Q_+,G);\bQ)$.
\end{nota}
 
\subsection{A strategy to compute the cup product}
We fix $\nu,\nu'\ge0$ throughout the rest of the section; our aim is to compute the cup product
$H^{2\nu}(\bB;\bQ)\otimes H^{2\nu'}(\bB;\bQ)\to H^{2\nu+2\nu'}(\bB;\bQ)$.
We fix $\sca{S}\in H^{2\nu}(\bB;\bQ)$ and $\sca{S'}\in H^{2\nu'}(\bB;\bQ)$:
our aim is to express the cup product $\sca{S}\smile\sca{S'}\in H^{2\nu+2\nu'}(\bB;\bQ)$ as a linear combination
of generators $\sca{T}$, for $T$ varying among conjugacy classes of $\Q$ contained in $\Q_{\nu+\nu'}$.

The restriction map $H^*(\bB;\bQ)\to H^*(\FN_{\nu+\nu'}\bB;\bQ)$
is an isomorphism in degrees $*\leq 2\nu+2\nu'$:
therefore, it suffices to compute the cup product $H^{2\nu}\otimes H^{2\nu'}\to H^{2\nu+2\nu'}$ for the space
$\FN_{\nu+\nu'}\bB$. In the rest of the subsection we use the abbreviation $\FN_\bullet=\FN_\bullet\bB$.

The argument to compute the cup product on $\FN_{\nu+\nu'}$ is based on certain
subspaces $\FNlr$, $\FNlfat$ and $\FNrfat$ of $\FN_{\nu+\nu'}$; there are
inclusions $\FN_{\nu+\nu'-1}\subset\FNlr$ as well as $\FN_{\nu-1}\subset\FNlfat$ and $\FN_{\nu'-1}\subset\FNrfat$, and we will prove that the last two
inclusions are, in fact, homotopy equivalences.
Postponing the actual definition of $\FNlr$, $\FNlfat$ and $\FNrfat$, we introduce some notation.
\begin{nota}
\label{nota:dellrfat}
We introduce several subspaces of $\FN_{\nu+\nu'}$:
\[
\FNlrfat=\FNlfat\cup\FNrfat;\quad\quad  \fFN_{\nu,\nu'}=\FN_{\nu+\nu'}\setminus\FNlr;
\]
\[
\dellfat\fFN_{\nu+\nu'}=\fFN_{\nu+\nu'}\cap \FNlfat; \quad \quad
\delrfat\fFN_{\nu+\nu'}=\fFN_{\nu+\nu'}\cap \FNrfat;\quad\quad
\dellrfat\fFN_{\nu+\nu'}=\fFN_{\nu+\nu'}\cap\FNlrfat;
\]
\[
\dellfat\fFN_{\nu,\nu'}=\fFN_{\nu,\nu'}\cap\FNlfat;\quad\quad
\delrfat\fFN_{\nu,\nu'}=\fFN_{\nu,\nu'}\cap\FNrfat;\quad\quad
\dellrfat\fFN_{\nu,\nu'}=\fFN_{\nu,\nu'}\cap\FNlrfat.
\]
\end{nota}
There is the following square of inclusions of subspaces, where both horizontal arrows are inclusions of a closed subspace of $\FN_{\nu+\nu'}$
in the interior of a larger subspace of $\FN_{\nu+\nu'}$:
\[
 \begin{tikzcd}
 \FN_{\nu+\nu'-1} \ar[r,hook]\ar[d,hook] & \FNfat_{\nu+\nu'-1}\ar[d,hook]\\
 \FNlr \ar[r,hook] & \FNlrfat.
 \end{tikzcd}
\]
This implies that the inclusions of couples
$(\fFN_{\nu+\nu'},\dellrfat\fFN_{\nu+\nu'})\subset(\FN_{\nu+\nu'},\FNlrfat)$
and $(\fFN_{\nu,\nu'},\dellrfat\fFN_{\nu,\nu'})\subset(\FN_{\nu+\nu'},\FNlrfat)$
satisfy excision.
Finally, the bundle projection $\fp_{\nu+\nu'}\colon\fFN_{\nu+\nu'}\to\cB G$ exhibits also the subspaces
$\dellfat\fFN_{\nu+\nu'}$, $\delrfat\fFN_{\nu+\nu'}$, $\dellrfat\fFN_{\nu+\nu'}$, $\fFN_{\nu,\nu'}$,
$\dellfat\fFN_{\nu,\nu'}$, $\delrfat\fFN_{\nu,\nu'}$ and $\dellrfat\fFN_{\nu,\nu'}$ as bundles over $\cB G$, with suitable fibres: local trivialisations for these bundles can be obtained by restricting the local trivialisations of $\fp_{\nu+\nu'}$ given in the proof of Proposition \ref{prop:fpnubundle}.

The previous technical results will allow us to write two commutative diagrams of cohomology groups,
where we understand $\bQ$-coefficients for cohomology.
We state the two diagrams as propositions for future reference.
\begin{prop}
 \label{prop:firstdiagram}
There is a commutative diagram of cohomology groups as follows
\[
 \begin{tikzcd}
 H^{2\nu}(\FN_{\nu+\nu'})\otimes H^{2\nu'}(\FN_{\nu+\nu'}) \ar[r,"\smile"] &  H^{2\nu+2\nu'}(\FN_{\nu+\nu'})\\
 H^{2\nu}(\FN_{\nu+\nu'},\FNlfat)\otimes H^{2\nu'}(\FN_{\nu+\nu'},\FNrfat) \ar[r,"\smile"]\ar[u,"\cong"]\ar[d]& 
 H^{2\nu+2\nu'}(\FN_{\nu+\nu'},\FNlrfat)\ar[u]\ar[d,"\cong"]\\
 H^{2\nu}(\fFN_{\nu+\nu'},\dellfat\fFN_{\nu+\nu'})\otimes H^{2\nu'}(\fFN_{\nu+\nu'},\delrfat\fFN_{\nu+\nu'})\ar[d]\ar[r,"\smile"] & 
 H^{2\nu+2\nu'}(\fFN_{\nu+\nu'},\dellrfat\fFN_{\nu+\nu'})\ar[d,"\cong"]\\
 H^{2\nu}(\fFN_{\nu,\nu'},\dellfat\fFN_{\nu,\nu'})\otimes H^{2\nu'}(\fFN_{\nu,\nu'},\delrfat\fFN_{\nu,\nu'})\ar[r,"\smile"] & 
 H^{2\nu+2\nu'}(\fFN_{\nu,\nu'},\dellrfat\fFN_{\nu,\nu'}).
 \end{tikzcd}
\]
\end{prop}
We will consider $\sca{S}\otimes\sca{S'}$ as an element in $H^{2\nu}(\FN_{\nu+\nu'},\FNlfat)\otimes H^{2\nu'}(\FN_{\nu+\nu'},\FNrfat)$ in the previous diagram, and compute explicitly the image of the cup product $\sca{S}\smile\sca{S'}$ in $H^{2\nu+2\nu'}(\fFN_{\nu+\nu'},\dellrfat\fFN_{\nu+\nu'})
\cong H^{2\nu+2\nu'}(\fFN_{\nu,\nu'},\dellrfat\fFN_{\nu,\nu'})$.
\begin{prop}
\label{prop:seconddiagram}
 There is a commutative diagram of cohomology groups as follows
 \[
 \begin{tikzcd}
   H^{2\nu+2\nu'}(\FN_{\nu+\nu'}) &\\
   H^{2\nu+2\nu'}(\FN_{\nu+\nu'},\FNlrfat)\ar[u]\ar[r]\ar[d,"\cong"] &
   H^{2\nu+2\nu'}(\FN_{\nu+\nu'},\FNfat_{\nu+\nu'-1})\ar[ul,"\cong",swap] \ar[d,"\cong"]\\
   H^{2\nu+2\nu'}(\fFN_{\nu+\nu'},\dellrfat\fFN_{\nu+\nu'})\ar[r,"\theta"] &
   H^{2\nu+2\nu'}(\fFN_{\nu+\nu'},\delfat\fFN_{\nu+\nu'}).
 \end{tikzcd}
\]
\end{prop}
We will compute the image along the natural map $\theta$ of $\sca{S}\smile\sca{S'}$,
and identify it with the class $\sca{S}\cdot\sca{S'}\in \cA(\Q)_{2\nu+2\nu'}$.

\subsection{Proof of Propositions \ref{prop:firstdiagram} and \ref{prop:seconddiagram}}
\begin{defn}
Recall Definition \ref{defn:delfat}. We let $\zcentrel=\frac 27+\frac{\sqrt{-1}}2$
and $\zcentrer=\frac 57+\frac{\sqrt{-1}}2$; note that the homothety centred at $\zcentrel$
with rescaling factor $1/4$ maps $\mcR$ to the open square $(1/4,3/8)\times(7/16,9/16)$,
i.e. $(1/4,3/8)\times(7/16,9/16)=\frac{7\zcentrel}8 +\frac 18\mcR$;
similarly $(5/8,3/4)\times(7/16,9/16)=\frac{7\zcentrer}8 +\frac 18\mcR$. See Figure \ref{fig:dellrfat}.

For $\bullet=\mathrm{l},\mathrm{r}$, we define a subspace $\delfat_{\bullet}\cR$ of $\cR$ by
$
\delfat_{\bullet}\cR:=\cR\setminus(\frac{7z_{\mathrm{c},\bullet}}8 +\frac 18\mcR)
$. The identity of $\C$ is a map of nice couples
$(\cR,\del\cR)\to(\cR,\delfat_{\bullet}\cR)$,
giving rise to a map $(\Id_\C)_*\colon\FN_{\nu+\nu'}=\FN_{\nu+\nu'}\bB\to \Hur(\cR,\delfat_{\bullet}\cR;\Q,G)$.
Recall Definition \ref{defn:FN}: we define $\FNlfat\subset\FN_{\nu+\nu'}$ as the preimage along $(\Id_\C)_*$ of
$\FN_{\nu-1}\Hur(\cR,\dellfat\cR;\Q,G)$ and, respectively, 
$\FNrfat\subset\FN_{\nu+\nu'}$ as the preimage along $(\Id_\C)_*$ of
$\FN_{\nu'-1}\Hur(\cR,\delrfat\cR;\Q,G)$.
\end{defn}

\begin{figure}[ht]
 \begin{tikzpicture}[scale=5,decoration={markings,mark=at position 0.38 with {\arrow{>}}}]
  \draw[dashed,->] (-.05,0) to (1.1,0);
  \draw[dashed,->] (0,-.1) to (0,1.1);
  \draw[very thick, fill=black!40!white] (0,0) rectangle (1,1);
  \fill[white, opacity=.4] (.25,.25) rectangle (.75, .75);
  \fill[white] (.25,.4375) rectangle (.375,.5625);
  \fill[white] (.75,.4375) rectangle (.625,.5625);
  \draw[thin] (0,1) to (.2857,.5) to (0,0) to (.7143,.5) to (1,1) to 
  (.2857,.5) to (1,0) to (.7143,.5) to (0,1);
  \draw[dotted] (.5,0) to (.5,1);
  \node at (.34,.5){\tiny$\zcentrel$};
  \node at (.66,.5){\tiny$\zcentrer$};
 \end{tikzpicture}
 \caption{The complement of the left (respectively, right) white square is $\dellfat\cR$
 (respectively, $\delrfat\cR$), whereas the total grey area is $\delfat\cR$; the dotted vertical
 line splits the total grey area into $\delfat\cRl$ and $\delfat\cRr$.
 }
\label{fig:dellrfat}
\end{figure}

Roughly speaking, a configuration $\fc=(P,\psi,\phi)\in\FN_{\nu+\nu'}$ lies in $\FNlfat$ if the sum of the norms of
the values of the monodromy $\psi$ around points of $P$ lying in the open square $(1/4,3/8)\times(7/16,9/16)$
does not exceed $\nu-1$; similarly for $\FNrfat$, referring to the open square $(5/8,3/4)\times(7/16,9/16)$
and replacing the threshold $\nu-1$ with $\nu'-1$.
To keep the notation simple, we avoid adding the indices $\nu$ and $\nu'$ to $\FNlfat$ and $\FNrfat$.
Note that 
$\FN_{\nu-1}\subset\FNlfat$ and $\FN_{\nu'-1}\subset\FNrfat$.
\begin{lem}
\label{lem:FNnuFNlfat}
The inclusions
$\FN_{\nu-1}\subset\FNlfat$ and $\FN_{\nu'-1}\subset\FNrfat$ are homotopy equivalences. 
\end{lem}
\begin{proof}
We focus on the inclusion $\FN_{\nu-1}\subset\FNlfat$, the other one being analogous.
Recall Definition \ref{defn:cHbB} and the proof of Lemma \ref{lem:excision}.
For $\fc\in\FN_{\nu+\nu'}$ we denote by $\mathfrak{W}_{\mathrm{l},\nu}(\fc)$ the supremum
in $[1/8,1]$ of all $1/8<\epsilon<1$
for which $\pa{\cH^{\bB}_{\zcentrel}}_*(\fc,1/\epsilon)\in\FN_{\nu-1}\bB$. We define a homotopy
 \[
  H^{\bB}_{\mathrm{l},\nu}\colon\FN_{\nu+\nu'}\times[0,1]\to\FN_{\nu+\nu'},\quad\quad
  (\fc,s)\mapsto \pa{\cH^{\bB}_{\zcentrel}}_*\pa{\fc\,,\,1-s+\frac{s}{\mathfrak{W}_{\mathrm{l},\nu}(\fc)}}.
 \]
 The homotopy $H^{\bB}_{\mathrm{l},\nu}$ restricts to a deformation retraction of $\FNlfat$
 onto $\FN_{\nu-1}$. 
\end{proof}
Lemma \ref{lem:FNnuFNlfat} implies the top left isomorphism in Proposition \ref{prop:firstdiagram}.
\begin{lem}
 \label{lem:FNfatinFNlrfat}
The space $\FNfat_{\nu+\nu'-1}$ is contained in the union
$\FNlfat\cup\FNrfat$.
\end{lem}
\begin{proof}
 Let $\fc\in\FNfat_{\nu+\nu'-1}$, use Notation \ref{nota:fc}, and for $\bullet=\mathrm{l,r}$ denote by
 $P_\bullet\subset P$ the intersection of $P$ with the open square
 $(\frac{7z_{\mathrm{c},\bullet}}8+\frac 18\mcR)$. Without loss of generality, we may assume that there are indices $0\le l'\le l''\le l$ such that $P_\mathrm{l}=\set{z_1,\dots,z_{l'}}$
 and $P_{\mathrm{r}}=\set{z_{l'+1},\dots,z_{l''}}$. Let $\gen_1,\dots,\gen_k$ be an admissible generating set. Then the hypothesis $\fc\in\FNfat_{\nu+\nu'-1}$, together with the fact that $(\frac{7\zcentrel}8+\frac 18\mcR)$
 and $(\frac{7\zcentrer}8+\frac 18\mcR)$ are disjoint and contained in $(\frac{\zcentre}2+\frac 12\mcR)$, implies the inequality
 \[
  \sum_{i=1}^{l'}N(\psi(\gen_i))+\sum_{i=l'+1}^{l''} N(\psi(\gen_i))\le \nu+\nu'-1.
 \]
It follows that at least one of the following two inequalities holds
\[
  \sum_{i=1}^{l'}N(\psi(\gen_i))\le \nu-1;\quad\quad\sum_{i=l'+1}^{l''} N(\psi(\gen_i))\le \nu'-1;
\]
the first inequality implies $\fc\in\FNlfat$, the second implies $\fc\in\FNrfat$.
\end{proof}

\begin{nota}
 \label{nota:cRlr}
 We introduce several subspaces of $\cR$, see Figure \ref{fig:dellrfat}:
 \[
 \cRl=[0,1/2]\times[0,1],\quad\quad\quad\quad \cRr=[1/2,1]\times[0,1],
 \]
 \[
 \mcRl=(0,1/2)\times(0,1),\quad\quad\quad\quad \cRr=(1/2,1)\times(0,1),
 \]
 \[
 \del\cRl=\cRl\setminus\mcRl,\quad\quad\quad\quad \del\cRr=\cRr\setminus\mcRr,
 \]
 \[
 \delfat\cRl=\cRl\setminus \pa{\frac{7\zcentrel}8+\frac 18\mcR},\quad\quad\quad\quad 
 \delfat\cRl=\cRl\setminus \pa{\frac{7\zcentrel}8+\frac 18\mcR},
 \]
 \[
 \delfat\mcRl=\mcRl\setminus \pa{\frac{7\zcentrel}8+\frac 18\mcR},\quad\quad\quad\quad 
 \delfat\mcRl=\mcRl\setminus \pa{\frac{7\zcentrel}8+\frac 18\mcR}. 
 \]
 \end{nota}

\begin{defn}
 \label{defn:FNlr}
 Recall Definition \ref{defn:FN}.
 The identity of $\C$ induces maps of nice couples $(\cR;\del)\to(\cR,\cR\setminus\mcRl)$ and
 $(\cR;\del)\to(\cR,\cR\setminus\mcRr)$.
 We define $\FNlr\subset\FN_{\nu+\nu'}$ as the subspace of configurations $\fc$ such that at least one of the following conditions holds:
 \begin{itemize}
  \item  $(\Id_\C)_*(\fc)\in\Hur(\cR,\cR\setminus\mcRl;\Q_+,G)$ has norm $\leq\nu-1$;
  \item  $(\Id_\C)_*(\fc)\in\Hur(\cR,\cR\setminus\mcRr;\Q_+,G)$ has norm $\leq\nu'-1$.
 \end{itemize}
\end{defn}
Roughly speaking, the complement $\fFN_{\nu,\nu'}$ of $\FNlr$ in $\FN_{\nu+\nu'}$ contains those configurations
$\fc=(P,\psi,\phi)$ such that $P\subset\mcR\setminus\set{1/2}\times(0,1)$, the sum of the norms of the values
of $\psi$ around points of $P\cap\mcRl$ is equal to $\nu$, and the sum of the norms of the values
of $\psi$ around points of $P\cap\mcRr$ is equal to $\nu'$.
\begin{lem}
 \label{lem:FNlrinFNlrfat}
 The space $\FNlr$ is contained in the interior of $\FNlrfat$, considered as subspace of $\FN_{\nu+\nu'}$.
\end{lem}
\begin{proof}
Given $\fc=(P,\psi,\phi)\in\FNlr$, it suffices to note that for any adapted covering $\uU$ of $P$ with the sets
$U_i$ of diameter at most $1/8$, the restricted normal neighbourhood $\fU(\fc;\uU)\cap\FN_{\nu+\nu'}$
is contained in $\FNlrfat$.
\end{proof}
Lemma \ref{lem:FNlrinFNlrfat} implies, together with the inclusion $\FN_{\nu+\nu'-1}\subset\FNlr$,
that the couple $(\FN_{\nu+\nu'},\FNlrfat)$ satisfies excision with respect to the subspaces
$\FN_{\nu+\nu'-1}$ and $\FNlr$, i.e. the two bottom right vertical isomorphisms in Proposition \ref{prop:firstdiagram} hold.
This concludes the proof of Proposition \ref{prop:firstdiagram}.

In the same way, the two bottom vertical isomorphisms of Proposition \ref{prop:seconddiagram} follow from excision
of the subspace $\FN_{\nu+\nu'-1}$, whereas the top diagonal isomorphism follows from the computation
of $H^*(\bB;\bQ)$ using the Leray spectral sequence. This concludes also the proof of Proposition \ref{prop:seconddiagram}.

\subsection{Conclusion of the proof of Theorem \ref{thm:HbBQG}}
As remarked already, $\fp_{\nu+\nu'}$ exhibits
all subspaces of $\fFN_{\nu+\nu'}$ occurring in the bottom rows of Propositions \ref{prop:firstdiagram} and \ref{prop:seconddiagram} as bundles over $\cB G$: the proof of Proposition \ref{prop:fpnubundle} provides local trivialisations also for these bundles. Our next aim is to compute the cohomology groups of these bundles and of the couples of bundles they form.

\begin{defn}
For $\nu,\nu'\ge0$ we denote by $\Q_{\nu,\nu'}\subset\Q_{\nu}\times\Q_{\nu'}$ the subset of couples
$(a,b)$ for which the product $ab$ is defined in $\Q$.
We set
\[
\Hur(\mcR;\Q_+)_{\nu,\nu'}:=\coprod_{(a,b)\in\Q_{\nu,\nu'}} \Hur(\mcRl;\Q_+)_a\times\Hur(\mcRr;\Q_+)_b.
\]
We regard $\Hur(\mcR;\Q_+)_{\nu,\nu'}$ as a subspace of $\Hur(\mcR;\Q_+)_{\nu+\nu'}$ by regarding
each product $\Hur(\mcRl;\Q_+)_a\times\Hur(\mcRr;\Q_+)_b$ as a subspace of $\Hur(\mcR;\Q_+)_{ab}$ under the inclusion given by
\[
\begin{tikzcd}
 \Hur(\mcRl;\Q_+)_a\times\Hur(\mcRr;\Q_+)_b 
 \ar[d,"\fri^{\C}_{\bS_{0,1/2}}\times\fri^{\C}_{\bS_{1/2,1}}"] & \\
 \Hur^{\bS_{0,1/2}}(\mcRl;\Q_+)_a\times\Hur^{\bS_{1/2,1}}(\mcRr;\Q_+)_b \ar[d,"-\sqcup-"]& \\
 \Hur^{\bS_{0,1}}(\mcR;\Q_+)_{ab}\ar[r,"(\fri^{\C}_{\bS_{0,1}})^{-1}"] &\Hur(\mcR;\Q_+)_{ab}.
\end{tikzcd}
\]

\end{defn}
The generic fibre of the bundle $\fp_{\nu+\nu'}\colon\fFN_{\nu,\nu'}\to\cB G$ is homeomorphic to the space $\Hur(\mcR;\Q_+)_{\nu,\nu'}$. To see this, note that the same argument of the proof of Proposition \ref{prop:fpnubundle}
identifies the fibre of $\fp_{\nu+\nu'}\colon\fFN_{\nu,\nu'}\to\cB G$ with the subspace of
$\Hur(\mcRl\cup\mcRr;\Q_+)_{\nu+\nu'}$
containing configurations $\fc$ with the following properties:
\begin{itemize}
 \item $\fri^{\C}_{\bS_{0,1/2}}$ sends $\fc$ to a configuration in
 $\Hur^{\bS_{0,1/2}}(\mcRl;\Q_+)$ with total monodromy in $\Q_\nu$;
 \item $\fri^{\C}_{\bS_{1/2,1}}$ sends $\fc$ to a configuration in
 $\Hur^{\bS_{1/2,1}}(\mcRr;\Q_+)$ with total monodromy in $\Q_{\nu'}$;
 \item the product of the elements $\totmon(\fri^{\C}_{\bS_{0,1/2}}(\fc))$ and
 $\totmon(\fri^{\C}_{\bS_{1/2,1}}(\fc))$ is defined in $\Q$.
\end{itemize}
We can further identify the fibres of the following restrictions of $\fp_{\nu+\nu'}$:
\begin{itemize}
 \item the fibre of $\fFN_{\nu,\nu'}\to\cB G$ is identified with $\Hur(\mcR;\Q_+)_{\nu,\nu'}$, as already mentioned;
 \item the fibre of $\dellfat\fFN_{\nu,\nu'}\to\cB G$ is identified with the intersection of $\Hur(\mcR;\Q_+)_{\nu,\nu'}$ with the product $\delfat\Hur(\mcRl;\Q_+)_\nu\times\Hur(\mcRr;\Q_+)_{\nu'}$, i.e. with the space
 \[
  \dellfat\Hur(\mcR;\Q_+)_{\nu,\nu'}:=\coprod_{(a,b)\in\Q_{\nu,\nu'}}
  \delfat\Hur(\mcRl;\Q_+)_a\times\Hur(\mcRr;\Q_+)_b,
 \]
 where $\delfat\Hur(\mcRl;\Q_+)_a:=\Hur(\mcRl;\Q_+)_a\setminus
 \Hur(\frac{7\zcentrel}8+\frac 18\mcR;\Q_+)_a$;
 \item the fibre of $\delrfat\fFN_{\nu,\nu'}\to\cB G$ is identified with the intersection of $\Hur(\mcR;\Q_+)_{\nu,\nu'}$ with the product $\Hur(\mcRl;\Q_+)_\nu\times\delfat\Hur(\mcRr;\Q_+)_{\nu'}$, i.e. with the space
 \[
  \delrfat\Hur(\mcR;\Q_+)_{\nu,\nu'}:=\coprod_{(a,b)\in\Q_{\nu,\nu'}}\Hur(\mcR;\Q_+)_a\times\delfat\Hur(\mcR;\Q_+)_b,
 \]
  where $\delfat\Hur(\mcRr;\Q_+)_b:=\Hur(\mcRr;\Q_+)_b\setminus
 \Hur(\frac{7\zcentrer}8+\frac 18\mcR;\Q_+)_b$;
 \item the fibre of $\dellrfat\fFN_{\nu,\nu'}\to\cB G$ is identified with the union
 $\dellrfat\Hur(\mcR;\Q_+)_{\nu,\nu'}:=\dellfat\Hur(\mcR;\Q_+)_{\nu,\nu'}\cup
 \delrfat\Hur(\mcR;\Q_+)_{\nu,\nu'}$.
\end{itemize}
By using functoriality with respect to suitable semialgebraic homeomorphisms of $\C$,
we can identify the couples of spaces
\[
(\Hur(\mcRl;\Q_+)_a,\delfat\Hur(\mcRl;\Q_+)_a) \cong (\Hur(\mcR;\Q_+)_a,\delfat\Hur(\mcR;\Q_+)_a);
\]
\[
(\Hur(\mcRr;\Q_+)_b,\delfat\Hur(\mcRr;\Q_+)_b) \cong (\Hur(\mcR;\Q_+)_b,\delfat\Hur(\mcR;\Q_+)_b).
\]

\begin{nota}
For $d\ge0$ we denote by
$\bQ\qua{\Q_{\nu,\nu'}}[d]$ the graded $\bQ$-vector space concentrated in degree $d$,
 with basis the set $\Q_{\nu,\nu'}$.
 
 For $d,d'\ge0$ we denote by $\mu_{\nu,\nu'}\colon \bQ\qua{\Q_{\nu,\nu'}}[d]\otimes \bQ\qua{\Q_{\nu,\nu'}}[d']\to \bQ\qua{\Q_{\nu,\nu'}}[d+d']$ the pairing given
 by $(a,b)\otimes (a,b)\mapsto (a,b)$ for $(a,b)\in\Q_{\nu,\nu'}$ and $(a,b)\otimes (a',b')\mapsto 0$
 for $(a,b)\neq (a',b')\in\Q_{\nu,\nu'}$.
\end{nota}
Consider the bottom row of Proposition \ref{prop:firstdiagram}: all groups involved are cohomology groups of couples of bundles
over $\cB G$.
The cup product on fibres reads as follows, where we simplify the notation by writing $\Hur$ for $\Hur(\mcR;\Q_+)$
\[
 \begin{tikzcd}
 H^{2\nu}(\Hur_{\nu,\nu'},\dellfat)\otimes H^{2\nu'}(\Hur_{\nu,\nu'},\delrfat)\ar[r,"\smile"]\ar[d,"\cong"] &
 H^{2\nu+2\nu'}(\Hur_{\nu,\nu'},\dellrfat)\ar[d,"\cong"]\\
 \bQ\qua{\Q_{\nu,\nu'}}[2\nu] \otimes
 \bQ\qua{\Q_{\nu,\nu'}}[2\nu']  \ar[r,"\mu_{\nu,\nu'}",swap] &
 \bQ\qua{\Q_{\nu,\nu'}}[2\nu+2\nu']. 
 \end{tikzcd}
\]
\begin{nota}
For $a\in\Q_\nu$ we denote by $\fc_{\mathrm{l},a}\in\Hur(\mcRl;\Q_+)_a$ the unique configuration
supported on $\set{\zcentrel}$; similarly, 
for $b\in\Q_{\nu'}$ we denote by $\fc_{\mathrm{r},b}\in\Hur(\mcRr;\Q_+)_b$ the unique configuration
supported on $\set{\zcentrer}$.
\end{nota}
The three cohomology groups in the top row of the previous diagram have bases in bijection with the set $\Q_{\nu,\nu'}$.
For instance, $H^{2\nu}(\Hur_{\nu,\nu'},\dellfat))$ has a basis
given by the cohomology duals of the homology classes
$[\Hur(\mcRl;\Q_+)_a,\delfat]\otimes[\fc_{\mathrm{r},b}]$.
Here $[\Hur(\mcRl;\Q_+)_a,\delfat]\in H_{2\nu}(\Hur(\mcRl;\Q_+)_a,\delfat;\bQ)$ is the fundamental
homology class, and $[\fc_{\mathrm{r},b}]\in H_0(\Hur(\mcRr;\Q_+)_b;\bQ)$
denotes the ``ground'' homology class.

A similar description gives a basis for $H^{2\nu'}(\Hur_{\nu,\nu'},\delrfat)$ in bijection
with $\Q_{\nu,\nu'}$, whereas for $H^{2\nu+2\nu'}(\Hur_{\nu,\nu'},\dellrfat)$ we consider
the cohomology duals of the homology classes
$[\Hur(\mcRl;\Q_+)_a,\delfat]\otimes[\Hur(\mcRr;\Q_+)_b,\delfat]$.
Taking $G$-invariants, we obtain an explicit computation of the bottom row of Proposition
\ref{prop:firstdiagram} as follows:
\[
 \begin{tikzcd}
 H^{2\nu}(\fFN_{\nu,\nu'},\dellfat\fFN_{\nu,\nu'})\otimes H^{2\nu'}(\fFN_{\nu,\nu'},\delrfat\fFN_{\nu,\nu'})\ar[r,"\smile"] \ar[d,"\cong"]& 
 H^{2\nu+2\nu'}(\fFN_{\nu,\nu'},\dellrfat\fFN_{\nu,\nu'})\ar[d,"\cong"]\\
 \Big(\bQ\qua{\Q_{\nu,\nu'}}[2\nu]\Big)^G \otimes
 \Big(\bQ\qua{\Q_{\nu,\nu'}}[2\nu']\Big)^G  \ar[r,"\mu_{\nu,\nu'}",swap] &
 \Big(\bQ\qua{\Q_{\nu,\nu'}}[2\nu+2\nu']\Big)^G.
 \end{tikzcd}
\]

\begin{lem}
\label{lem:scaSrestriction}
Let $S\subset\Q_\nu$ be a conjugacy class of $\Q$, and consider $\sca{S}$ as a cohomology class in the group
$\in H^{2\nu}(\FN_{\nu+\nu'},\FNlfat)\cong H^{2\nu}(\FN_{\nu+\nu'},\FN_{\nu-1})\cong H^{2\nu}(\FN_{\nu+\nu'})$; the restriction of $\sca{S}$
to $ H^{2\nu}(\fFN_{\nu,\nu'},\dellfat\fFN_{\nu,\nu'})$
is the element
\[
\sum_{(a,b)\in \Q_{\nu,\nu'}\cap S\times \Q_{\nu'}} (a,b)\in\bQ\qua{\Q_{\nu,\nu'}}[2\nu].
\]

Similarly, for a conjugacy class $S'\subset\Q_{\nu'}$, the class $\sca{S'}\in H^{2\nu'}(\FN_{\nu+\nu'},\FNrfat)$ restricts to
\[
\sum_{(a,b)\in \Q_{\nu,\nu'}\cap \Q_\nu\times S'} (a,b)\in\bQ\qua{\Q_{\nu,\nu'}}[2\nu']\cong
H^{2\nu'}(\fFN_{\nu,\nu'},\delrfat\fFN_{\nu,\nu'}).
\]
\end{lem}
\begin{proof}
We focus on the first part of the statement, the second being analogous.
Fix $(a,b)\in\Q_{\nu,\nu'}$; the couple of spaces $(\Hur(\mcRl;\Q_+)_a,\delfat)\times\fc_{\mathrm{r},b}$
can be embedded into the couple of bundles $(\fFN_{\nu,\nu'},\dellfat\fFN_{\nu,\nu'})$
as part of the fibre over the basepoint of $\cB G$.
We consider the homology class
\[
[\Hur(\mcRl;\Q_+)_a,\delfat]\otimes[\fc_{\mathrm{r},b}]\in H_{2\nu}(\fF_{\nu,\nu'},\dellfat\fFN_{\nu,\nu'};\bQ);
\]
such classes generate the group $H_{2\nu}(\fF_{\nu,\nu'},\dellfat\fFN_{\nu,\nu'};\bQ)$,
so in order to identify the restriction of $\sca{S}\in H^{2\nu}(F_{\nu+\nu'},\FNlfat;\bQ)$ to
$H^{2\nu}(\fF_{\nu,\nu'},\dellfat\fFN_{\nu,\nu'};\bQ)$,
it suffices to compute the Kronecker pairing of the restricted $\sca{S}$ with all homology classes of the form
$[\Hur(\mcRl;\Q_+)_a,\delfat]\otimes[\fc_{\mathrm{r},b}]$.
Let $\fj$ be the composite
\[
\fj\colon(\Hur(\mcRl;\Q_+)_a,\delfat)\times \set{\fc_{\mathrm{r},b}} \hookrightarrow (\fFN_{\nu,\nu'},\dellfat\fFN_{\nu,\nu'})
\subset(\FN_{\nu+\nu'},\FNlfat);
\]
then we can also consider the homology class $\fj_*([\Hur(\mcRl;\Q_+)_a,\delfat]\otimes[\fc_{\mathrm{r},b}])$ in
$H_{2\nu}(F_{\nu+\nu'},\FNlfat;\bQ)$, and compute its Kronecker pairing with $\sca{S}$.

Fix a homotopy $\cH\colon\C\times[0,1]\to\C$ with the following properties:
\begin{itemize}
 \item for all $0\le s\le 1$, $\cH(-,s)$ is a semialgebraic self-map of $\C$ fixing $\C\setminus\mcR$ pointwise, and sending $\dellfat\cR$ into itself;
 \item $\cH(-,0)$ is the identity of $\C$;
 \item $\cH(-,1)$
 restricts to a homeomorphism of couples $(\mcRl,\delfat\mcRl)\overset{\cong}{\to}(\mcR,\delfat\mcR)$,
 and it sends $\zcentrer\mapsto\zdiamright$.
\end{itemize}
Then $\cH$ induces a homotopy of maps of couples
\[
 \cH_*\colon (F_{\nu+\nu'},\FNlfat)\times[0,1]\to(F_{\nu+\nu'},\FNlfat),
\]
and composing this homotopy with $\fj$ we obtain a homotopy of maps of couples
\[
 H=\cH\circ(\fj\times\Id) \colon 
 (\Hur(\mcRl;\Q_+)_a,\delfat)\times\set{\fc_{\mathrm{r},b}}\times[0,1]\to (\FN_{\nu+\nu'},\FNlfat).
\]
Since $H(-,0)=\fj$, we have
\[
\fj_*([\Hur(\mcRl;\Q_+)_a,\delfat]\otimes[\fc_{\mathrm{r},b}])=
H(-,1)_*([\Hur(\mcRl;\Q_+)_a,\delfat]\otimes[\fc_{\mathrm{r},b}]),
\]
so we can focus on the Kronecker pairing of the latter homology class with $\sca{S}$.

Now note that, by construction,
$H(-,1)$ can be considered as a map of couples
\[
(\Hur(\mcRl;\Q_+)_a,\delfat)\times\set{\fc_{\mathrm{r},b}}\to (\fFN_{\nu},\delfat\fFN_\nu)\subset (\FN_{\nu+\nu'},\FNlfat);
\]
more precisely
$H(-,1)$ restricts to a homeomorphism of
$(\Hur(\mcRl;\Q_+)_a,\delfat)\times \set{\fc_{\mathrm{r},b}}$ onto
$(\Hur(\mcR;\Q_+)_a,\delfat)$, where the latter couple is considered as part of the fibre
of $\fp_\nu\colon(\fFN_\nu,\delfat\fFN_\nu)\to\cB G$ over the unique configuration
$\fc_{\mathrm{r},\fe(b)}=(P,\psi)$ satisfying the following properties:
\begin{itemize}
 \item $\fc_{\mathrm{r},\fe(b)}$ is supported on the set $P=\set{0,\zdiamright}$;
 \item if $\gamma$ is a simple loop in $\bS_{0,\infty}$ spinning clockwise around $\zdiamright$,
 then $\psi([\gamma])=\fe(b)\in G$.
\end{itemize}
It follows that $H(-,1)_*([\Hur(\mcRl;\Q_+)_a,\delfat]\otimes[\fc_{\mathrm{r},b}])$
is the image of the homology class $[\Hur(\mcR;\Q_+)_a,\delfat]\in H_{2\nu}(\fFN_\nu,\delfat\fFN_\nu)$
under the inclusion $(\fFN_\nu,\delfat\fFN_\nu)\subset (\FN_{\nu+\nu'},\FNlfat)$;
the Kronecker pairing of the latter class with $\sca{S}$ is 1 if and only if $a$ belongs to $S$.
\end{proof}
Lemma \ref{lem:scaSrestriction} implies that
for $S\subset\Q_\nu$ and $S'\subset\Q_{\nu'}$, the restriction of the 
cup product $\sca{S}\smile\sca{S'}$ to the cohomology group
in the bottom right corner of Proposition \ref{prop:firstdiagram}, can be identified with the $G$-invariant element
\[
\sum_{(a,b)\in \Q_{\nu,\nu'}\cap S\times S'} (a,b)\ \in \ \Big(\bQ\qua{\Q_{\nu,\nu'}}[2\nu+2\nu']\Big)^G
\ \cong H^{2\nu+2\nu'}(\fFN_{\nu,\nu'},\dellrfat\fFN_{\nu,\nu'};\bQ).
\]
The following lemma concludes the proof of Theorem \ref{thm:HbBQG}.
\begin{lem}
There is the following commutative diagram:
 \[
  \begin{tikzcd}
H^{2\nu+2\nu'}(\Hur(\mcR;\Q_+)_{\nu+\nu'},\dellrfat;\bQ)\ar[r]\ar[d,"\cong"] & H^{2\nu+2\nu'}(\Hur(\mcR;\Q_+),\delfat);\bQ)\ar[d,"\cong"]\\
\bQ\qua{\Q_{\nu,\nu'}}[2\nu+2\nu']\ar[r,"\sca{-}\cdot\sca{-}"] &\bQ[\Q]_{2\nu+2\nu'}.
  \end{tikzcd}
 \]
Here the vertical maps are the canonical isomorphisms, given by the bases of the top homology
groups of elements of the form $[\Hur(\mcRl;\Q_+)_a,\delfat]\otimes[\Hur(\mcRr;\Q_+)_b,\delfat]$,
for $(a,b)\in\Q_{\nu,\nu'}$, and, respectively, of the form $[\Hur(\mcR;\Q_+)_c,\delfat]$, for $c\in\Q_{\nu+\nu'}$.
Moreover, $\sca{-}\cdot\sca{-}$ denotes the map $(a,b)\mapsto \sca{a}\cdot\sca{b}$.
Passing to $G$-invariants, we obtain the following commutative diagram:
 \[
  \begin{tikzcd}
H^{2\nu+2\nu'}(\fFN_{\nu+\nu'},\dellrfat\fFN_{\nu+\nu'};\bQ)\ar[r,"\theta"]\ar[d,"\cong"] & H^{2\nu+2\nu'}(\fFN_{\nu+\nu'},\delfat\fFN_{\nu+\nu'};\bQ)\ar[d,"\cong"]\\
\Big(\bQ\qua{\Q_{\nu,\nu'}}[2\nu+2\nu']\Big)^G\ar[r,"\sca{-}\cdot\sca{-}"] &\cA(\Q)_{2\nu+2\nu'}.
  \end{tikzcd}
 \]
\end{lem}
\begin{proof}
We first argue that commutativity of the first diagram implies commutativity of the second.
The Serre spectral sequences computing the cohomology of the couples of bundles $(\fFN_{\nu+\nu'},\dellrfat\fFN_{\nu+\nu'})$
and $(\fFN_{\nu+\nu'},\delfat\fFN_{\nu+\nu'})$ have both a second page concentrated in the
single entry in position $(0,2\nu+\nu')$,
with value, respectively, the $G$-invariants
\[
H^{2\nu+2\nu'}(\Hur(\mcR;\Q_+)_{\nu+\nu'},\dellrfat;\bQ)^G\quad\mbox{ and }\quad
H^{2\nu+2\nu'}(\Hur(\mcR;\Q_+),\delfat);\bQ)^G.
\]
This implies that the spectral sequences collapse and
that the top row of the second diagram is obtained from the top row of the first diagram by taking $G$-invariants.
This, together with the fact that all arrows in the first diagram are $G$-equivariant, shows that the second
diagram is obtained from the first by taking $G$-invariants.

It suffices therefore to prove commutativity of the first diagram.
Recall that $\Hur(\mcR;\Q_+)_{\nu+\nu'}$ is an orientable $\bQ$-homology manifold; moreover,
both differences $\Hur(\mcR;\Q_+)_{\nu+\nu'}\setminus\dellrfat\Hur(\mcR;\Q_+)_{\nu+\nu'}$
and $\Hur(\mcR;\Q_+)_{\nu+\nu'}\setminus\delfat\Hur(\mcR;\Q_+)_{\nu+\nu'}$ are relatively compact
inside $\Hur(\mcR;\Q_+)_{\nu+\nu'}$. We can therefore apply Poincar\'e-Lefschetz duality
and reduce to proving commutativity of the following diagram, where we use the abbreviation $\Hur$ for
$\Hur(\mcR;\Q_+)$:
 \[
  \begin{tikzcd}
 H_0(\Hur_{\nu+\nu'}\setminus\dellrfat \Hur_{\nu+\nu'};\bQ)\ar[r]\ar[d,"\cong"] & H_0(\Hur_{\nu+\nu'}\setminus\delfat \Hur_{\nu+\nu'};\bQ)\ar[d,"\cong"]\\
 \bQ\qua{\Q_{\nu,\nu'}}[0] \ar[r,"\sca{-}\cdot\sca{-}"] & \bigoplus_{c\in\Q_{\nu+\nu'}}\bQ.
  \end{tikzcd}
 \]
 Here the bottom right group is abstractly isomorphic to $\bQ[\Q]_{2\nu+2\nu'}$, but lives naturally in degree 0.
 The vertical isomorphisms are given as follows:
 \begin{itemize}
  \item the left vertical map comes from the identification of $\Hur_{\nu+\nu'}\setminus\dellrfat \Hur_{\nu+\nu'}$ with
  \[
   \coprod_{(a,b)\in\Q_{\nu,\nu'}}
   \Hur\pa{\frac{7z_{\mathrm{c,l}}}8+\frac 18\mcR;\Q_+}_a\times \Hur\pa{\frac{7z_{\mathrm{c,r}}}8+\frac 18\mcR;\Q_+}_b
  \]
  induced by the maps $\fri^{\C}_{\bS_{0,1/2}}$ and $\fri^{\C}_{\bS_{1/2,1}}$; the second space
  has connected components in bijection with $\Q_{\nu,\nu'}$, by taking the total monodromies of the two factors;
  \item the right vertical map comes from the identification of $\Hur_{\nu+\nu'}\setminus\delfat \Hur_{\nu+\nu'}$ with
  \[
  \Hur\pa{\frac{\zcentre}2+\frac 12\mcR;\Q_+}_{\nu+\nu'};
  \]
  the second space has connected components in bijection with $\Q_{\nu+\nu'}$ by taking the total monodromy.
 \end{itemize}
Commutativity of the last diagram follows from the observation that for all $(a,b)\in\Q_{\nu,\nu'}$, if we set
$c=ab\in\Q$, then the inclusion of $\Hur_{\nu+\nu'}\setminus\dellrfat \Hur_{\nu+\nu'}$ 
into $\Hur_{\nu+\nu'}\setminus\delfat \Hur_{\nu+\nu'}$ restricts to an inclusion
\[
\Hur\pa{\frac{7z_{\mathrm{c,l}}}8+\frac 18\mcR;\Q_+}_a\times \Hur\pa{\frac{7z_{\mathrm{c,r}}}8+\frac 18\mcR;\Q_+}_b
\subset  \Hur\pa{\frac{\zcentre}2+\frac 12\mcR;\Q_+}_c.
\]
\end{proof}
\subsection{Stable rational cohomology of classical Hurwitz spaces}
We apply Theorem \ref{thm:HbBQG} in the case of a finite
PMQ $\Q$ with trivial multiplication: recall from \cite[Example 6.20]{Bianchi:Hur1} that every PMQ with trivial product
is Poincar\'e, in particular it is $\bQ$-Poincar\'e.
The algebra $\cA(\Q)$ is isomorphic in this case to $\bQ[x_S\,|\, S\in\conj(\Q_+)]/(x_S^2)$,
i.e. the quotient of the polynomial ring over $\bQ$ with one variable $x_S$ in degree two
for each conjugacy class $S\subset \Q_+$, modulo the ideal generated by the squares $x_S^2$.
A minimal Sullivan model for $\cA(\Q)$ is given by the commutative differential graded algebra $(\bA(\Q),d)$, where
\[
\bA(\Q)=\bQ[x_S\,|\, S\in\conj(\Q_+)]\otimes\Lambda_{\bQ}[y_S\,|\, S\in\conj(\Q_+)],
\]
with $x_S$ in degree two and $y_S$ in degree three, and where the unique non-trivial differentials
are $d(y_S)=x_S^2$ for all $S\in\conj(\Q_+)$.
Looping twice the minimal Sullivan model (i.e. decreasing all degrees of the $x_S$ and $y_S$ by 2),
and restricting to one connected component, we obtain that the rational cohomology
of $\Omega_0B\mHurm(\Q)$ is isomorphic to $\Lambda_{\bQ}[y'_S\,|\, S\in\conj(\Q_+)]$,
i.e. it is the free exterior algebra over $\bQ$ generated by classes $y'_S$ in degree one,
one for each $S\in\conj(\Q_+)$: the class $y'_S$ is obtained by looping twice $y_S\in\bA(\Q)$.

There is a special case of interest, namely when $\Q$ has the form $c\sqcup\set{\one_{\Q}}$,
for $c\subset G$ a finite, conjugacy invariant subset of a group $G$: then by \cite[Proposition 7.3]{Bianchi:Hur2} the space $\mHurm_+(\Q)$ is homotopy equivalent to the topological monoid
$\Hur^c_G:=\coprod_{n\ge0}\Hur^c_{G,n}$ of classical Hurwitz spaces with monodromies in $c$,
see \cite[Subsection 2.6]{EVW:homstabhur} and \cite[Subsection 4.2]{ORW:Hurwitz}. Adding a disjoint
unit, we obtain a homotopy equivalence of topological monoids
$\mHurm(\Q)\simeq \set{\one}\sqcup\Hur^c_G$; we can now recall that the weak homotopy type of the group completion of a topological monoid does not change up to homotopy if we add a disjoint unit to the monoid and, thus, we obtain
weak equivalences
\[
 \Omega B\mHurm(\Q)\simeq \Omega B\pa{\set{\one}\sqcup\Hur^c_G}\simeq \Omega B\Hur^c_G.
\]
Thus we obtain a computation of the rational cohomology of $\Omega B\Hur^c_G$,
which is computed already in \cite[Corollary 5.4]{ORW:Hurwitz};
a claim of the result already appears in the withdrawn preprint \cite{EVW:homstabhurII},
as a combination of the statement of \cite[Theorem 2.8.1]{EVW:homstabhurII} for $n=2$ and $X=BG$, and the discussion in \cite[Subsection 5.6]{EVW:homstabhurII}. See also the conjecture in
\cite[Subsection 1.5]{EVW:homstabhur}, which predicts the previous computation for $c$ being the conjugacy class of transpositions in a symmetric group on at least three letters.

\appendix
\section{Deferred proofs}
\subsection{Proof of Proposition \ref{prop:Hurmtopmon}}
\label{subsec:prop:Hurmtopmon}
The two cases $\mHurm$ and $\bHurm$ are analogous, so we will 
focus on the case of $\bHurm$, which is slightly more difficult.

Recall Definition \ref{defn:Hurmoore}. We define in a symmetric way $\bHurm^-$
as the subspace of $[0,\infty)\times\Hur(\bcR_{\R},\bdel)$
containing couples $(t,\fc)$ with $\fc$ supported in $\mathring{\bS}_{-t,0}$.

Note that $\bHurm^-$ is contained in the subspace
$[0,\infty)\times\Hur(\bcR_{\R^-},\bdel)$,
where we define $(\bcR_{\R^-},\bdel)$ as the nice couple 
$\big((-\infty,0)\times[0,1],(-\infty,0)\times\set{0,1}\big)$.

By \cite[Proposition 4.4]{Bianchi:Hur2} the assignment $(t,\fc)\mapsto (t,(\tau_{-t})_*(\fc))$
gives a continous map $\tau^-\colon\bHurm\to[0,\infty)\times\Hur(\bcR_{\R},\bdel)$;
note that $\tau^-$ has values in the subspace $\bHurm^-$; in fact $\tau^-$ is a homeomorphism
$\bHurm\cong\bHurm^-$.

Recall \cite[Definition 3.16]{Bianchi:Hur2}. The following composition of continous
maps takes values in
$\bHurm\subset[0,\infty)\times \Hur(\bcR_{\R},\bdel)$
and it coincides with $\mu\colon\bHurm\times\bHurm\to\bHurm$ as a map of sets:
\[
 \begin{tikzcd}[column sep=.5cm]
\bHurm\times \bHurm\ar[r,"\tau^-\times\Id"] & \bHurm^-\times\bHurm \ar[ld,"\subset"']\\
\left[0,\infty\right)\times\Hur(\bcR_{\R^-},\bdel)\times\left[0,\infty\right)\times\Hur(\bcR_{\infty},\bdel)
\ar[d,"\Id\times \fri^{\C}_{\bS_{-\infty,0}}\times \Id\times\fri^{\C}_{\bS_{0,\infty}}"] & \\
\left[0,\infty\right)\times\Hur^{\bS_{-\infty,0}}(\bcR_{\R^-},\bdel)\times\left[0,\infty\right)\times\Hur^{\bS_{0,\infty}}(\bcR_{\infty},\bdel)
\ar[d,"-\sqcup-"] & \\
\left[0,\infty\right)\times\left[0,\infty\right)\times \Hur^{\bS_{-\infty,\infty}}(\bcR_{\R^-}\cup\bcR_{\infty},\bdel)
\ar[d,equal] &\\
\left[0,\infty\right)\times\left[0,\infty\right)\times \Hur(\bcR_{\R^-}\cup\bcR_{\infty},\bdel)
\ar[d,"\subset"]& \\
\left[0,\infty\right)\times\left[0,\infty\right)\times \Hur(\bcR_{\R},\bdel)\ar[r,"\hat\tau"] &
\left[0,\infty\right)\times \Hur(\bcR_{\R},\bdel).
 \end{tikzcd}
\]
Here, by abuse of notation, we denote by $-\sqcup-$ the map $(t,\fc,t',\fc')\mapsto(t,t',\fc\sqcup\fc')$;
the map $\hat\tau$ is defined by $(t,t',\fc)\mapsto (t+t',(\tau_t)_*(\fc))$. This shows continuity
of the product $\mu$.

To prove associativity of $\mu$, let $\fc_i=(P_i,\psi_i,\phi_i)$ be a configuration in $\Hur(\bcR_{t_i},\bdel)$
for $i=1,2,3$.
Under the identification $\fri^{\C}_{\bS_{0,t_i}}$
we can regard $\fc_i$ as a configuration in $\Hur^{\bS_{0,t_i}}(\bcR_{t_i},\bdel)$.
Then $(\tau_{t_1})_*(\fc_2)$
is a configuration in $\Hur^{\bS_{t_1,t_1+t_2}}(\bcR_{t_2}+t_1,\bdel)$
and $(\tau_{t_1+t_2})_*(\fc_3)$ is a configuration in $\Hur^{\bS_{t_1+t_2,t_1+t_2+t_3}}(\bcR_{t_3}+t_1+t_2,\bdel)$).
The compositions
$\fc_1\sqcup\big((\tau_{t_1})_*(\fc_2)\sqcup (\tau_{t_1+t_2})_*(\fc_3)\big)$
and $\big(\fc_1\sqcup(\tau_{t_1})_*(\fc_2)\big)\sqcup (\tau_{t_1+t_2})_*(\fc_3)$
represent the same configuration in
$\Hur^{\bS_{0,t_1+t_2+t_3}}\pa{\bcR_{t_1}\cup\bcR_{t_2}+t_1\cup\bcR_{t_3}+t_1+t_2,\bdel}$,
and by inclusion and change of ambient to $\C$, the same configuration in $\Hur(\bcR_{t_1+t_2+t_3},\bdel)$.
This proves associativity of $\mu$.

The fact that $(0,(\emptyset,\one,\one))$ is a two-sided unit for $\mu$ follows directly from Definition \ref{defn:muttp}, together with the fact that $\tau_0$ is the identity of $\C$.

\subsection{Proof of Lemma \ref{lem:fcagenerate}}
\label{subsec:fcagenerate}
We start by proving that the elements $\pi_0(1,\fc_a)$ generate $\pi_0(\mHurm)$ as a monoid.
Let $(t,\fc)\in\mHurm$ and use Notation \ref{nota:fc}.
If $P=\emptyset$, then we can continuously
reduce $t$ to 0, so that $\pi_0(t,\fc)=\pi_0(0,(\emptyset,\one,\one))$ is
the neutral element of $\pi_0(\mHurm)$.
Suppose from now on that $P\neq\emptyset$.
 
 Suppose first that $P=\set{z}$ consists of a single point.
 By Lemma \ref{lem:HurmvsHur} we can connect $(t,\fc)$ with a configuration
 of the form $(1,\fc')$;
 we can then connect $(1,\fc')$ with a configuration
 of the form $\fc_a$: for this we can use any homotopy of $\C$ through semialgebraic
 homeomorphisms which is at all times
 supported on $\mcR$ (i.e. $\C\setminus\mcR$ is fixed pointwise at all times),
 and pushes the unique point $z'\in P'$ to $\zcentre$.
 It follows that $(t,\fc)$ and $(1,\fc_a)$ are connected in $\mHurm$.
 
 Suppose now that $\abs{P}\geq 2$. We can perturb the positions of the points $z_i\in P$ inside
 $\mcR_t$ and assume that their real parts $\Re(z_i)$ are all different: again, we use a
 semialgebraic isotopy
 of $\C$ supported on $\mcR_t$, starting from the identity of $\C$ and ending with a
 semialgebraic homeomorphism of $\C$ mapping $P$
 to a set of points with distinct real parts.
 
 Without loss of generality we assume that $P$ already has this property and we also assume
 $\Re(z_1)<\dots<\Re(z_k)$; choose positive real numbers $0=t_0<t_1<\dots<t_k=t$ such that
 $t_{i-1}<\Re(z_i)<t_i$ for all $1\leq i\leq k$.
 In particular we can regard $\fc$ as a configuration in
 \[
 \Hur(\mcR_{t_1}\cup \mcR_{t_2-t_1}+t_1\cup\dots\cup\mcR_{t_k-t_{k-1}}+t_{k-1}).
 \]
 
 Recall \cite[Definition 3.15]{Bianchi:Hur2}, and for all $1\leq i\leq k$ let $\fc_i\in\Hur(\mcR_{t_i-t_{i-1}})$ be the image of $\fc$ along the following composition:
 \[
 \begin{tikzcd}
 \Hur(\mcR_{t_1}\cup \mcR_{t_2-t_1}\!+\!t_1\cup\dots\cup\mcR_{t_k-t_{k-1}}\!+\!t_{k-1})\ar[r,"\fri_{\bS_{t_{i-1},t_i}}^{\C}"]
 & \Hur^{\bS_{t_{i-1},t_i}}(\mcR_{t_i-t_{i-1}}\!+\!t_{i-1})\ar[dl,"(\fri_{\bS_{t_{i-1},t_i}}^{\C})^{-1}"]\\
 \Hur(\mcR_{t_i-t_{i-1}}\!+\!t_{i-1})
 \ar[r,"(\tau_{-t_{i-1}})_*",'] & \Hur(\mcR_{t_i-t_{i-1}}).
 \end{tikzcd}
 \]
 Then $(t,\fc)$ is equal to the product $(t_1-t_0,\fc_1)\cdot\dots\cdot(t_k-t_{k-1},\fc_k)$
 in $\mHurm$, and each $\fc_i$ is supported on the single point $\tau_{-t_{i-1}}(z_i)$. 
 It follows that the elements $\pi_0(1,\fc_a)$ generate $\pi_0(\mHurm)$ as a monoid.
 
 Now let $a,b\in\Q$, and note that $(1,\fc_a)\cdot(1,\fc_b)$ has the form $(2,\fc)$,
 for some $\fc=(P,\psi)\in\Hur(\mcR_2)$ with
 $P=\set{z_1=\zcentre,z_2=\zcentre+1}$.
 Let $\gen_1,\gen_2$ be the admissible generating set for $\fG(P)$ with $\gen_1$ represented by a loop supported on $\bS_{-\infty,1}$
 and $\gen_2$ represented by a loop supported on $\bS_{1,+\infty}$.
 
 We can fix a semialgebraic isotopy
 $\cH\colon\C\times[0,1]\to\C$ supported on $\mcR_2$,
 starting from the identity of $\C$ and swapping at time 1 the two points
 of $P$ according to a clockwise half Dehn twist. We use the notation $\fc'=(P,\psi'):=\cH_*(\fc;1)\in\Hur(\mcR_2)$.
 
 The homeomorphism $\cH(-,1)\colon\CmP\to\CmP$ induces an automorphism of the fundamental group
 $\cH(-,1)_*\colon \fG(P)\to\fG(P)$ which restricts to an automorphism
 of the fundamental PMQ $\cH(-,1)_*\colon\fQ_{\mcR_2}(P)\to\fQ_{\mcR_2}(P)$.
 
 Note that $\cH(-,1)_*$ sends $\gen_2\mapsto\gen_1$ and $\gen_1^{\gen_2}\mapsto \gen_2$.
 By definition we have $\psi'=\psi\circ \cH(-,1)_*^{-1}$, hence $\psi'(\gen_1)=\psi(\gen_2)=b$
 and $\psi'(\gen_2)=\psi(\gen_1^{\gen_2})=a^b$. It follows that
 $(P,\psi')=(1,\fc_b)\cdot(1,\fc_{a^b})$,
 hence $\pi_0(1,\fc_a)\cdot\pi_0(1,\fc_b)=\pi_0(1,\fc_{b})\cdot\pi_0(1,\fc_{a^b})$.
 
 Suppose now that the product $ab$ is defined in $\Q$. Again by Lemma \ref{lem:HurmvsHur} we can connect
 $(2,\fc):=(1,\fc_a)\cdot(1,\fc_b)$ to a configuration of the form $(1,\fc'')$, with
 $\fc''\in\Hur(\mcR)$; for instance we can take $\fc''=\Lambda_*(\fc,1/2)$,
 where the semialgebraic isotopy $\Lambda\colon\C\times(0,\infty)\to\C$ was introduced in the proof of Lemma \ref{lem:HurmvsHur}.
 In fact, we have $\fc''\in\Hur_+(\mcR)_{ab}$; by \cite[Corollary 6.5]{Bianchi:Hur2} the space 
 $\Hur_+(\mcR)_{ab}$ is contractible, in particular it is connected; hence $\pi_0(1,\fc_a)\cdot\pi_0(1,\fc_b)=\pi_0(1,\fc'')=\pi_0(1,\fc_{ab})$.

\subsection{Proof of Proposition \ref{prop:BbHurmbHurm+contractible}}
\label{subsec:BbHurmbHurm+contractible}
\begin{defn}
 \label{defn:bHurmfl}
 We introduce a subspace $\bHurmfl\subset\bHurm_+$. A couple $(t,\fc)\in\bHurm_+$ belongs to $\bHurmfl$
 if $t\geq 1$ and the point $\zfl_t:=t-\frac 12$ belongs
 to the support of $\fc$.
\end{defn}

Note that $\bHurmfl$ is invariant under the action of $\bHurm$ by left multiplication.

\begin{defn}
 \label{defn:bHurmsh}
Let $\fc_{\one}^{\mathrm{d}}\in \Hur(\bcR,\bdel)$ be as in the proof of Lemma \ref{lem:bHurmweonbHurm+},
and note that $(1,\fc_{\one}^{\mathrm{d}})\in\bHurmfl$; in fact,
the right multiplication map $\mu(-;(1,\fc_{\one}^{\mathrm{d}}))\colon\bHurm\to\bHurm$
has image inside $\bHurmfl$. We define $\bHurmsh\subset\bHurmfl$
as the image of $\mu(-;(1,\fc_{\one}^{\mathrm{d}}))$.
\end{defn}

\begin{figure}[ht]
 \begin{tikzpicture}[scale=3.5,decoration={markings,mark=at position 0.38 with {\arrow{>}}}]
  \draw[dashed,->] (-.05,0) to (1.7,0);
  \draw[dashed,->] (0,-1.1) to (0,1.1);
  \node at (0,-1) {$*$};
  \draw[dotted, fill=gray, opacity=.5] (0,0) rectangle (1.6,1);
  \draw[very thick] (0,0) to (1.6,0);
  \draw[very thick] (0,1) to (1.6,1);
  \draw[dotted] (.6,0) to (.6,1);
  \draw[dotted] (1.1,0) to (1.1,1);
  \node at (.1,.3){$\bullet$}; 
  \node at (.3,1){$\bullet$}; 
  \node at (.45,.5){$\bullet$};
  \node at (1.1,.0){$\bullet$};\node at (1.05,.08){\tiny$\zfl_t$};
  \node at (1.45,.8){$\bullet$};
  \draw[thin, looseness=1.2, postaction={decorate}] (0,-1) to[out=82,in=-90]  (.05,.1) to[out=90,in=-90] (.01,.3)  to[out=90,in=90] node[above]{\tiny$a_1$} (.2,.3)  to[out=-90,in=90] (.1,.1) to[out=-90,in=80] (0,-1);
  \draw[thin, looseness=1.2, postaction={decorate}] (0,-1) to[out=79,in=-90] (.28,.9)
  to[out=90,in=-90] (.26,1) to[out=90,in=90] (.34,1) to[out=-90,in=90] node[left]{\tiny$g_2$}
  (.3,.9) to[out=-90,in=77]  (0,-1);
  \draw[thin, looseness=1.2, postaction={decorate}] (0,-1) to[out=76,in=-90] (.4,.3)
  to[out=90,in=-90] (.35,.5) to[out=90,in=90] node[above]{\tiny$a_3$} (.49,.5) to[out=-90,in=90] (.45,.3) to[out=-90,in=74]  (0,-1);
  \draw[thin, looseness=1.2, postaction={decorate}] (0,-1) to[out=52,in=-120]  (1.06,-.1) to[out=60,in=-90] (.9,.05)  to[out=90,in=90] node[above]{\tiny$g_4\neq\one$} (1.2,.05)  to[out=-90,in=50] (1.15,-.1) to[out=-130,in=44] (0,-1);
  \draw[thin, looseness=1.2, postaction={decorate}] (0,-1) to[out=36,in=-90] (1.4,.7)
  to[out=90,in=-90] (1.35,.8) to[out=90,in=90] node[above]{\tiny$a_5$} (1.49,.8) to[out=-90,in=90]
  (1.45,.7) to[out=-90,in=34]  (0,-1);
 
\begin{scope}[shift={(1.8,0)}]
  \draw[dashed,->] (-.05,0) to (1.7,0);
  \draw[dashed,->] (0,-1.1) to (0,1.1);
  \node at (0,-1) {$*$};
  \draw[dotted, fill=gray, opacity=.5] (0,0) rectangle (1.6,1);
  \draw[very thick] (0,0) to (1.6,0);
  \draw[very thick] (0,1) to (1.6,1);
  \draw[dotted] (.6,0) to (.6,1);
  \draw[dotted] (1.1,0) to (1.1,1);
  \node at (.1,.3){$\bullet$}; 
  \node at (.3,1){$\bullet$}; 
  \node at (.45,.5){$\bullet$};
  \node at (1.1,.0){$\bullet$};\node at (1.05,.08){\tiny$\zfl_t$};
  \draw[thin, looseness=1.2, postaction={decorate}] (0,-1) to[out=82,in=-90]  (.05,.1) to[out=90,in=-90] (.01,.3)  to[out=90,in=90] node[above]{\tiny$a_1$} (.2,.3)  to[out=-90,in=90] (.1,.1) to[out=-90,in=80] (0,-1);
  \draw[thin, looseness=1.2, postaction={decorate}] (0,-1) to[out=79,in=-90] (.28,.9)
  to[out=90,in=-90] (.26,1) to[out=90,in=90] (.34,1) to[out=-90,in=90] node[left]{\tiny$g_2$}
  (.3,.9) to[out=-90,in=77]  (0,-1);
  \draw[thin, looseness=1.2, postaction={decorate}] (0,-1) to[out=76,in=-90] (.4,.3)
  to[out=90,in=-90] (.35,.5) to[out=90,in=90] node[above]{\tiny$a_3$} (.49,.5) to[out=-90,in=90] (.45,.3) to[out=-90,in=74]  (0,-1);
  \draw[thin, looseness=1.2, postaction={decorate}] (0,-1) to[out=52,in=-120]  (1.06,-.1) to[out=60,in=-90] (.9,.05)  to[out=90,in=90] node[above]{\tiny$g_4=\one$} (1.2,.05)  to[out=-90,in=50] (1.15,-.1) to[out=-130,in=44] (0,-1);
\end{scope}
 \end{tikzpicture}
 \caption{Two configurations in $\bHurmfl$; note that only the right one belongs to $\bHurmsh$.}
\label{fig:flsh}
\end{figure}

See Figure \ref{fig:flsh}. Note that also the subspace $\bHurmsh$ is invariant under the action of $\bHurm$
by left multiplication.
Moreover, $\mu(-,(1,\fc_{\one}))\colon\bHurm\to\bHurmsh$ is a homeomorphism
of $\bHurm$-left modules, and $\bHurmsh\subset\bHurmfl$ contains
precisely all couples $(t,\fc)$ such that, using Notation \ref{nota:fc}, the following hold:
 \begin{itemize}
  \item $\set{\zfl_t}\subseteq P\subset\set{\zfl_t}\cup\bcR_{t-1}$;
  \item $\psi\colon\fQ(P)\to\Q$ factors through $\fQ(P\setminus\set{\zfl_t})$
 along the point-forgetting map $\fri^P_{P\setminus\set{\zfl_t}}\colon\fQ(P)\to\fQ(P\setminus\set{\zfl_t})$,
 see \cite[Notation 2.17]{Bianchi:Hur2};
 \item similarly, $\phi\colon\fG(P)\to G$ factors through $\fG(P\setminus\set{\zfl_t})$
 along $\fri^P_{P\setminus\set{\zfl_t}}\colon\fG(P)\to\fG(P\setminus\set{\zfl_t})$.
 \end{itemize}
Passing to bar constructions, we obtain inclusions of spaces
\[
 B(\bHurm,\bHurmsh)\hookrightarrow B(\bHurm,\bHurmfl)\hookrightarrow B(\bHurm,\bHurm_+).
\]
The first space $B(\bHurm,\bHurmsh)$ is homeomorphic
to $E\bHurm$ and, hence, is contractible, as $\bHurm$ is a unital monoid.
We will prove that the inclusions $B(\bHurm,\bHurmfl)\hookrightarrow B(\bHurm,\bHurm_+)$ 
and $B(\bHurm,\bHurmsh)\hookrightarrow B(\bHurm,\bHurmfl)$ are weak homotopy equivalences:
this will conclude the proof of Proposition \ref{prop:BbHurmbHurm+contractible}.

\begin{lem}
 \label{lem:firstcontraction}
 The inclusion $\bHurmfl\subset\bHurm_+$ is a homotopy equivalence.
\end{lem}
\begin{proof}
The argument is similar to that of the proof of Lemma \ref{lem:bHurmweonbHurm+}, but a bit more care is needed.

We define a continous map
$\expl^{\flat}\colon(0,\infty)\times \Ran_+(\bcR_\infty)\times[0,1]\to \Ran_+(\bcR_\infty)$
by the formula
\[
 \expl^{\flat}(t,\set{z_1,\dots,z_k},s)=\set{(1-s)z_1+st,\dots,(1-s)z_k+st}.
\]

We also let $\fj\colon(0,\infty)\to(0,\infty)$ be given by the formula
\[
\fj(t)=\left\{
\begin{array}{ll}
t-\frac 12 &\mbox{for }t\ge 1\\
\frac t2 &\mbox{for }t\le 1.
\end{array}
\right.
\]
We consider then the homotopy $\cH^{\flat}\colon\bHurm_+\times[0,1]\to(0,\infty)\times\Hur(\bcR_{\infty},\bdel)$ given by the formula
\[
 \cH^{\flat}(t,\fc;s)=\Big(t\ ,\ \fc\times \expl^{\flat}\pa{\fj(t),\epsilon(\fc), s} \Big).
\]

Let $\bHurm_{+,t\geq 1}$ denote the subspace of $\bHurm_+$ containing all couples
$(t,\fc)$ with $t\geq1$. Then the homotopy $\cH^{\flat}$ has the following properties:
\begin{itemize}
 \item $\cH^{\flat}(-,s)$ sends $\bHurm_+$ inside $\bHurm_+$ for all $0\le s \le 1$;
 \item $\cH^{\flat}(-,0)$ is the identity of $\bHurm_+$;
 \item $\cH^{\flat}(-,s)$ preserves the subspaces $\bHurmfl$ and $\bHurm_{+,t\geq 1}$ at all times;
 \item $\cH^{\flat}(-,1)$ sends $\bHurm_{+,t\geq 1}$ inside $\bHurmfl$.
\end{itemize}
This shows that the inclusion $\bHurmfl\hookrightarrow\bHurm_{+,t\geq 1}$ is a homotopy equivalence.
Moreover there is a deformation retraction of
$\bHurm_+$ onto $\bHurm_{+,t\geq 1}$ given by the formula
\[
 ((t,\fc);s)\mapsto \mu\big(\ (\max\set{0,s-t},(\emptyset,\one,\one))\ ,\ (t,\fc)\ \big).
\]
It follows that the inclusion $\bHurmfl\subset\bHurm_+$ is a homotopy equivalence.
\end{proof}
By Lemma \ref{lem:firstcontraction} the inclusion of semi-simplicial spaces
\[
B_{\bullet}(\bHurm,\bHurmfl)\subset B_{\bullet}(\bHurm,\bHurm_+)
\]
is levelwise a homotopy equivalence; after (thick) geometric realisation we obtain the following corollary.
\begin{cor}
 \label{cor:firstwebHurm}
The inclusion $ B(\bHurm,\bHurmfl)\subset B(\bHurm,\bHurm_+)$ is a weak homotopy equivalence.
\end{cor}

 For the second step of the proof of Proposition \ref{prop:BbHurmbHurm+contractible}
 we define a homotopy
 \[
 \cH^{\sharp}\colon\bHurmfl\times[0,1]\to\bHurmfl
 \]
 by setting
  $\cH^{\sharp}\colon((t,\fc);s)\mapsto (t+s, \fc\times\set{\zfl_{t+s}})$
 for $0\leq s\leq 1$ and $(t,\fc)\in\bHurmfl$; here
$\fc\times\set{\zfl_{t+s}}$ is the evaluation at
$(\fc,\set{\zfl_{t+s}})$ of the external product 
  \[
  -\times-\colon\Hur\pa{\bcR_{\infty},\bdel} \times
  \Ran(\bcR_{\infty}) \to \Hur\pa{\bcR_{\infty},\bdel}.
  \]
Roughly speaking, the homotopy $\cH^{\sharp}$ has the following effects on configurations
 $(t,\fc)\in\bHurmfl$:
 \begin{itemize}
  \item it increases by 1 the first component $t$ of a couple $(t,\fc)\in\bHurmfl$;
  \item it splits $\zfl_t$, which already belongs to the support of $\fc$, into two points; one point remains at the position
  $\zfl_t$ and ``keeps the local monodromy'' (which is only defined as an element of $G$,
  because $\zfl_t\in\bdel\bcR_{\infty}$); the other point moves gradually
  to a distance 1 to right and has trivial local monodromy $\one$
  (also only defined as element of $G$).
 \end{itemize}
Note that $\cH^{\sharp}$ has the following properties:
 \begin{itemize}
  \item $\cH^{\sharp}(-;0)$ is the identity of $\bHurmfl$;
  \item $\cH^{\sharp}(-;1)$ has values inside $\bHurmsh$;
  \item for all $0\leq s\leq 1$ the map $\cH^{\sharp}(-;s)$ is
  equivariant with respect to the action of $\bHurm$ on $\bHurmfl$
  by left multiplication.
 \end{itemize}
 As a consequence, $\cH^{\sharp}$ induces a homotopy of $B(\bHurm_+,\bHurmfl)$ starting from the identity and ending with
 a map $B(\bHurm_+,\bHurmfl)\to B(\bHurm_+,\bHurmsh)$.
 We obtain the following lemma, which is the last step needed to prove Proposition \ref{prop:BbHurmbHurm+contractible}.
 \begin{lem}
  \label{lem:seconcontraction}
  The space $B(\bHurm_+,\bHurmfl)$ admits a deformation into
  its contractible subspace $B(\bHurm_+,\bHurmsh)$. As a consequence
  $B(\bHurm_+,\bHurmfl)$ is weakly contractible.
 \end{lem}

\subsection{Proof of Proposition \ref{prop:sigmaquotient}}
\label{subsec:sigmaquotient}
We first give a rough idea of the proof: the value of $\check\sigma$ at a given sequence $(\uw;\ut,\ufc)$ in $\Delta^p\times\pa{\mHurm}^p$ (respectively, in $\Delta^p\times\pa{\bHurm}^p$) is obtained by combining several steps: the first step is computing the product $\hat\mu(\uw;\ut,\ufc)=(t_1,\fc_1)\cdot\dots\cdot(t_p,\fc_p)$ in $\mHurm$ (in $\bHurm$). The product map $\hat\mu$, however, does not factor through $B\mHurm$ (respectively, $B\bHurm$): one of the reasons is that if $w_0=0$, the sequence $(\uw;\ut,\ufc)$ is equivalent in $B\mHurm$ (respectively, in $B\bHurm$) to the sequence obtained by forgetting $w_0$ and $(t_1,\fc_1)$; yet the product 
$(t_1,\fc_1)\cdot\dots\cdot(t_p,\fc_p)$ is, in general, different from the product $(t_2,\fc_2)\cdot\dots\cdot(t_p,\fc_p)$. In fact, letting $0\le i_{\min}\le i_{\max}\le p$ be as in the discussion before Definition \ref{defn:epsilon}, we might reduce $(\uw;\ut,\ufc)$ to an equivalent sequence $(w_{i_{\min}},\dots,w_{i_{\max}};t_{i_{\min}+1},\dots,t_{i_{\max}};\fc_{i_{\min}+1},\dots,\fc_{i_{\max}})$, yet the product $(t_1,\fc_1)\cdot\dots\cdot(t_p,\fc_p)$ is, in general, different from the subproduct $(t_{i_{\min}+1},\fc_{i_{\min}+1})\cdot\dots\cdot(t_{i_{\max}},\fc_{i_{\max}})$.

We solve this problem by using the barycentres $b^-_\epsilon<b^+_\epsilon$, and by considering only the part of the configuration $(t_1,\fc_1)\cdot\dots\cdot(t_p,\fc_p)$ lying in the strip $\bS_{b^-_\epsilon,b_\epsilon^+}$: more precisely, using a suitable expansion and translation (implemented via the homotopies $\tau_*$ and $\Lambda_*$), we map the strip $\bS_{b^-_\epsilon,b_\epsilon^+}$ to the strip $\bS_{0,1}$, and then we collapse the two parts of the obtained configuration lying outside the strip $\bS_{0,1}$: we first collapse all points lying on left (respectively, on right) of $\bS_{0,1}$ to the segment $0\times[0,1]$ (respectively, $1\times[0,1]$) via the homotopies $\kappa^+_*$ and $\kappa^-_*$, we further collapse these vertical segments to the two points $\zdiamleft$ and $\zdiamright$ via $(\cHdiamo_1)_*$, and finally we get rid of the residual information of the $G$-valued local monodromy at $\zdiamleft$ and $\zdiamright$ by quotienting by the action of $G\times G^{op}$. The key observation is that the part of $(t_1,\fc_1)\cdot\dots\cdot(t_p,\fc_p)$ lying in the strip $\bS_{b^-_\epsilon,b_\epsilon^+}$ only depends on the subproduct $(t_{i_{\min}+1},\fc_{i_{\min}+1})\cdot\dots\cdot(t_{i_{\max}},\fc_{i_{\max}})$. The conditions on $\epsilon$, moreover, ensure that either $\epsilon=0$, i.e. $b_\epsilon^-=b^-$ and $b_\epsilon^+=b^+$, or the part of $\hat\mu(\uw;\ut,\ufc)$ lying inside $\bS_{b^-_\epsilon,b_\epsilon^+}$ is ``empty'' and, in particular, independent of the positive value of $\epsilon$; this is the rough reason why $\check\sigma$ and, hence, $\sigma$ and $\bar\sigma$, do not depend on $\epsilon$. 

Before starting the proof of Proposition \ref{prop:sigmaquotient}, we give a definition.
  \begin{defn}
  \label{defn:agreeing}
Let $\fT$ be a nice couple, let $\fc,\fc'\in\Hur(\fT;\Q,G)$ be two configurations, and let $*\in\bT\subset\C$ be a contractible subspace. We say that $\fc$ and $\fc'$ \emph{agree on $\bT$} if the following hold, using Notation \ref{nota:fc}:
\begin{itemize}
  \item $P\cap \bT=P'\cap\bT$;
  \item for every loop $\gamma\subset\bT\setminus P$ representing an element
  $[\gamma]$ in $\fQ_{\fT}(P)$ we have $\psi(\gamma)=\psi'(\gamma)$;
  \item for every loop $\gamma\subset\bT\setminus P$ representing an element
  $[\gamma]$ in $\fG(P)$ we have $\phi(\gamma)=\phi'(\gamma)$).
\end{itemize}
\end{defn}
Consider the particular case in which $\fT$ splits as a disjoint union of nice couples $\fT_1,\fT_2$
with $\fT_1\subset\mathring{\bT}$ and $\fT_2$ contained in the interior of $\C\setminus\bT$:
then \cite[Definition 3.15]{Bianchi:Hur2} gives configurations $\fri^{\C}_{\bT}(\fc)$ and
$\fri^{\C}_{\bT}(\fc')$ in $\Hur^{\bT}(\fT_1;\Q,G)$, and the condition that
$\fc$ and $\fc'$ agree on $\bT$ is equivalent to the equality $\fri^{\C}_{\bT}(\fc)=\fri^{\C}_{\bT}(\fc')$.
 
Now let $(\uw;\ut,\ufc)$ be as in Notation \ref{nota:uwutufc} and assume
$w_j=0$ for a fixed $0\leq j\leq p$;
let $(\ut',\ufc')=d_j(\ut,\ufc)$ (see Definition \ref{defn:BM}), let $\uw'$ be obtained from $\uw$ by removing the (vanishing) $j$\textsuperscript{th} coordinate,
and present $\uw'$ as $(w'_0,\dots,w'_{p-1})$, $\ut'$ as $(t'_1,\dots,t'_{p-1})$ and $\ufc'$ as $(\fc'_1,\dots,\fc'_{p-1})$.
We want to prove that $\check\sigma(\uw;\ut,\ufc)=\check\sigma(\uw';\ut',\ufc')$: this
implies that $\check\sigma$ descends to a map $\sigma$ defined on $B\mHurm$ (on $B\bHurm$).

Let $a_0,\dots,a_p,b,b^+,b^-$ be computed as in Definition \ref{defn:barycentres} with respect
to $(\uw;\ut,\ufc)$ and, similarly, let $a'_0,\dots,a'_{p-1},b',(b^+)',(b^-)'$ be  computed with respect
to $(\uw';\ut',\ufc')$. We observe that $\epsilon(\uw;\ut,\ufc)=0$ if and only if the support $P$ of $\hmu(\uw;\ut,\ufc)$ intersects non-trivially the strip $\bS_{b^-,b^+}$. Similarly, $\epsilon(\uw';\ut,\ufc)=0$ if and only if the support $P'$ of $\hmu(\uw';\ut',\ufc')$ intersects non-trivially $\bS_{(b^-)',(b^+)'}$. We now observe that for $j>0$ we have $P\cap\bS_{b^-,b^+}=P'\cap \bS_{(b^-)',(b^+)'}$, whereas for $j=0$ we have $\tau(P\cap\bS_{b^-,b^+},-t_1)=P'\cap \bS_{(b^-)',(b^+)'}$. In particular either intersection is empty if and only if the other is, or in other words, $\epsilon(\uw;\ut,\ufc)>0$ if and only if $\epsilon(\uw';\ut',\ufc')>0$. In this case, however, we have that both $\check\sigma(\uw;\ut,\ufc)$ and $\check\sigma(\uw';\ut',\ufc')$ coincide with the basepoint $\fclr$ of $\Hur(\bdiamolr,\bdel)_{G,G^{op}}$ (respectively, of $\Hur(\diamo,\del)_{G,G^{op}}$) and, in particular, they are equal. This proves also the first statement of Proposition \ref{prop:sigmaquotient}.

From now on we assume $\epsilon(\uw;\ut,\ufc)=\epsilon(\uw';\ut',\ufc')=0$, allowing us to use the lower and upper barycentres instead of their $\epsilon$-variations in the computations of the rest of the argument.
Note that $b-b^-=b'-(b^-)'$ and $b^+-b=(b^+)'-b'$.
In particular, in the following we assume $b^+-b^-=(b^+)'-(b^-)'>0$.
 
First, we give an alternative description of the configuration $\hmu^b(\uw;\ut,\ufc)$.
We regard $\fc_i$, for $1\leq i\leq p$, as a configuration in $\Hur^{\bS_{0,t_i}}(\mcR_{t_i})$
(in $\Hur^{\bS_{0,t_i}}(\bcR_{t_i},\bdel)$), and consider the configuration
\[
(\tau_{a_{i-1}-b^-})_*(\fc_i)\in\Hur^{\bS_{a_{i-1}-b^-,a_i-b^-}}(\mcR_{t_i}+a_i-b^-)
\]
\[
\pa{\mbox{respectively, }(\tau_{a_{i-1}-b^-})_*(\fc_i)
\in
\Hur^{\bS_{a_{i-1}-b^-,a_i-b^-}}(\bcR_{t_i}+a_i-b^-,\bdel)}.
\]
Using the disjoint union map $-\sqcup-$ from \cite[Definition 3.16]{Bianchi:Hur2} we obtain a configuration
\[
 \fc:=(\tau_{a_0-b^-})_*(\fc_1)\sqcup\dots\sqcup (\tau_{a_{p-1}-b^-})_*(\fc_p),
\]
which \emph{a priori} belongs to
\[
\Hur^{\bS_{a_0-b^-,a_p-b^-}}\big(\mcR_{t_1}+a_0-b^-\sqcup\dots\sqcup\mcR_{t_p}+a_{p-1}-b^-\big)
\]
\[
\Big(\mbox{respectively, }\Hur^{\bS_{a_0-b^-,a_p-b^-}}\big(\bcR_{t_1}+a_0-b^-\sqcup\dots\sqcup\bcR_{t_p}+a_{p-1}-b^-,\bdel\big)\Big),
\]
but can be naturally considered as a configuration in $\Hur(\mcR_{\R})$ (respectively, $\Hur(\bcR_{\R},\bdel)$).
The equality $(\tau_{-b^-})_*(\hmu(\uw;\ut,\ufc))=\fc$
follows directly from Definitions \ref{defn:muttp} and \ref{defn:hmub}. It follows then from Definition
\ref{defn:hmub} that $\hmu^b(\uw;\ut,\ufc)$ is equal to
$\Lambda_*\pa{\fc;\frac{1}{b^+-b^-}}$. In a similar way, we obtain a configuration
\[
  \fc':=(\tau_{a'_0-(b^-)'})_*(\fc'_1)\sqcup\dots\sqcup (\tau_{a'_{p-2}-(b^-)'})_*(\fc'_{p-1}),
\]
which we consider as a configuration in $\Hur(\mcR_\R)$ (in $\Hur(\bcR_\R,\bdel)$), and identifications
$\fc'=(\tau_{-(b^-)'})_*(\hmu(\uw';\ut',\ufc'))$ and $\hmu^b(\uw';\ut',\ufc')=\Lambda_*\pa{\fc';\frac{1}{(b^+)'-(b^-)'}}$

\begin{lem}
 \label{lem:agreeing}
The configurations $\fc$ and $\fc'$ agree on the vertical strip  $[0,b^+-b^-]\times\R$.
\end{lem}
\begin{proof}
We use Notation \ref{nota:fc} and argue the statement differently, depending on the value of $j$.
 \begin{itemize}
  \item If $1\leq j\leq p-1$, then $\fc=\fc'$.
  \item If $j=0$, note that $a_1-b^-=a'_0-(b^-)'\ge0$ and $a_p-b^-=a'_{p-1}-(b^-)'\ge b^+-b^-$; then both $\fc$ and $\fc'$ can be regarded as configurations in
  \[
  \Hur\big(\mcR_{t_1}+a_0-b^-\sqcup\mcR_{a_p-a_1}+a_1-b^-\big),
  \]
  \[
  \Big(\mbox{respectively, }\Hur\big(\bcR_{t_1}+a_0-b^-\sqcup\bcR_{a_p-a_1}+a_1-b^-,\bdel\big)\Big),
  \]
  and the restriction map $\fri^{\C}_{\bS_{a_1-b^-,a_p-b^-}}$ from \cite[Definition 3.15]{Bianchi:Hur2} sends $\fc$ and $\fc'$ to the same configuration in the space
  \[
  \Hur^{\bS_{a_1-b^-,a_p-b^-}}(\mcR_{a_p-a_1}+a_1-b^-)\quad
  \Big(\mbox{respectively, }\Hur^{\bS_{a_1-b^-,\infty}}(\mcR_{a_p-a_1}+a_1-b^-)\Big);
  \]
  this common image is essentially the configuration $\fc'$.
  Now we use that $\bS_{0,b^+-b^-}\subset \bS_{a_1-b^-,a_p-b^-}$ to conclude that
  $\fc$ and $\fc'$ also agree on $\bS_{0,b^+-b^-}$ and, finally, we notice that agreeing on $\bS_{0,b^+-b^-}$ is equivalent to agreeing on $[0,b^+-b^-]\times\R$, as all configurations are supported in $\bH$.  
  \item If $j=p$, note that $a_0=a'_0$,  $b^-=(b^-)'$ and $b^+=(b^+)'\le a_{p-1}=a'_{p-1}$;
  then both $\fc$ and $\fc'$ can be regarded as configurations in
  \[
  \Hur\big(\mcR_{a_{p-1}}+a_0-b^-\sqcup\mcR_{t_p}+a_{p-1}-b^-\big),
  \]
  \[
  \Big(\mbox{respectively, }\Hur\big(\bcR_{a_{p-1}}+a_0-b^-\sqcup\bcR_{t_p}+a_{p-1}-b^-,\bdel\big)\Big),
  \]
  and $\fri^{\C}_{\bS_{a_0-b^-,a_{p-1}-b^-}}$ sends $\fc$ and $\fc'$
  to the same configuration in the space
  \[
  \Hur^{\bS_{-\infty,a_{p-1}-b^-}}(\mcR_{a_{p-1}}+a_0-b^-)
  \]
  \[
  \Big(\mbox{respectively, }\Hur^{\bS_{-\infty,a_{p-1}-b^-}}(\mcR_{a_{p-1}}+a_0-b^-)\Big);
  \]
  this common image is essentially the configuration $\fc'$.
  Now we use that $\bS_{0,b^+-b^-}\subset \bS_{a_0-b^-,a_{p-1}-b^-}$ to conclude that
  $\fc$ and $\fc'$ also agree on $\bS_{0,b^+-b^-}$, and again we note that agreeing on $\bS_{0,b^+-b^-}$ is equivalent to agreeing on $[0,b^+-b^-]\times\R$.
 \end{itemize}
\end{proof}

By applying $\Lambda_*\pa{-;\frac{1}{b^+-b^-}}=\Lambda_*\pa{-;\frac{1}{(b^+)'-(b^-)'}}$
we obtain that, in a similar way, $\hmu^b(\uw;\ut,\ufc)$ and $\hmu^b(\uw';\ut',\ufc')$
agree on the vertical strip $[0,1]\times\R$.
 
Note that the vertical strip $[0,1]\times\R$ is preserved by the homotopies $\kappa^-,\kappa^+$ and $\cHdiamo$ at all times,
and the homotopies $\kappa^-,\kappa^+$ restrict even to the identity on $[0,1]\times\R$.
This, together with Lemma \ref{lem:checksigmaindiamo}, implies that
$\hmu^{\diamo}_{\bdel\bdiamolr}\fc:=\tilde\sigma(w_i;t_i,\fc_i)$ and
$\tilde\fc':=\hmu^{\diamo}_{\bdel\bdiamolr}(w'_i;t'_i,\fc'_i)$
are both supported on the set $\tilde P:=\cHdiamo_1(P)\cup \bdel\bdiamolr$. Let $\tilde k=|\tilde P|$ and $\tilde l=|P\setminus\del\diamo|$.
 
Recall that $\tilde\fc=(\tilde P,\tilde\psi,\tilde\phi)$ and $\tilde\fc'=(\tilde P,\tilde\psi',\tilde\phi')$
are configurations in $\Hur(\bdiamolr,\bdel)_{\bdel\bdiamolr}$ (in $\Hur(\diamo,\del)_{\bdel\bdiamolr}$).
We can choose an lr-based admissible generating set $\tilde\gen_1,\dots,\tilde\gen_{\tilde k}$ of $\fG(\tilde P)$ (see \cite[Definition 6.10]{Bianchi:Hur2}) with the following properties:
\begin{itemize}
 \item two generators, denoted $\tilde\genleft$ and $\tilde\genright$, are represented by loops spinning clockwise around $\zdiamleft$ and $\zdiamright$ respectively;
 \item the other generators are represented by loops contained in the strip $[0,1]\times\R$;
 in particular, we assume that $\tilde\gen_1,\dots,\tilde\gen_{\tilde l}$ correspond
 to points in $P\setminus\del\diamo$.
\end{itemize}
Since $\tilde\fc$ and $\tilde\fc'$ agree on $[0,1]\times\R$, we have that $\tilde\psi$ and $\tilde\psi'$ agree on the admissible generators
$\tilde\gen_1,\dots,\tilde\gen_{\tilde l}$, whereas
$\tilde\phi$ and $\tilde\phi'$ agree on all admissible generators
$\tilde\gen_1,\dots,\tilde\gen_{\tilde k}$ except, possibly, the two generators $\tilde\genleft$ and $\tilde\genright$.
It follows that $\tilde\fc$ and $\tilde\fc'$ have the same image in the quotient
$\Hur(\bdiamolr,\bdel)_{G,G^{op}}$ (respectively, $\Hur(\diamo,\del)_{G,G^{op}}$),
i.e. $\check\sigma(\uw;\ut,\ufc)=\check\sigma(\uw';\ut',\ufc')$. This concludes the proof that $\check\sigma$ descends to a map $\sigma$ defined on $B\mHurm$ (on $B\bHurm$).

For the second statement of the proposition, let $(\uw;\ut,\ufc)$ be as in Notation \ref{nota:uwutufc} and assume $(t_j,\fc_j)=(0,(\emptyset,\one,\one))$ for a fixed $0\leq j\leq p$;
let $(\ut',\ufc')=d_j(\ut,\ufc)$, so that vice versa $(\ut,\ufc)=s_j(\ut',\ufc')$, and let $\uw'$ be obtained from $\uw$ by replacing the consecutive entries $w_j$ and $w_{j+1}$ with their sum $w_j+w_{j+1}$;
present $\uw'$ as $(w'_0,\dots,w'_{p-1})$, $\ut'$ as $(t'_1,\dots,t'_{p-1})$ and $\ufc'$ as $(\fc'_1,\dots,\fc'_{p-1})$.
We want to prove that $\check\sigma(\uw;\ut,\ufc)=\check\sigma(\uw';\ut',\ufc')$: this
implies that $\check\sigma$ descends to a map $\sigma$ defined on $B\mHurm$ (on $B\bHurm$).

For this it suffices to note that $\hmu(\uw;\ut,\ufc)=\hmu(\uw';\ut',\ufc')$ and also the barycentres
$b,b^-,b^+$ are equal when computed with respect to $(\uw;\ut,\ufc)$ or $(\uw';\ut',\ufc')$. The formula for $\check\sigma$ only depends on these parameters.

\bibliography{Bibliography3.bib}
\bibliographystyle{alpha}

\end{document}